\documentclass[a4paper,11pt,openright,twoside]{article}
\usepackage{geometry}
\usepackage{etex}%
\usepackage[utf8]{inputenc}
\usepackage[english]{babel}
\usepackage[T1]{fontenc}
\usepackage{acronym}%
\usepackage{epstopdf}%
\usepackage{epigraph}%
\usepackage{indentfirst}%
\usepackage{amsmath,amssymb, amsthm, amsfonts}
\usepackage{braket}%
\usepackage{empheq}%
\usepackage{hyperref}
\usepackage{mathtools}
\usepackage{dcpic}%
\usepackage[all,cmtip,2cell]{xy} 
\UseAllTwocells
\input{diagxy}
\usepackage{wrapfig}%
\usepackage{subfig,caption}%
\usepackage{stmaryrd}%
\usepackage{textcomp}%
\usepackage{mathrsfs}%
\usepackage{lmodern}
\usepackage{tikz}
\usetikzlibrary{graphs,quotes,matrix}
\usepackage[all]{xy}

\newcommand{\plus}[1]{\mathop{\amalg}\limits_{#1}}
\newcommand\nbd\nobreakdash
\newcommand{\Cat}{{\mathcal{C}\mspace{-2.mu}\it{at}}}
\newcommand{\nCat}[1]{{#1}\hbox{\protect\nbd-}\kern1pt\Cat}	

\newcommand{\cd}[2][]{\vcenter{\hbox{\xymatrix#1{#2}}}}

\makeatletter
\def\matrixobject@{%
	\edef \next@{={\DirectionfromtheDirection@ }}%
	\expandafter \toks@ \next@ \plainxy@
	\let\xy@@ix@=\xyq@@toksix@
	\xyFN@ \OBJECT@}
\let\xy@entry@@norm=\entry@@norm
\def\entry@@norm@patched{%
	\let\object@=\matrixobject@
	\xy@entry@@norm }
\AtBeginDocument{\let\entry@@norm\entry@@norm@patched}
\makeatother

\newcommand{\Twocong}[2][0.5]{\ar@{}[#2] \save ?(#1)*{\cong}\restore}
\newcommand{\Twoeq}[2][0.5]{\ar@{}[#2] \save ?(#1)*{=}\restore}
\newcommand{\Ltwocell}[3][0.5]{\ar@{}[#2] \ar@{=>}?(#1)+/r 0.2cm/;?(#1)+/l 0.2cm/^{#3}}
\newcommand{\Rtwocell}[3][0.5]{\ar@{}[#2] \ar@{=>}?(#1)+/l 0.2cm/;?(#1)+/r 0.2cm/^{#3}}
\newcommand{\Utwocell}[3][0.5]{\ar@{}[#2] \ar@{=>}?(#1)+/d  0.2cm/;?(#1)+/u 0.2cm/_{#3}}
\newcommand{\Dtwocell}[3][0.5]{\ar@{}[#2] \ar@{=>}?(#1)+/u  0.2cm/;?(#1)+/d 0.2cm/^{#3}}
\newcommand{\Ultwocell}[3][0.5]{\ar@{}[#2] \ar@{=>}?(#1)+/dr  0.2cm/;?(#1)+/ul 0.2cm/^{#3}}
\newcommand{\Urtwocell}[3][0.5]{\ar@{}[#2] \ar@{=>}?(#1)+/dl  0.2cm/;?(#1)+/ur 0.2cm/^{#3}}
\newcommand{\Dltwocell}[3][0.5]{\ar@{}[#2] \ar@{=>}?(#1)+/ur  0.2cm/;?(#1)+/dl 0.2cm/^{#3}}
\newcommand{\Drtwocell}[3][0.5]{\ar@{}[#2] \ar@{=>}?(#1)+/ul  0.2cm/;?(#1)+/dr 0.2cm/^{#3}}

\newcommand{\Ulthreecell}[3][0.5]{\ar@{}[#2] \ar@3?(#1)+/dr  0.2cm/;?(#1)+/ul 0.2cm/_{#3}}

\newcommand{\gpd}{{\mathcal{G}\mspace{-2.mu}\it{pd}}}
\newcommand{\ngpd}[1]{{#1}\hbox{\protect\nbd-}\kern1pt\gpd}
\newcommand{\wgpd}{ \ngpd{\infty}}	

\newcommand{\D}{\mathcal{D}}

\newcommand{\M}{\mathbf{M}}

\newcommand{\R}{\mathcal{R}}

\newcommand{\G}{\mathbb{G}}
\newcommand{\p}{\mathbb{P}}

\newcommand{\C}{\mathscr{C}}
\newcommand{\A}{\mathscr{A}}

\newcommand{\F}{\mathscr{F}}
\newcommand{\cyl}{\mathbf{Cyl}}
\newcommand{\Cyl}{\underline{\mathbf{Cyl}}}
\newcommand{\ulthreecell}[3][0.5]{\ar@{}[#2] \ar@3?(#1)+/dr  0.2cm/;?(#1)+/ul 0.2cm/_{#3}}

\renewcommand{\epsilon}{\varepsilon}
\renewcommand{\theta}{\vartheta}
\renewcommand{\rho}{\varrho}
\renewcommand{\phi}{\varphi}

\newcounter{ctr} \numberwithin{ctr}{section}

\theoremstyle{definition}

\theoremstyle{definition}
\newtheorem{defi}[ctr]{Definition}
\theoremstyle{definition}

\theoremstyle{definition}
\newtheorem{rmk}[ctr]{Remark}
\theoremstyle{definition}
\newtheorem{ex}[ctr]{Example}
\theoremstyle{plain}
\newtheorem{thm}[ctr]{Theorem}
\theoremstyle{plain}

\theoremstyle{plain}
\newtheorem{prop}[ctr]{Proposition}
\theoremstyle{plain}
\newtheorem{lemma}[ctr]{Lemma}
\theoremstyle{plain}
\newtheorem{cor}[ctr]{Corollary}
\theoremstyle{definition}
\newtheorem{conj}[ctr]{Conjecture}

\DeclareMathOperator{\ob}{Ob}

\DeclareMathOperator{\colim}{colim}

\DeclareMathOperator{\f}{Fib}

\DeclareMathOperator{\height}{ht}
\makeatletter

\begin{document}

	\title{Towards a globular path object for weak $\infty$-groupoids}
	\author{Edoardo Lanari\thanks{Supported by Macquarie University iMQRes PhD scholarship}}

	\maketitle
	\begin{abstract}
 
 	The goal of this paper is to address the problem of building a path object for the category of Grothendieck (weak) $\infty$-groupoids. This is the missing piece for a proof of Grothendieck's homotopy hypothesis. We show how to endow the putative underlying globular set with a system of composition, a system of identities and a system of inverses, together with an approximation of the interpretation of any map for a theory of $\infty$-categories.
 	Finally, we introduce a coglobular $\infty$-groupoid representing modifications of $\infty$-groupoids, and prove  some basic properties it satisfies, that will be exploited to interpret all $2$-dimensional categorical operations on cells of the path object $\p X$ of a given $\infty$-groupoid $X$.
\end{abstract}
\section{Introduction}
Alexander Grothendieck first introduced his definition of (a model of) weak $\infty$-groupoids in the famous 1983 letter to Quillen, see \cite{Gr}. He wanted to have a completely algebraic model of these highly structured gadgets, encoding every possible composition operation, inverses in all codimensions and coherence constraints needed for a sensible definition of weak $\infty$-groupoid. Moreover, he conjectured that these algebraic structures modeled all homotopy types, the so-called ``homotopy hypothesis''.

Roughly speaking, Grothendieck $\infty$-groupoids are strict models of a new sort of algebraic theory (called ``globular theories''), whose arities are given by the so called ``globular sums'', i.e. suitable pasting of $n$-disks for every $n\geq 0$. Not every globular theory is a suitable theory for $\infty$-groupoids, but rather only a specific subclass of these, called ``coherators'', which should be thought as suitable cofibrant replacements of the globular theory $\tilde{\Theta}$ for strict $\infty$-groupoids (see Section 3 of \cite{MA}). These are contractible and cellular theories in a precise sense, which ensures (respectively) that all the sensible operations on cells of an $\infty$-groupoid exist and that the relations between these existing operations only hold up to higher cells (rather than being identities as in the strict case).

In 2010 George Maltsiniotis made this notion more precise in his paper \cite{MA}, where he defines a category of Grothendieck $\infty$-groupoids denoted by $\wgpd$. In fact, he defines many categories of such, parametrized by the choice of a coherator, but in the end these should all give rise to equivalent models. He also constructs an adjunction of the form:
\begin{equation}
\label{adj top grpd}
\xymatrixcolsep{1pc}
\vcenter{\hbox{\xymatrix{
			**[l]\wgpd \xtwocell[r]{}_{\Pi_{\infty}}^{\vert \bullet \vert}{'\perp}& **[r] \mathbf{Top}}}}
\end{equation}
The homotopy hypothesis states that the adjunction in \eqref{adj top grpd} induces an equivalence of homotopy theories (i.e. the underlying $(\infty,1)$-categories).

In this paper, we address the problem of proving this conjecture, and we manage to make some progress by taking some necessary steps towards its complete solution. As we clarify later on in this introduction, the problem reduces to show that adjoining an $n$-cell along its source does not change the homotopy type of an $\infty$-groupoid, and the idea is to prove that by constructing a path object for the fibration category $\wgpd$. We define its underlying globular set and endow it with a non-trivial algebraic structure, namely a system of composition, a system of identities and a system of inverses. We also succeed in constructing a non-functorial interpretation of all the required structure and manage to fix the non-functoriality in low dimension by introducing modifications (a higher dimensional analogue of natural transformations). What is now left to do is to prove a conjecture that compares M.Batanin's definition of $\infty$-groupoids and Grothendieck's one, together with generalizing this process of ``correcting'' the construction to make it functorial in all higher dimensions.

In a follow-up paper we adapt this strategy to construct a path object for Grothendieck 3-groupoids, defined using an appropriate notion of 3-truncated coherator. For this path object, we prove the existence of the desired factorization of the diagonal map, which implies the existence of the semi-model structure provided the abovementioned conjecture holds true in this dimension. This also implies that Grothendieck $3$-groupoids model homotopy $3$-types.
%

In order to give meaning to the statement that constitutes the homotopy hypothesis, one has to define the homotopy theory of $\infty$-groupoids. Already in \cite{Gr} there is a very natural and simple functorial definition of homotopy groups of a given $\infty$-groupoid, which induces, just like in the case of topological spaces or Kan complexes, a notion of weak equivalence. In details, a map $f\colon X \rightarrow Y$ in $\wgpd$ is a weak equivalence if it induces a bijection on the set of path components $\pi_0(X)\cong \pi_0(Y)$, as well as isomorphisms of homotopy groups:
\[
\bfig
\morphism(0,50)|a|/@{>}@<0pt>/<800,0>[\pi_{n}(X,x)`\pi_{n}\left(Y,f(x) \right);\pi_{n}(f)]
\efig
\]
for every object $x$ of $X$.

This choice of weak equivalences endows the category $\wgpd$ with a relative category structure, and we can also consider $\mathbf{Top}$ as a relative category, with the usual notion of weak equivalences of topological spaces. It turns out that the functor $\Pi_{\infty}$ preserves weak equivalences, and thus a precise statement of the homotopy hypothesis is that the functor $\Pi_{\infty}$ is a weak equivalence of relative categories (e.g. it induces a weak equivalence at the level of the simplicial localization of the relative categories involved). Note that its left adjoint may not be a morphism of relative categories and, therefore, the adjunction in \eqref{adj top grpd} may not be one of relative categories. Recently, Simon Henry proved in \cite{Hen} that it suffices to construct a semi-model structure on $\infty$-groupoids to validate the homotopy hypothesis. This semi-model structure should have, as generating cofibration (resp. trivial cofibrations), the ``boundary inclusions'' $\partial \colon S^{n-1}\rightarrow D_n$ (resp. the ``source maps'' $\sigma_n \colon D_n \rightarrow D_{n+1}$) for every $n\geq 0$, which clearly resemble the ones for the Serre model structure on topological spaces.

The only difficult part in proving the existence of this semi-model structure is showing that given a cocartesian square of the form
\[
\bfig 
\morphism(0,0)|a|/@{>}@<0pt>/<600,0>[D_n`X;]
\morphism(0,0)|l|/@{>}@<0pt>/<0,-400>[D_n`D_{n+1};\sigma_n]
\morphism(600,0)|r|/@{>}@<0pt>/<0,-400>[X`X^+;i]
\morphism(0,-400)|l|/@{>}@<0pt>/<600,0>[D_{n+1}`X^+;]
\efig 
\] the map $i\colon X \rightarrow X^+$ is a weak equivalence of $\infty$-groupoids. In fact, it is straightforward to show that to prove the ``pushout lemma'' in this situation it is enough to construct a path object $\p X$ for every cofibrant object $X$ (even finitely cellular would do), i.e. a factorization of the diagonal map $\Delta \colon X \rightarrow X \times X$ into a weak equivalence followed by a fibration as displayed below:
\[
\bfig 
\morphism(0,0)|a|/@{>}@<0pt>/<500,0>[X`\p X;\simeq]
\morphism(500,0)|a|/@{>>}@<0pt>/<500,0>[\p X` X\times X;p]
\efig
\]
In this paper we therefore address the problem of defining such functorial path object. The natural way of obtaining it is to define the path object as a right adjoint functor, or, equivalently, as a functor co-represented by a co-$\infty$-groupoid object in the category $\wgpd$. More precisely, one needs to construct a globular functor $\cyl\colon\mathfrak{C} \rightarrow \wgpd$ which represent cylinders on globular sums, where $\mathfrak{C}$ is the chosen coherator whose category of models we denoted by $\wgpd$, and then set $\p X = \wgpd(\cyl(\bullet),X)$. We do not achieve this result in its entirety,  but significant and necessary steps are made towards a complete construction of it. In particular, we define the underlying globular set of this path object in Definition \ref{path space defi}, i.e. the set of its $n$-cells for every $n\geq 0$, and we endow this with a system of composition, of identities and of inverses in Theorem \ref{algebraic structure on PX}. Furthermore, we also define an interpretation of each homogeneous operation in a theory for $\infty$-categories in a non-functorial fashion, in the sense that it does not respect substitution of operations on the nose. This is a non-trivial combinatorial problem which we solve by making use of trees to exhibit the cylinder on a globular sum as the colimit of a zig-zag diagram of globular sums. This allows us to define, given any operation in a theory for $\infty$-categories $\rho\colon D_n \rightarrow A$ a vertical stack of cylinders, which we then compose using a vertical composition operation that we introduce in full generality in Subsection 9.2, thus getting a map $\hat{\rho}\colon\cyl(D_n) \rightarrow \cyl(A)$. The idea is to then use modifications, i.e. ``homotopies of cylinders'', to adjust the source and target of this approximation of $\cyl(\rho)$ and make it functorial. At this point, we define the globular set of modifications in an $\infty$-groupoid $X$ and use it to interpret all operations of dimension less or equal than 2. What is left to do is extend this process to higher dimensions, in order to interpret all the necessary coherences to endow the path object with the structure of an $\infty$-category, and then prove that a certain conjecture we state below holds true, so as to obtain a proof of the existence of the semi-model structure and thus a proof of the homotopy hypothesis.

Sections 2 serves as a recap of (some of) the definitions and constructions related to $\infty$-groupoids in the literature. The following section is a recap on direct categories and lifting of factorization systems on a category $\C$ to a functor category $\C^{\D}$ for a direct category $\D$, together with a proof of contractibility of globular sums. 

In Section 4 we introduce the suspension-loop space adjunction. This is the first piece of original content, and it will be exploited extensively throughout this paper. In section 5 we define the coglobular object of cylinders on $n$-globes. This is enough to get a functor $\p \colon \wgpd \rightarrow [\G^{op},\mathbf{Set}]$, which sends an $\infty$-groupoid $X$ to the underlying globular set $\p X$ of the putative path object on it.

Section 6 is devoted to the study of this globular set $\p X$ we have just defined, and we show in Theorem \ref{algebraic structure on PX} how to endow it with some non-trivial algebraic structure, such as a system of composition, a system of identities and a system of inverses.

Beyond the algebraic structure just mentioned, an $\infty$-groupoid also involves ``coherence'' data. To equip $\p X$ with choices of these data, one has to define $\cyl(\rho)\colon \cyl(D_n)\rightarrow \cyl(A)$ for any given operation $\rho$ in a chosen theory for $\infty$-categories. For this purpose, in Section 7 we adopt the formalism of trees to give a more explicit description of the cylinder on a globular sum as the colimit of a suitable zig-zag diagram of globular sums. In the following section, we define vertical composition of a stack of cylinders, thus leading to a first ``naive'' interpretation of the operations. However, this interpretation does not respect the source and target of an operation. To correct this issue, we introduce degenerate cylinders in Section 9 and generalize the previous result to construct in Definition \ref{rho hat defi} a non-functorial (in the sense we explained above) approximation of all the structural maps that we would like our path object to be endowed with. Despite its non-functoriality, this interpretation satisfies some nice and necessary properties the correct one should satisfy, but the strategy we used only allows us to interpret the ``categorical'' operations. This means that we can define a map $\hat{\rho}\colon\cyl(D_n)\rightarrow \cyl(A)$ for every $\rho \colon D_n \rightarrow A$ in a coherator for $\infty$-categories $\mathfrak{D}$.
However, we conjecture that a coherator for $\infty$-categories that can be endowed with a system of inverses can be promoted to one for $\infty$-groupoids, and we have already shown how to endow the putative path object with inverses in all codimensions.

Finally, in Section 10 we introduce modifications of $\infty$-groupoids, and after having proven some of their properties we then proceed to show in the last section of the present paper how to use them to correct the non-functorial interpretation to get a valid interpretation of all operations of dimension less or equal to $2$.
\section{Background}
\subsection{Globular theories and coherators}
The preliminary concepts and definitions needed for understanding this paper can be found in \cite{AR1} and \cite{MA}. Nevertheless, we present here a concise summary of these, to make this work as self-contained as possible.

We start by defining the category of globes, which will serve as the starting point for everything that follows.
\begin{defi}
Let $\G$ be the category obtained as the quotient of the free category on the graph
\[
\bfig
\morphism(0,0)|a|/@{>}@<2pt>/<300,0>[0`1;\sigma_0]
\morphism(0,0)|b|/@{>}@<-2pt>/<300,0>[0`1;\tau_0]
\morphism(300,0)|a|/@{>}@<2pt>/<300,0>[1`\ldots;\sigma_1]
\morphism(300,0)|b|/@{>}@<-2pt>/<300,0>[1`\ldots;\tau_1]
\morphism(720,0)|a|/@{>}@<2pt>/<400,0>[n`n+1;\sigma_n]
\morphism(720,0)|b|/@{>}@<-2pt>/<400,0>[n`n+1;\tau_n]
\morphism(1120,0)|a|/@{>}@<2pt>/<400,0>[n+1`\ldots;\sigma_{n+1}]
\morphism(1120,0)|b|/@{>}@<-2pt>/<400,0>[n+1`\ldots;\tau_{n+1}]
\efig
\]
by the set of relations $\sigma_k \circ \sigma_{k-1}=\tau_k \circ \sigma_{k-1}$, $\sigma_k \circ \tau_{k-1}=\tau_k \circ \tau_{k-1}$ for $k\geq 1$.

Given integers $j>i$, define $\sigma^j_i=\sigma_{j-1}\circ \sigma^{j-1}_i$, where $\sigma^{i+1}_i=\sigma_i$. The maps $\tau^j_i$ are defined similarly, with the appropriate changes.
\end{defi}
The category of globular sets is the presheaf category $[\G^{op},\mathbf{Set}]$.
\begin{defi}
A map $f\colon X \rightarrow Y$ of globular sets is said to be $n$-bijective if $f_k\colon X_k \rightarrow Y_k$ is a bijection of sets for every $k\leq n$, and it is $n$-fully faithful if the following square is cartesian for all $i\geq n$:
\[
\bfig 
\morphism(0,0)|a|/@{>}@<0pt>/<800,0>[X_{i+1}`Y_{i+1};f_{i+1}]
\morphism(0,0)|l|/@{>}@<0pt>/<0,-400>[X_{i+1}`X_i\times X_i;(s,t)]
\morphism(800,0)|r|/@{>}@<0pt>/<0,-400>[Y_{i+1}`Y_i\times Y_i;(s,t)]
\morphism(0,-400)|l|/@{>}@<0pt>/<800,0>[X_{i}\times X_i`Y_{i}\times Y_i;f_i\times f_i]
\efig 
\]
We denote the class of $n$-bijective morphisms by $\mathbf{bij_n}$, and that of $n$-fully faithful ones by $\mathbf{ff_n}$.
\end{defi}
The following result holds true, and its proof is left as a simple exercise
\begin{prop}
	\label{fact syst glob set}
The pair $(\mathbf{bij_n},\mathbf{ff_n})$ is an orthogonal factorization system on the category of globular sets $[\G^{op},\mathbf{Set}]$.
\end{prop}
The globes themselves are not enough to capture the basic shapes, or arities, of a theory of $\infty$-groupoids or $\infty$-categories.
We instead have to consider globular sums, which are special kinds of pastings of globes.
\begin{defi}
A table of dimensions is a sequence of integers of the form 
\[\begin{pmatrix}
i_1 &&i_2 & \ldots&i_{m-1} & &i_m\\
& i'_1 & &\ldots&& i'_{m-1}
\end{pmatrix}\]
satisfying the following inequalities: $i'_k<i_k$ and $i'_k<i_{k+1}$ for every $0\leq k\leq m-1$.\\
Given a category $\mathcal{C}$ and a functor $F\colon \G \rightarrow \mathcal{C}$, a table of dimensions as above induces a diagram of the form
\[
\bfig
\morphism(0,0)|l|/@{>}@<0pt>/<-400,400>[F(i'_1)`F(i_1);F(\sigma_{i'_1}^{i_1})]
\morphism(0,0)|a|/@{>}@<0pt>/<400,400>[F(i'_1)`F(i_2);F(\tau_{i'_1}^{i_2})]
\morphism(800,0)|l|/@{>}@<0pt>/<-400,400>[F(i'_2)`F(i_2);F(\sigma_{i'_2}^{i_2})]
\morphism(800,0)|a|/@{>}@<0pt>/<400,400>[F(i'_2)`F(i_3);F(\tau_{i'_2}^{i_3})]
\morphism(1300,150)|r|/@{}@<0pt>/<100,0>[ `\ldots;]
\morphism(2000,0)|l|/@{>}@<0pt>/<-400,400>[F(i'_{m-1})`F(i_{m-1});F(\sigma_{i'_{m-1}}^{i_{m-1}})]
\morphism(2000,0)|a|/@{>}@<0pt>/<400,400>[F(i'_{m-1})`F(i_{m});F(\tau_{i'_{m-1}}^{i_{m}})]
\efig 
\]
A globular sum of type $F$ (or simply globular sum) is the colimit in $\mathcal{C}$ (if it exists) of such a diagram.

We also define the dimension of this globular sum to be $\dim(A)=\max \{i_k\}_{k \in \{1, \ldots, \ m\}}$. Given a globular sum $A$, we denote with $\iota_k^A$ the colimit inclusion $F({i_k})\rightarrow A$, dropping subscripts when there is no risk of confusion.
\end{defi}
We denote by $\Theta_0$ the full subcategory of globular sets spanned by the globular sums of type $y\colon \G \rightarrow [\G^{op},\mathbf{Set}]$, where $y$ is the Yoneda embedding. Moreover, we denote $y(i)$ by $D_i$ and  the globular sum corresponding to the table of dimensions 
\[\begin{pmatrix}
1 &&1 & \ldots&1 & &1\\
& 0 & &\ldots&& 0
\end{pmatrix}\] by $D_1^{\otimes k}$, where the integer $1$ appears exactly $k$ times.
\begin{defi}
 A globular theory is a pair $(\mathfrak{C},F)$, where $F\colon \Theta_0 \rightarrow \mathfrak{C}$ is a bijective on objects and globular sum preserving functor.
 
 We denote by $\mathbf{GlTh}$ the category of globular theories and globular sums preserving functors. More precisely, a morphism $H\colon (\mathfrak{C},F) \rightarrow (\mathfrak{D},G)$ in $\mathbf{GlTh}$ is a functor $H\colon \mathfrak{C} \rightarrow \mathfrak{D}$ such that $H\circ F=G$.
\end{defi}
\begin{defi}
Given a globular theory $(\mathfrak{C},F)$, we define the category of its models, denoted $\mathbf{Mod}(\mathfrak{C})$, to be the category of globular product preserving functors $G\colon \mathfrak{C}^{op} \rightarrow \mathbf{Set}$. Clearly, the Yoneda embedding $y\colon \mathfrak{C} \rightarrow [\mathfrak{C}^{op},\mathbf{Set}]$ factors through $\mathbf{Mod}(\mathfrak{C})$, and it will still be denoted by $y\colon \mathfrak{C} \rightarrow \mathbf{Mod}(\mathfrak{C})$.
\end{defi}
Again, we denote the image of $i$ under $y$ by $D_i$.
\begin{prop}
	\label{UP of models}
Given a globular theory $\mathfrak{C}$, the category of models $\mathbf{Mod}(\mathfrak{C})$ enjoys the following universal property: given a cocomplete category $\D$, a cocontinuous functor $F\colon \mathbf{Mod}(\mathfrak{C}) \rightarrow \D$ is determined up to a unique isomorphism by a functor $\overline{F}\colon \mathfrak{C} \rightarrow \D$, corresponding to its restriction along the Yoneda embedding, that preserves globular sums.
\end{prop}
The $\infty$-groupoids we are going to consider are presented as models of a certain class of globular theories, namely the cellular and contractible ones.
\begin{defi}
	\label{contr glob th}
Two maps $f,g\colon D_n\rightarrow A$ in a globular theory are said to be parallel if either $n=0$ or $f\circ \epsilon= g\circ \epsilon$ for $\epsilon=\sigma,\tau$.
A pair of parallel maps $(f,g)$ is said to be admissible if $\dim(A) \leq n+1$.
A globular theory $(\mathfrak{C},F)$ is called contractible if for every admissible pair of maps $f,g\colon D_n\rightarrow A$ there exists an extension $h\colon D_{n+1}\rightarrow A$ rendering the following diagram serially commutative
\[
\bfig 
\morphism(0,0)|a|/@{>}@<3pt>/<500,0>[D_n`A;f]
\morphism(0,0)|b|/@{>}@<-3pt>/<500,0>[D_n`A;g]

\morphism(0,0)|r|/@{>}@<3pt>/<0,-400>[D_n`D_{n+1};\tau_n]
\morphism(0,0)|l|/@{>}@<-3pt>/<0,-400>[D_n`D_{n+1};\sigma_n]

\morphism(0,-400)|r|/@{>}@<0pt>/<500,400>[D_{n+1}`A;h]
\efig 
\]
\end{defi}
Contractibility ensures the existence of all the operations that ought to be part of the structure of an $\infty$-groupoid. However, it does not guarantee weakness of the models, and indeed there exists a contractible globular theory (often denoted by $\tilde{\Theta}$) whose models are strict $\infty$-groupoids.

To remedy this, we introduce the concept of cellularity, or freeness, to restrict the class of globular theories we consider. This notion is based on a construction explained in paragraph 4.1.3 of \cite{AR1},  which we record in the following proposition.
\begin{prop}
	\label{univ prop of glob th}
Given a globular theory $\mathfrak{C}$ and set $X$ of admissible pairs in it, there exists another globular theory $\mathfrak{C}[X]$ equipped with a morphism $\phi\colon \mathfrak{C} \rightarrow \mathfrak{C}[X]$ with the following universal property: given a globular theory $\mathfrak{D}$, a morphism $\mathfrak{C}[X] \rightarrow \mathfrak{D}$ is determined up to a unique isomorphism by its precomposition $F$ with $\phi$ and a choice of an extension as in Definition \ref{contr glob th} for the image under $F$ of each admissible pair in $X$.
\end{prop}
In words, $\mathfrak{C}[X]$ is obtained from $\mathfrak{C}$ by universally adding a lift for each pair in $X$.
\begin{defi}
	\label{cell glob th}
A globular theory $(\mathfrak{C},F)$ is called cellular if there exists a functor $\mathfrak{C}_{\bullet} \colon \omega \rightarrow \mathbf{GlTh}$, where $\omega$ is the first uncountable ordinal, such that:
\begin{enumerate}
\item $\mathfrak{C}_0 \cong \Theta_0$;
\item for every $n \geq 0$, there exists a family $X$ of admissible pairs of arrows in $\mathfrak{C}_n$ (as in Definition \ref{contr glob th}) such that $\mathfrak{C}_{n+1}\cong \mathfrak{C}_n[X]$;
\item $\colim_{n \in \omega}\mathfrak{C}_{n}\cong \mathfrak{C}$.
\end{enumerate}
A contractible and cellular globular theory is called a coherator.
\end{defi}
Equivalently, one could replace $\omega$ with an arbitrary ordinal and ask for $\mathfrak{C}_{\bullet}$ to be cocontinuous, adjoining just one single pair of maps at each successor ordinal stage.
\begin{defi}
Given a coherator $\mathfrak{C}$, the category of $\infty$-groupoids of type $\mathfrak{C}$ is the category $\mathbf{Mod}(\mathfrak{C})$ of models of $\mathfrak{C}$.
\end{defi}
We now briefly give a sample of why this should be considered a definition of a model of $\infty$-groupoids.

Firstly, the restriction of an $\infty$-groupoid $X\colon\mathfrak{C}^{op} \rightarrow \mathbf{Set}$ to ${\Theta_0}^{op}$ gives an object of $\mathbf{Mod}(\Theta_0)\simeq [\G^{op},\mathbf{Set}]$, which we call the underlying globular set of $X$. The set $X_n$ represents the set of $n$-cells of $X$.

Turning to the algebraic structure acting on these sets of cells, Section 3 of \cite{AR2} shows how to endow the underlying globular set of an $\infty$-groupoid with all the sensible operations it ought to have to deserve to be called such.

For example, we can build operations that represent binary composition of a pair of $1$-cells, codimension-$1$ inverses for $2$-cells and an associativity constraint for composition of $1$-cells by solving, respectively, the following extension problems:
\[
\bfig 
\morphism(0,0)|a|/@{>}@<3pt>/<600,0>[D_0`D_1\plus{ D_0}D_1;i_0\circ \sigma_0]
\morphism(0,0)|b|/@{>}@<-3pt>/<600,0>[D_0`D_1\plus{ D_0}D_1;i_1\circ \tau_0]

\morphism(0,0)|r|/@{>}@<3pt>/<0,-400>[D_0`D_{1};\tau_0]
\morphism(0,0)|l|/@{>}@<-3pt>/<0,-400>[D_0`D_{1};\sigma_0]

\morphism(0,-400)|r|/@{>}@<0pt>/<600,400>[D_{1}`D_1\plus{ D_0}D_1;\nabla^1_0]

\morphism(1000,0)|a|/@{>}@<3pt>/<500,0>[D_1`D_2;\tau_1]
\morphism(1000,0)|b|/@{>}@<-3pt>/<500,0>[D_1`D_2;\sigma_1]

\morphism(1000,0)|r|/@{>}@<3pt>/<0,-400>[D_1`D_{2};\tau_1]
\morphism(1000,0)|l|/@{>}@<-3pt>/<0,-400>[D_1`D_{2};\sigma_1]

\morphism(1000,-400)|r|/@{>}@<0pt>/<500,400>[D_{2}`D_2;\omega^2_1]

\morphism(2000,0)|a|/@{>}@<3pt>/<1400,0>[D_1`D_1 \plus{ D_0} D_1 \plus{ D_0} D_1;(\nabla^1_0\plus{ D_0} 1_{D_1})\circ \nabla^1_0]
\morphism(2000,0)|b|/@{>}@<-3pt>/<1400,0>[D_1`D_1 \plus{ D_0} D_1 \plus{ D_0} D_1;(1_{D_1} \plus{ D_0} \nabla^1_0 )\circ \nabla^1_0]

\morphism(2000,0)|r|/@{>}@<3pt>/<0,-400>[D_1`D_{2};\tau_1]
\morphism(2000,0)|l|/@{>}@<-3pt>/<0,-400>[D_1`D_{2};\sigma_1]

\morphism(2000,-400)|r|/@{>}@<0pt>/<1400,400>[D_{2}`D_1 \plus{ D_0} D_1 \plus{ D_0} D_1;\alpha]
\efig 
\]
In a similar fashion one can build every sensible operation a weak $\infty$-groupoid ought to be endowed with.

Whenever a choice of such operations is understood, at the level of models (i.e. $\infty$-groupoids) we denote with the familiar juxtaposition of cells the (unbiased) composition of them, and with the exponential notation $A^{-1}$ we denote the codimension-$1$ inverse of an $n$-cell $A$.

We will need to choose some operations once and for all, so we record here their definition.
Choose an operation $\nabla^1_0\colon D_1 \rightarrow D_1 \amalg_{D_0} D_1$ as above, and define $w=\nabla^1_0$.
Next, pick operations $D_2 \rightarrow D_2 \amalg_{D_0} D_1$ and $D_2 \rightarrow D_1 \amalg_{D_0} D_2$ whose source and target are given, respectively by $\left((\sigma\amalg_{ D_0}1 )\circ w,(\tau \amalg_{ D_0}1 )\circ w\right)$ and $\left((1\amalg_{ D_0} \sigma) \circ w,(1\amalg_{ D_0} \tau) \circ w\right)$. Proceeding in this way we get specified whiskering maps
\begin{equation}
\label{w's maps}
_{n}w\colon D_n \rightarrow D_n \plus{D_0} D_1$$ 
$$w_n\colon D_n \rightarrow D_1 \plus{D_0} D_n
\end{equation}
We will often avoid writing down all the subscripts, when they are clear from the context.
\begin{defi}
	\label{whiskering w}
	Given a globular sum $A$, whose table of dimensions is
	\[\begin{pmatrix}
	i_1 &&i_2 & \ldots&i_{n-1} & &i_n\\
	& i'_1 & &\ldots&& i'_{n-1}
	\end{pmatrix}\]
	we define a map $_A w\colon A\rightarrow A \amalg_{ D_0} D_1$ by \[w_{i_1 +1 }\plus{w_{i'_1 +1} } \ldots \plus{w_{i'_{n-1} +1 }} w_{i_n +1}\colon D_{i_1 +1 }\plus{D_{i'_1 +1} } \ldots \plus{D_{i'_{n-1} +1 }} D_{i_n +1} \rightarrow (D_{i_1 +1 }\plus{D_{i'_1 +1} } \ldots \plus{D_{i'_{n-1} +1 }} D_{i_n +1})\plus{D_0}D_1  \] noting  that the target is isomorphic to \[(D_{i_1 +1 }\plus{D_0}D_1)\plus{D_{i'_1 +1} \plus{D_0} D_1 } \ldots \plus{D_{i'_{n-1} +1 } \plus{D_0} D_1} (D_{i_n +1} \plus{D_0} D_1)\]
	In a completely analogous manner we define a map $w_A\colon A \rightarrow D_1\amalg_{D_0}A$.
\end{defi}Consider the forgetful functor \[\mathbf{U}\colon \wgpd \rightarrow [\G^{op},\mathbf{Set}]\simeq \mathbf{Mod}(\Theta_0)\] induced by the structural map $\Theta_0 \rightarrow \mathfrak{C}$. Given a map of $\infty$-groupoids $f\colon X \rightarrow Y$, we can factor the map $\mathbf{U}(f)$ as $\mathbf{U}(f)=g\circ h$, where $h$ is $n$-bijective and $g$ is $n$-fully faithful thanks to Proposition \ref{fact syst glob set}. It is not hard to see that the target of $h$ can be endowed with the structure of an $\infty$-groupoid so that $g$ and $h$ are maps of such. This fact, thanks to Proposition 2 of \cite{BG}, provides the following result that will be used in this paper.
\begin{prop}
\label{fact of maps of gpds}
The orthogonal factorization system $(\mathbf{bij_n},\mathbf{ff_n})$ on globular sets lifts to one on $\wgpd$ via the forgetful functor $\mathbf{U}\colon \wgpd \rightarrow [\G^{op},\mathbf{Set}]$.
\end{prop}
This means, in particular, that every map in $\wgpd$ admits a unique factorization $f=g\circ h$ where $\mathbf{U}(h)$ is $n$-bijective and $\mathbf{U}(g)$ is $n$-fully faithful, and that $n$-bijective maps are closed under colimits in $\wgpd$.
\begin{ex}
\label{source/targets are n-2 bij}
The maps $\sigma_n,\tau_n\colon D_n \rightarrow D_{n+1}$ are $(n-2)$-bijective. Indeed, since the forgetful functor $\mathbf{U}$ creates the factorization system $(\mathbf{bij_n},\mathbf{ff_n})$ on $\wgpd$ for every $n\geq 0$, its left adjoint $\mathbf{F}\colon [\G^{op},\mathbf{Set}] \rightarrow \wgpd$ preserves the left class. Now it is enough to observe that $\mathbf{F}$ sends source and target maps of globular sets to source and target maps of $\infty$-groupoids, and for the former it is easy to check the statement on $(n-2)$-bijectivity.
\end{ex}
Given a globular sum $A$, whose table of dimensions is \[\begin{pmatrix}
i_1 &&i_2 & \ldots&i_{m-1} & &i_m\\
& i'_1 & &\ldots&& i'_{m-1}
\end{pmatrix}\]
we define its boundary to be the globular sum whose table of dimensions is 
\[\begin{pmatrix}
\bar\imath_1 &&\bar\imath_2 & \ldots&\bar\imath_{m-1} & &\bar\imath_m\\
& i'_1 & &\ldots&& i'_{m-1}
\end{pmatrix}\]
where we set \[\bar\imath_k=\begin{cases}
i_k-1&\text{if} \ i_k= \dim(A)\\
i_k&\text{otherwise}
\end{cases}\]
The maps $\sigma,\tau\colon D_n \rightarrow D_{n+1}$ for $n\geq 0$ induce maps 
\begin{equation}
\label{partial defi}
\partial_{\sigma},\partial_{\tau}\colon \partial A \rightarrow A
\end{equation}
Thanks to what we observed in Example \ref{source/targets are n-2 bij}, we have the following result.
\begin{prop}
	\label{partial are n-2 bij}
	Given a globular sum $A$, with $0<n=\dim(A)$, the maps $\partial_{\sigma},\partial_{\tau}\colon \partial A \rightarrow A$ are $(n$-$2)$-bijective.
\end{prop}
Let us now see how to adapt the main definitions to the case of $\infty$-categories, following \cite{AR1}.
The definition is essentially the same as that of $\infty$-groupoids, except we have to restrict the class of admissible maps.
\begin{defi}
Given a globular theory $(\mathfrak{C},F)$, we say that a map $f$ in $\mathfrak{C}$ is globular if it is in the image of $\Theta_0$ under $F$.

On the other hand, $f$ is called homogeneous if for every factorization $f=g\circ f'$ where $g$ is a globular map, $g$ must be the identity.

$\mathfrak{C}$ is said to be homogeneous if it comes endowed with a globular sum preserving functor $H\colon \mathfrak{C} \rightarrow \Theta$ that detects homogeneous maps, in the sense that a map $f$ in $\mathfrak{C}$ is homogeneous if and only if $H(f)$ is such, where $\Theta$ is the globular theory for strict $\infty$-categories, as defined in \cite{AR1}. If this is the case, then given an homogeneous map $\rho\colon D_n \rightarrow A$ we have $n\geq \dim(A)$, and every map $f$ admits a unique factorization as a homogeneous map followed by a globular one.
\end{defi}

\begin{defi}
Let $(\mathfrak{C},F)$ be a globular theory. A pair of maps $(f,g)$ with $f,g\colon D_n \rightarrow A$ is said to be admissible for a theory of $\infty$-categories (or just admissible, in case there is no risk of confusion with the groupoidal case) if either $n=0$, or both of them are homogeneous maps or else if there exists homogeneous maps $f',g'\colon D_n \rightarrow \partial A$ such that the following diagrams commute
\[
\bfig 
\morphism(0,0)|a|/@{>}@<0pt>/<400,0>[D_n`A;f]

\morphism(0,0)|l|/@{>}@<0pt>/<0,-400>[D_n`\partial A;f']

\morphism(0,-400)|r|/@{>}@<0pt>/<400,400>[\partial A`A;\partial_{\sigma}]

\morphism(1000,0)|a|/@{>}@<0pt>/<400,0>[D_n`A;g]

\morphism(1000,0)|l|/@{>}@<0pt>/<0,-400>[D_n`\partial A;g']

\morphism(1000,-400)|r|/@{>}@<0pt>/<400,400>[\partial A`A;\partial_{\tau}]

\efig 
\]
\end{defi}
The definition of a coherator for $\infty$-categories is totally analogous to that for $\infty$-groupoids, i.e. it is a contractible and cellular globular theory, except the pair of maps that we consider in both cases have to be the admissible ones in the sense of the previous definition.

More precisely, the pairs appearing in Definition \ref{contr glob th} and in point 2 of Definition \ref{cell glob th} must be pairs of admissible maps.
\begin{defi}
An $\infty$-category is a model of a coherator for $\infty$-categories.
\end{defi} 
\section{Direct categories and cofibrations}
\begin{defi}(see also \cite{Ho}, Chapter 5)
	\label{direct category}
	A direct category is a pair $(\C,d)$, where $\C$ is a small category and $d:\ob(\C) \rightarrow \lambda$ is a function into an ordinal $\lambda$ , such that if there is a non-identity morphism $f:a \rightarrow b$ in $\C$, then $d(a)< d(b)$.
	
	Given a cocomplete category $\D$ and a functor $X\colon \C \rightarrow \D$, we define the latching object of $X$ at an object $c\in \C$ to be the object of $\D$ given by
	\[L_c(X)=\colim_{c' \in \C_{< d(c)}\downarrow c}X(c')\]
	This defines a functor $L_c$ from the functor category $[\C,\D]$ to the category $\D$, together with a natural transformation $\epsilon_c\colon L_c \Rightarrow ev_c$, with codomain the functor given by evaluation at $c$.
	We also define the latching map of a natural transformation $\alpha \colon  X\rightarrow Y$ in $\D^{\C}$ at an object $c\in \C$ to be the map of $\D$ 
	\[\hat{L}_c(\alpha)\colon X(c) \plus{L_c(X)} L_c(Y) \rightarrow Y(c)\]
	induced by $L_c(f)$ and $\epsilon_c$.
\end{defi}
We now prove two results on direct categories and weak orthogonality, denoted by $\pitchfork$, that will be used in what follows. 
\begin{lemma}
	\label{Reedy construction}
	Let $\D$ be a direct category and $\C$ a category equipped with two classes of arrows $(\mathcal{L},\R)$ such that $\mathcal{L}\pitchfork \R$.
	If we define

	\[	 \mathcal{L}^{\D}= \{ \alpha\colon X \rightarrow Y \ \text{in} \ \C^{\D} \ | \ \hat{L}_d(\alpha)\in \mathcal{L} \ \forall d \in \D \} \]
	and \[\R^{\D}=\{\alpha\colon X \rightarrow Y \ \text{in} \ \C^{\D} \ | \ \alpha_d\colon X(d) \rightarrow Y(d) \in \R  \ \forall d \in \D \}\] we have $\mathcal{L}^{\D} \pitchfork \R^{\D}$.
	\begin{proof}
		Consider a commutative square \[
		\vcenter{\hbox{\xymatrix@!0@=15mm{
					A\ar[r]^-{a}\ar[d]_-{l}&X\ar[d]^-{r}\\
					B\ar[r]_-{b}&Y
		}}}
		\]
		where $l \in \mathcal{L}^{\D}$ and $r \in \R^{\D}$, and let $d \colon\ob(\D) \rightarrow \lambda$ be the degree functor . The idea is to use transfinite induction on the degree of objects of $\D$ to find a lift for the given square. Clearly, the only non trivial step is extending a lift for the restriction of the square to $\C^{\D_{\leq \alpha}}$ to a lift for the restriction of the square to $\C^{\D_{\leq \alpha +1}}$.
		
		Consider an object $e\in\D$ such that $d (e)= \alpha$. We have an induced square
		
		\[
		\bfig
		\morphism(0,0)|a|/@{>}@<0pt>/<800,0>[	A(e) \plus{L_e(A)} L_e(B)`X(e);]
		\morphism(0,0)|a|/@{>}@<0pt>/<0,-500>[A(e) \plus{L_e(A)} L_e(B) `B(e);\hat{L}_e(l)]
		\morphism(800,0)|a|/@{>}@<0pt>/<0,-500>[X(e) `Y(e);r_e]
		\morphism(0,-500)|a|/@{>}@<0pt>/<800,0>[B(e)`Y(e);]
		\efig
		\] where the upper-horizontal map is induced by $a$ and the lifts at lower degrees that exist by inductive assumption.
		Choose a filler $k_e\colon B(e) \rightarrow X(e)$, which exists since the left-hand side arrow is in $\mathcal{L}$ and the right-hand side one is in $\R$.
		The collection $\{k_e\colon d(e)=\alpha\}$ gives the desired extension to $\C^{\D_{\leq \alpha +1}}$.
	\end{proof}
\end{lemma}
\begin{lemma}
	\label{preservation of cofs^D}
Let $A,B$ be two cocomplete categories equipped, respectively, with two classes of arrows $(\mathcal{L}_A,\mathcal{R}_A)$ and $(\mathcal{L}_B,\mathcal{R}_B)$ such that $\mathcal{L}_A \pitchfork \mathcal{R}_A$ and $\mathcal{L}_B \pitchfork\mathcal{R}_B$.
Given a cocontinuous functor $F\colon A \rightarrow B$ such that $F(\mathcal{L}_A)\subset \mathcal{L}_B$ and a direct category $\D$, the induced map $F^{\D}\colon A^{\D} \rightarrow B^{\D}$ preserves the direct cofibrations, i.e.\[F(\mathcal{L}^{\D}_A)\subset \mathcal{L}^{\D}_B\]
\begin{proof}
Given a map $\alpha\colon X \rightarrow Y$ in the functor category $A^{\D}$ which is in $\mathcal{L}^{\D}_A$, we have by definition that $F^{\D}(\alpha)$ is in $\mathcal{L}^{\D}_B$ if and only if the latching map at every object $d\in \D$
 \[\hat{L}_d(F(\alpha))\colon FX(d) \plus{L_d(FX)} L_d(FY) \rightarrow FY(d)\] is in $\mathcal{L}_B$. But this map is isomorphic to  $F(\hat{L}_d(\alpha))$ since $F$ is cocontinuous, and this map is in $\mathcal{L}_B$ since $F$ sends maps in $\mathcal{L}^{\D}_A$ to maps in $\mathcal{L}^{\D}_B$ and the latching map of $\alpha$ at $d$ is in $\mathcal{L}_A$.
\end{proof}
\end{lemma}
\begin{ex}
	\label{G is a direct cat}
	The category of globes $\mathbb{G}$ has a natural structure of direct category, with degree function defined by
	\[\begin{matrix}
	\deg \colon \mathbb{G} \rightarrow \mathbb{N} \\
	\qquad n \mapsto n 
	\end{matrix}\]
	Every time we have a coglobular object $\mathbf{D}_{\bullet}\colon \G \rightarrow \C$ in a finitely cocomplete category, we can consider the latching map of $!\colon \emptyset \rightarrow \mathbf{D}_{\bullet}$ at $n$, i.e. the map \[\hat{L}_n(!)\colon L_n(\mathbf{D}_{\bullet}) \rightarrow \mathbf{D}_n\]
	Notice that \[\hat{L}_1(!)=(\mathbf{D}(\sigma_0),\mathbf{D}(\tau_0))\colon \mathbf{D}_0 \coprod \mathbf{D}_0\rightarrow \mathbf{D}_1\] and the other latching maps are obtained inductively from the following cocartesian square
	\[
	\bfig 
	
	\morphism(0,0)|a|/@{>}@<0pt>/<500,0>[L_n(\mathbf{D}_{\bullet})` \mathbf{D}_n;\hat{L}_n(!)]
	\morphism(0,0)|a|/@{>}@<0pt>/<0,-500>[L_n(\mathbf{D}_{\bullet})` \mathbf{D}_n;\hat{L}_n(!)]
	\morphism(500,0)|a|/@{>}@<0pt>/<0,-500>[ \mathbf{D}_n`L_{n+1}(\mathbf{D}_{\bullet});]
	\morphism(0,-500)|a|/@{>}@<0pt>/<500,0>[\mathbf{D}_n`L_{n+1}(\mathbf{D}_{\bullet}); ]
	
	\morphism(500,-500)|l|/@{-->}@<0pt>/<300,-300>[L_{n+1}(\mathbf{D}_{\bullet})` \mathbf{D}_{n+1};\exists ! \hat{L}_{n+1}(!)]
	
	\morphism(500,0)|a|/{@{>}@/^2em/}/<300,-800>[ \mathbf{D}_{n}` \mathbf{D}_{n+1}; \mathbf{D}(\sigma_n)]
	\morphism(0,-500)|l|/{@{>}@/_2em/}/<800,-300>[\mathbf{D}_{n}` \mathbf{D}_{n+1};\mathbf{D}(\tau_n) ]
	\efig 
	\]
\end{ex}
When $\mathbf{D}_{\bullet}\colon\G \rightarrow \mathfrak{C}\rightarrow \wgpd$ is the canonical coglobular $\infty$-groupoid, we will also denote $L_{n}(\mathbf{D}_{\bullet})$ by $S^{n-1}$, borrowing  this notation from topology.
\begin{defi}
Let $I$ (resp. $J$) be the set of boundary inclusions $\{S^{n-1} \rightarrow D_n\}_{n \geq 0}$ (resp. $\{\sigma_n\colon D_n \rightarrow D_{n+1}\}_{n \geq 0}$), and $\mathbb{I}$ its saturation, i.e. the set $^{\pitchfork}(I^{\pitchfork})$ (resp. $\mathbb{J}= {}^{\pitchfork}(J^{\pitchfork})$). 

We say that a map of $\infty$-groupoids $f\colon X \rightarrow Y$ is a cofibration (resp. trivial cofibration) if it belongs to $\mathbb{I}$ (resp. $\mathbb{J}$).

The maps in the class $J^{\pitchfork}$ (resp. $I^{\pitchfork}$) are called fibrations (resp. trivial fibrations).
\end{defi} 
The small object argument provides a factorization system on $\infty$-groupoids given by cofibrations and trivial fibrations. Lemma \ref{Reedy construction} will be applied to this factorization system and to the the direct category structure on $\G$ as defined in Example \ref{G is a direct cat}, to provide a way of inductively extending certain maps in $\wgpd^{\G}$.

Let $*$ denote the terminal object in the category of $\infty$-groupoids. Since every map in $J$ admits a retraction, the following result is straightforward.
\begin{prop}
Every $\infty$-groupoid is fibrant, i.e. the unique map $X \rightarrow *$ is a fibration for every $X \in \wgpd$.
\end{prop}
\begin{defi}
An $\infty$-groupoid $X$ is said to be contractible if every map $S^{n-1} \rightarrow X$ admits an extension to $D_n$, or, equivalently, if the unique map $X \rightarrow *$ is a trivial fibration.
\end{defi}
\begin{prop}
	\label{glob sums are contractible}
Globular sums, seen as objects in the image of the Yoneda embedding functor $y\colon \mathfrak{C} \rightarrow \wgpd$, are contractible $\infty$-groupoids.
\begin{proof}
We proceed by induction on $n=\dim(A)$. If $n=0$ then $A=D_0$, in which case the statement is obvious.

Let $n>0$ and  let us prove that any map $\alpha\colon S^{m-1} \rightarrow A$ extends to $D_m$. By contractibility of $\mathfrak{C}$ we already know this is possible whenever $\dim(A)=n \leq m$, so we assume $m<n$. Consider $\partial A$, whose dimension is $n-1$ by construction, and is therefore contractible by inductive assumption. The map $\partial_{\sigma}\colon \partial A \rightarrow A$ is $(n-2)$-bijective, thanks to Proposition \ref{partial are n-2 bij}, thus $\alpha$ must factor through it since it consists of a pair of parallel $(m-1)$-cells, and contractibility of $\partial A$ allows us to find the desired extension.
\end{proof}
\end{prop}
\section{Suspension functor}
In this section we construct, given $X\in \wgpd$ and two $0$-cells $a,b \in X_0$, the $\infty$-groupoid of morphisms from $a$ to $b$, denoted by $\Omega (X,a,b)$.

This functor will be then extended to an adjunction
\[\xymatrixcolsep{1pc}
\vcenter{\hbox{\xymatrix{
			**[l]\wgpd \xtwocell[r]{}_{\Omega}^{\Sigma}{'\perp}& **[r] S^0 \downarrow\wgpd
}}}
\] between $\infty$-groupoids and doubly pointed $\infty$-groupoids.

To begin with, let us construct a cocontinuous functor $\Sigma\colon \wgpd \rightarrow S^0\downarrow \wgpd$.
First, define 
\[\xymatrixcolsep{2pc} \xymatrix{
	 \mathbb{G} \ar[r]^-{\Sigma}&S^0 \downarrow\wgpd }\]
on objects by 
\[D_n \mapsto (D_{n+1}, \sigma_0^{n+1}, \tau_0^{n+1})\]
and on generating morphisms by
\[ \begin{matrix}
 \sigma_m \mapsto \sigma_{m+1}\\
 \tau_m \mapsto \tau_{m+1}
\end{matrix}\]
Since $S^0 \downarrow\wgpd$ clearly admits globular sums of type $\Sigma$, this functor uniquely extends to a globular functor 
\[
\Sigma\colon	\Theta_0 \rightarrow S^0 \downarrow\wgpd \]
Factor $\Sigma$ as the composite of a bijective on objects and a fully faithful functor 
\[
\bfig
\morphism(0,0)|a|/@{>}@<0pt>/<400,0>[\Theta_0`\mathfrak{D};S]
\morphism(400,0)|a|/@{>}@<0pt>/<600,0>[\mathfrak{D}`S^0 \downarrow \wgpd;V]
\efig 
\] and observe that $S$ preserves globular sums, so that $\mathfrak{D} $ is a globular theory. We will inductively extend $S$ to a map $\mathfrak{C} \rightarrow \mathfrak{D}$ in $\mathbf{GlTh}$.

For this purpose, let $\gamma$ be an ordinal and consider a functor $\mathfrak{C}_{\bullet}\colon\gamma \rightarrow \textbf{GlTh}$ providing a tower for $\mathfrak{C}$.
Without loss of generality, we may assume that, for each $\alpha < \gamma$, $\mathfrak{C}_{\alpha +1}$ is obtained from $\mathfrak{C}_{\alpha }$ by adding a lift $\rho$ to a single parallel pair of morphisms as in the following diagram
\[ \bfig
\morphism(0,0)|a|/@{>}@<2pt>/<750,0>[D_n`A;h_1]
\morphism(0,0)|b|/@{>}@<-2pt>/<750,0>[D_n`A;h_2]
\morphism(0,0)|r|/@{>}@<2pt>/<0,-500>[D_n`D_{n+1};\tau]
\morphism(0,0)|l|/@{>}@<-2pt>/<0,-500>[D_n`D_{n+1};\sigma]
\morphism(0,-500)|r|/@{>}@<0pt>/<750,500>[D_{n+1}`A;\rho]
\efig
\] This means that $\rho \circ \sigma=h_1,\ \rho \circ \tau=h_2$.

Assume, by transfinite induction, that we have already defined a functor
\[\xymatrixcolsep{2pc} \xymatrix{
\mathfrak{C}_{\alpha} \ar[r]^-{\Sigma}&\mathfrak{D} }\] that matches on $\Theta_0$ with the one we started from.

Thanks to the universal property of $\mathfrak{C}_{\alpha +1}$, in order to extend this functor to one on $\mathfrak{C}_{\alpha +1}$ we only need to define an interpretation of $\rho$ under $\Sigma$.
More precisely, we need to define $\Sigma(\rho)\colon\Sigma D_{n+1} \cong D_{n+2} \rightarrow \Sigma A$ such that $\Sigma(\rho) \circ \Sigma(\sigma)=\Sigma(h_1)$ and $\Sigma(\rho) \circ \Sigma(\tau)=\Sigma(h_2)$.
The following diagram in $S^0 \downarrow \mathfrak{C}$ admits a filler in $\mathfrak{C}$ by contractibility, as indicated by the dotted arrow, which is automatically a map under $S^0$:
\[ \bfig
\morphism(0,0)|a|/@{>}@<2pt>/<750,0>[D_{n+1}`\Sigma A;\Sigma(h_1)]
\morphism(0,0)|b|/@{>}@<-2pt>/<750,0>[D_{n+1}`\Sigma A;\Sigma(h_2)]
\morphism(0,0)|r|/@{>}@<2pt>/<0,-500>[D_{n+1}`D_{n+2};\tau]
\morphism(0,0)|l|/@{>}@<-2pt>/<0,-500>[D_{n+1}`D_{n+2};\sigma]
\morphism(0,-500)|r|/@{.>}@<0pt>/<750,500>[D_{n+2}`\Sigma A;\Sigma(\rho)]
\efig
\]
and we define $\Sigma (\rho)$ to be a choice of such a filler.

We have thus succeeded in extending $\Sigma$ to $\mathfrak{C}_{\alpha +1}$. The case of limit ordinals follows from cocontinuity of $\mathfrak{C}_{\bullet}$, which is part of the requirements for a cellular globular theory.

Finally, we get by induction a functor
\[\xymatrixcolsep{2pc} \xymatrix{
	\mathfrak{C} \ar[r]^-{\Sigma}&\mathfrak{D} }\]
and we define the suspension functor to be the cocontinuous extension to $\wgpd$ of the composite $V\circ \Sigma\colon \mathfrak{C}\rightarrow S^0\downarrow \wgpd$, as in Proposition \ref{UP of models}.
 
 Being a cocontinuous functor between locally presentable categories, $\Sigma$ admits a right adjoint that gives rise to an adjunction
 \[\xymatrixcolsep{1pc}
 \vcenter{\hbox{\xymatrix{
 			**[l]\wgpd \xtwocell[r]{}_{\Omega}^{\Sigma}{'\perp}& **[r] S^0 \downarrow\wgpd
 }}}
 \] 
  By adjunction, the underlying globular set of $\Omega (X,a,b)$ is given by 
 \[\Omega (X,a,b)_n\colon =\{x \in X_{n+1}| \ s_0^{n+1}(x)=a, \ t_0^{n+1}(x)=b \} \]
 \begin{rmk}
 	\label{cocontinuity of UoSIgma}
If we compose $\Sigma$ with the forgetful functor $U\colon S^0 \downarrow\wgpd \rightarrow \wgpd $, we get a functor which is no longer cocontinuous.
Nevertheless, it is well known that $U$ creates connected colimits, therefore $U\circ \Sigma$ preserves all such.
Because $\Sigma(I)\subset I$, where $I=\{S^{n-1} \rightarrow D_n\}_{n \geq 0}$ is the set of sphere inclusions, we therefore have that $U\circ \Sigma$ preserves cofibrations (i.e. maps in $\mathbb{I}$, the saturation of $I$). A similar situation is treated in Lemma 1.3.52 of \cite{CIS}.
 \end{rmk}
To justify the notation we observe the following fact: if we interpret a map $(\alpha,\beta)\colon S^{n+1} \rightarrow X$ as a map $(\hat{\alpha}, \hat{\beta})\colon S^{n} \rightarrow \Omega (X,a,b)$, where $a=s_0^{n+1} (\alpha)$ and $b=t_0^{n+1} (\beta)$, then it holds true that
 \[\pi_{n}(\Omega ( X ,a,b),\hat{\alpha}, \hat{\beta}) \cong \pi_{n+1}(X,\alpha,\beta)\] where, by definition, given an $\infty$-groupoid $Y$ and two $(n-1)$-cells $a,b\in Y_{n-1}$, we have \[\pi_{n}(Y,a,b)=\{[f] \colon f\in Y_n, s(f)=a, t(f)=b\}\] with $[f]=[g]$ if and only if there exists an $(n+1)$-cell $H\in Y_{n+1}$ such that $s(H)=f, t(H)=g$ (see also \cite{AR2}, Definition 4.11).
 \begin{prop}
 	\label{omega preserves contractibles}
Let $(X,(a,b))$ be an object in $S^0 \downarrow\wgpd$. Assume that $X$ is a contractible $\infty$-groupoid. Then $\Omega (X,a,b)$ is again contractible.
\begin{proof}
Diagrams of the form:
\[ \bfig
\morphism(0,0)|a|/@{>}@<0pt>/<800,0>[S^{n-1}`\Omega (X,a,b);]
\morphism(0,0)|a|/@{>}@<0pt>/<0,-500>[S^{n-1}`D_n;]
\morphism(0,-500)|a|/@{-->}@<0pt>/<800,500>[D_n`\Omega (X,a,b);]
\efig \]
correspond, under the adjunction $\Sigma \dashv \Omega$, to diagrams under $S^0$ of the form

\[ \bfig
\morphism(0,0)|a|/@{>}@<0pt>/<800,0>[S^n`X;]
\morphism(0,0)|a|/@{>}@<0pt>/<0,-500>[S^n`D_{n+1};]
\morphism(0,-500)|a|/@{-->}@<0pt>/<800,500>[D_{n+1}`X;]
\efig \]
By assumption, all such diagrams admit an extension, which concludes the proof.
\end{proof}
 \end{prop}
The following lemma will be used quite frequently in the forthcoming sections. Its proof is straightforward and it is thus left to the reader.
\begin{lemma}
	\label{decomposition of glob sum}
For every globular sum $A$ there exist unique globular sums $\alpha_1,\ldots , \alpha_q$ such that 
\begin{equation}
A\cong \Sigma \alpha_1 \plus{ D_0} \Sigma \alpha_2\plus{ D_0} \ldots \plus{ D_0} \Sigma \alpha_q
\end{equation} the colimit being taken over the maps
\[
\bfig 
\morphism(0,0)|a|/@{>}@<0pt>/<-400,-400>[D_0`\Sigma \alpha_{i};\top]
\morphism(0,0)|a|/@{>}@<0pt>/<400,-400>[D_0`\Sigma \alpha_{i+1};\perp]
\efig 
\] where we denote the image via  $\Sigma\colon \wgpd \rightarrow S^0\downarrow \wgpd$ of any globular sum $B$ by $(\Sigma B, \perp,\top)$.
\end{lemma}
	\section{Cylinders on globes}
In this section we define $\infty$-groupoids $\cyl(D_n)$ for each $n \geq 0$ that represents cylinders between $n$-cells. These should be thought as homotopies between cells that are not parallel, so that one needs to provide first homotopies between the 0-dimensional boundary, then between the 1-dimensional boundary adjusted using those homotopies, and so on.

\begin{ex}
	By definition, $\cyl (D_0)$ is the free $\infty$-groupoid on a $1$-cell. Therefore, giving a $0$-cylinder in an $\infty$-groupoid $X$ is equivalent to specifying one of its $1$-cells.
	
	If we go one dimension up, we have that a $1$-cylinder $C\colon \cyl(D_1) \rightarrow X$ consists of the following data
	\[
	\bfig
	\morphism(0,0)|a|/@{>}@<0pt>/<300,0>[a`b;\alpha]
	\morphism(0,0)|l|/@{>}@<0pt>/<0,-300>[a`c;f]
	\morphism(300,0)|r|/@{>}@<0pt>/<0,-300>[b`d;g]
	\morphism(0,-300)|b|/@{>}@<0pt>/<300,0>[c`d; \beta]
	\morphism(250,-75)|a|/@{=>}@<0pt>/<-150,-150>[`; C]
	\efig\]
	Following the notation in the next section, we have that $f=C \circ \cyl (\sigma)$ and  $g=C \circ \cyl (\tau)$. Moreover, $\alpha=C \circ \iota_0$ and $\beta=C \circ \iota_1$.
	
	This cylinder represents the fact that to give a ``homotopy'' from $\alpha$ to $\beta$ we first have to give one between $a$ and $c$, and one from $b$ to $d$. Only then can we compose these with the cells we want to compare, and consider the 2-cells such as $C$ that fill the resulting square, thus giving us the homotopy we are looking for.
\end{ex}
\subsection{Cylinders and boundaries}
\begin{defi}
	\label{cyl defi}
We define, by induction on $n\in\mathbb{N}$, a coglobular object $\cyl(D_{\bullet})\in \wgpd^{\G}$, together with a map \[(\iota_0,\iota_1)\colon D_{\bullet} \coprod D_{\bullet} \rightarrow \cyl(D_{\bullet}) \]
We begin by setting
\[\cyl(D_0)=D_1, \ (\iota_0,\iota_1)_0=(\sigma,\tau)\colon D_0 \coprod D_0 \rightarrow D_1\]
Now, let $n>0$ and assume we have constructed \[\cyl(D_{\bullet})\in \wgpd^{\G_{\leq n-1 }} \ \text{and} \ (\iota_0, \iota_1)\colon D_{\bullet} \coprod D_{\bullet} \rightarrow \cyl(D_{\bullet})\]
We then define $\cyl(D_n)$ as the colimit in $\wgpd$ of the following diagram:
\begin{equation}
\label{cyl D_n}
\bfig
\morphism(-1000,-250)|a|/@{>}@<0pt>/<500,250>[D_n` D_n \plus{D_0} D_1;w]
\morphism(-1000,-250)|a|/@{>}@<0pt>/<500,-250>[D_n` \Sigma \cyl(D_{n-1});\Sigma (\iota_0)]
\morphism(-1000,-750)|r|/@{>}@<0pt>/<500,250>[D_n` \Sigma \cyl(D_{n-1}); \ \Sigma (\iota_1)]
\morphism(-1000,-750)|l|/@{>}@<0pt>/<500,-250>[D_n` D_1 \plus{D_0} D_n;w]
\efig
\end{equation}
Next, we define $\iota_0,\iota_1\colon D_n \rightarrow \cyl(D_n)$ respectively as the composites
\[\bfig
\morphism(0,0)|a|/@{>}@<0pt>/<500,0>[D_n` D_n \plus{D_0} D_1;\iota]
\morphism(500,0)|a|/@{>}@<0pt>/<600,0>[D_n \plus{D_0} D_1` \cyl(D_n);]
\efig
\] 
\[\bfig
\morphism(0,0)|a|/@{>}@<0pt>/<500,0>[D_n` D_1 \plus{D_0} D_n;\iota ]
\morphism(500,0)|a|/@{>}@<0pt>/<600,0>[D_1 \plus{D_0} D_n` \cyl(D_n);]
\efig
\] 
where the unlabelled maps are given by the colimit inclusions.

Finally, for $\epsilon=\sigma,\tau$, we construct the induced map $\cyl(\epsilon)\colon \cyl(D_{n-1}) \rightarrow \cyl(D_n)$ by induction.
We define $\cyl(\sigma), \cyl(\tau) \colon\cyl(D_0) \rightarrow \cyl(D_1)$ respectively as the lower and upper composite maps
\begin{equation}
\label{cyl(sigma)}
\bfig
\morphism(-1500,-500)|a|/{@{>}@/^2em/}/<1000,500>[\cyl(D_0)\cong D_1` D_1 \plus{D_0} D_1;i_1]
\morphism(-1500,-500)|b|/{@{>}@/^-2em/}/<1000,-500>[\cyl(D_0)\cong D_1` D_1 \plus{D_0} D_1;i_0]
\morphism(-1000,-250)|a|/@{>}@<0pt>/<500,250>[D_1` D_1 \plus{D_0} D_1;w]
\morphism(-1000,-250)|a|/@{>}@<0pt>/<500,-250>[D_1` \Sigma \cyl(D_0);\Sigma (\iota_0)]
\morphism(-1000,-750)|r|/@{>}@<0pt>/<500,250>[D_1` \Sigma \cyl(D_0); \ \Sigma (\iota_1)]
\morphism(-1000,-750)|l|/@{>}@<0pt>/<500,-250>[D_1` D_1 \plus{D_0} D_1;w]
\morphism(-500,-1000)|b|/{@{>}@/^-2em/}/<1000,500>[D_1 \plus{D_0} D_1` \cyl(D_1);]
\morphism(-500,0)|b|/{@{>}@/^2em/}/<1000,-500>[D_1 \plus{D_0} D_1` \cyl(D_1);]
\efig
\end{equation}
We then inductively define for $\epsilon=\sigma, \tau$ the structural map $\cyl(\epsilon)\colon\cyl(D_{n-1})\rightarrow \cyl(D_n)$ as the map induced on colimits by the following natural transformation
\[
\bfig

\morphism(-1300,-250)|a|/@{>}@<0pt>/<500,250>[D_{n-1}` D_{n-1} \plus{D_0} D_1;w]
\morphism(-1300,-250)|a|/@{>}@<0pt>/<500,-250>[D_{n-1}` \Sigma \cyl(D_n-2);\Sigma (\iota_0)]
\morphism(-1300,-750)|r|/@{>}@<0pt>/<500,250>[D_{n-1}` \Sigma \cyl(D_n-2); \ \Sigma (\iota_1)]
\morphism(-1300,-750)|l|/@{>}@<0pt>/<500,-250>[D_{n-1}` D_1 \plus{D_0} D_{n-1};w]

\morphism(-800,0)|a|/@{>}@<0pt>/<1300,0>[ D_{n-1} \plus{D_0} D_1`D_n \plus{D_0} D_1;\epsilon \plus{D_0} 1]
\morphism(-1300,-250)|a|/@{>}@<0pt>/<1300,0>[D_{n-1}` D_n ;\epsilon]
\morphism(-800,-500)|a|/@{>}@<0pt>/<1300,0>[\Sigma \cyl(D_n-2)` \Sigma \cyl(D_n-1);\Sigma \cyl(\epsilon)]
\morphism(-1300,-750)|a|/@{>}@<0pt>/<1300,0>[D_{n-1}` D_n ;\epsilon]
\morphism(-800,-1000)|a|/@{>}@<0pt>/<1300,0>[D_1  \plus{D_0}D_{n-1} `D_1 \plus{D_0} D_n;1 \plus{D_0}  \epsilon]

\morphism(0,-250)|a|/@{>}@<0pt>/<500,250>[D_n` D_n \plus{D_0} D_1;w]
\morphism(0,-250)|a|/@{>}@<0pt>/<500,-250>[D_n` \Sigma \cyl(D_n-1);\Sigma (\iota_0)]
\morphism(0,-750)|r|/@{>}@<0pt>/<500,250>[D_n` \Sigma \cyl(D_n-1); \ \Sigma (\iota_1)]
\morphism(0,-750)|l|/@{>}@<0pt>/<500,-250>[D_n` D_1 \plus{D_0} D_n;w]
\efig
\]
\end{defi}
\begin{defi}
Given an $\infty$-groupoid $X$, an $n$-cylinder in $X$ is a map $C:\cyl (D_n) \rightarrow X$. 
We denote the source and target cylinders of $C$ by, respectively
\[s(C)=C\circ \cyl(\sigma), \ t(C)=C\circ \cyl(\tau)\]
If $C\circ\iota_0=A, C\circ \iota_1=B$, then we write $C\colon A \curvearrowright B$.

By \eqref{cyl D_n}, an $n$-cylinder $F\colon A \curvearrowright B$ in an $\infty$-groupoid $X$ is given by a pair of $1$-cells $F_{s_0},F_{t_0}$ in $X$ (sometimes denoted just with $F_s,F_t$ when there is no risk of ambiguity) and an $(n-1)$-cylinder $\bar{F}\colon F_{t_0} A \curvearrowright B F_{s_0}$ in $\Omega \left( X, s(F_{s_0}), t(F_{t_0}) \right)$. $F$ and $\bar{F}$ will often be referred to as mutually transposed.
We will sometimes refer to the cell $C\circ \iota_0$ (resp. $C\circ \iota_1$) as the top (resp. bottom) cell of $C$, and denote it with $C_0$ (resp. $C_1$).
\end{defi}
\begin{ex}
	A $2$-cylinder $C\colon A \curvearrowright B$ in $X$ consists of a pair of $1$-cells $f=C_{s_0}, g=C_{t_0}$ and a $1$-cylinder $\bar{C}\colon gA\curvearrowright Bf$ in $\Omega \left( X, s(f),t(g)  \right) $. It can also be represented as the following data in $X$
	\[
	\xymatrix{s(f) \ar[dd]_{f} \rrtwocell{A}  \Dtwocell{ddrr}{} &&s(g) \ar[dd]^{g}  \\ & \\
		t(f)\ar[rr]^{t(B)} && t(g) }
	\quad \xymatrix{ \\ \Rrightarrow}\quad
	\xymatrix{s(f) \ar[rr]^{s(A)} \ar[dd]_{f}  \Dtwocell{ddrr}{} &&s(g) \ar[dd]^{g} \\ \\
		t(f) \rrtwocell{B} &&t(g)}
	\]
Or, in a way that better justifies its name, as 
\begin{equation}
\cd[@+3em]{
	A \ar@/^1em/[r]^-{} \ar@/_1em/[r]_-{} \ar[d]_-{f}
	\Dtwocell{r}{A} \Dtwocell[0.33]{dr}{}
	\Dtwocell[0.67]{dr}{} \ulthreecell{dr}{} &
	B \ar[d]^-{g} \\
	A' \ar@/^1em/[r]^-{} \ar@/_1em/[r]_-{} \Dtwocell{r}{B}  & B'
}
\end{equation}
where the front face is the square (i.e. $1$-cylinder) given by $t(C)$, and the back one is $s(C)$.
\end{ex}
We will often denote the source of the $n$-th latching map $\hat{L}_n(\iota_0,\iota_1)$ as $\partial \cyl (D_n)$. This can be constructed as the following pushout 
\begin{equation}
\label{2}
\bfig 
\morphism(0,0)|a|/@{>}@<0pt>/<750,0>[S^{n-1}\coprod S^{n-1}` \cyl(S^{n-1});(\iota_0,\iota_1)]
\morphism(0,0)|a|/@{>}@<0pt>/<0,-500>[S^{n-1}\coprod S^{n-1}`D_{n} \coprod D_{n};]
\morphism(0,-500)|a|/@{>}@<0pt>/<750,0>[D_{n} \coprod D_{n}` \partial \cyl(D_{n});]
\morphism(750,0)|a|/@{>}@<0pt>/<0,-500>[\cyl(S^{n-1})`  \partial \cyl(D_{n});]
\efig
\end{equation}
Let us now prove the following result:
 \begin{prop}
 	\label{direct cof cyl}
The natural map 
\[
\iota=(\iota_0, \iota_1)\colon D_{\bullet} \coprod D_{\bullet} \rightarrow \cyl(D_{\bullet})
\] is a direct cofibration in $\wgpd^{\G}$ (i.e. it belongs to the class $\mathbb{I}^{\G}$ according to the notation established in Lemma \ref{Reedy construction} ).
\begin{proof}
We prove by induction on $n$ that the latching map at $n$ \[\hat{L}_n(\iota_0, \iota_1)\colon \partial \cyl (D_n) \rightarrow \cyl(D_n)\] fits into a cocartesian square of the form
\begin{equation}
	\label{cyl' D_n}
	\bfig
	\morphism(0,0)|l|/@{>}@<0pt>/<0,-500>[S^{n}` \partial \cyl(D_{n});]
	\morphism(0,0)|a|/@{>}@<0pt>/<850,0>[S^{n}` D_{n+1};\partial]
	\morphism(0,-500)|a|/@{>}@<0pt>/<850,0>[\partial \cyl(D_{n})` \cyl(D_{n});\hat{L}_n(\iota_0, \iota_1)]
	\morphism(850,0)|a|/@{>}@<0pt>/<0,-500>[D_{n+1}` \cyl(D_{n});]
	\efig
	\end{equation} and is therefore in $\mathbb{I}$. Observe that the statement is trivially true by definition if $n=0$, so we assume $n>0$ and its validity for every $k<n$.
In fact, we are going to prove that there is a pushout square of the form
\begin{equation}
\label{1}
	\bfig
\morphism(0,0)|a|/@{>}@<0pt>/<1200,0>[\Sigma \partial \cyl(D_{n-1})` \Sigma  \cyl(D_{n-1});\Sigma\left(\hat{L}_{n-1}(\iota)\right)]
\morphism(0,0)|a|/@{>}@<0pt>/<0,-400>[\Sigma \partial \cyl(D_{n-1})` \partial \cyl(D_{n});]
\morphism(0,-400)|a|/@{>}@<0pt>/<1200,0>[\partial \cyl(D_{n})` \cyl(D_{n});\hat{L}_n(\iota)]
\morphism(1200,0)|a|/@{>}@<0pt>/<0,-400>[\Sigma  \cyl(D_{n-1})`  \cyl(D_n);]
\efig 
\end{equation}
 and then conclude by applying the inductive hypothesis.
 
 We prove this representably, i.e. we have to prove that, given $n$-cells $A,B$ and $1$-cells $C_s,C_t$ in an $\infty$-groupoid $X$, satisfying $t^n(A)=s(C_t), s^n(B)=t(C_s)$, together with an $(n-1)$-cylinder $C'\colon A' \curvearrowright B'$ in $\Omega \left( X,s^n(A),t^n(B)\right)$ and a pair of parallel $(n-1)$-cylinders in $X$ of the form $C^1\colon s(A)\curvearrowright s(B),C^2\colon t(A)\curvearrowright t(B)$, there exists a unique $n$-cylinder $C\colon A \rightarrow B$ in $X$ with $\overline{C}=C', s(C)=C^1,t(C)=C^2$ provided $s(C')=\overline{C^1}, t(C')=\overline{C^2}$. This fact is an easy consequence of Definition \ref{cyl defi}.
\end{proof}
 \end{prop}
\begin{defi}
We call $\partial \cyl (D_n)$ the boundary of the $n$-cylinder.
Given an $n$-cylinder in $X$ $C\colon\cyl (D_n) \rightarrow X$, we call the boundary of $C$, denoted by $\partial C$, the following composite
\[
\bfig
\morphism(0,0)|a|/@{>->}@<0pt>/<600,0>[ \partial \cyl (D_n)`    \cyl (D_n);] \morphism(600,0)|a|/@{>}@<0pt>/<500,0>[ \cyl (D_n)`   X;C]
\efig
\]
\end{defi}
Thanks to \eqref{2}, we know that specifying the boundary of an $n$-cylinder in an $\infty$-groupoid $X$ is equivalent to providing the following data:
\begin{itemize}
\item a pair of parallel $(n-1)$-cylinders $C\colon A\curvearrowright B,D\colon	A'\curvearrowright B'$ in $X$;
\item a pair of $n$-cells $\alpha\colon A \rightarrow A', \beta \colon B \rightarrow B'$ in $X$.
\end{itemize}
We can define a map of coglobular $\infty$-groupoids $\mathbf{C}_{\bullet}\colon \cyl(D_{\bullet})\rightarrow D_{\bullet}$ that fits into the following factorization of the codiagonal map
\[
\bfig
\morphism(0,0)|a|/@{>->}@<0pt>/<750,0>[ D_{\bullet} \coprod D_{\bullet}`\cyl (D_{\bullet});(\iota_0,\iota_1)]
\morphism(750,0)|a|/@{>}@<0pt>/<500,0>[ \cyl (D_{\bullet})`D_{\bullet};\mathbf{C}_{\bullet}]
\efig
\]
by solving the lifting problem
\[\bfig
\morphism(0,0)|a|/@{>}@<0pt>/<500,0>[ D_{\bullet} \coprod D_{\bullet}`D_{\bullet};\nabla]
\morphism(0,0)|a|/@{>->}@<0pt>/<0,-400>[D_{\bullet} \coprod D_{\bullet}`\cyl(D_{\bullet});(\iota_0,\iota_1)]
\morphism(0,-400)|r|/@{-->}@<0pt>/<500,400>[\cyl(D_{\bullet})`D_{\bullet};\mathbf{C}_{\bullet}]
\efig
\] using Propositions \ref{direct cof cyl} and \ref{glob sums are contractible}.

\section{Path-space of an $\infty$-groupoid}
The work done so far allows us to define the underlying globular set of a candidate for the path-object associated with an object $X\in \wgpd$, and to endow it with a non-trivial algebraic structure.
\begin{defi}
	\label{path space defi}
	We define a functor $\p\colon \wgpd \rightarrow [\G^{op}, \mathbf{Set}]$ by setting $\p X=\wgpd(\cyl(D_{\bullet}),X)$, where the globular structure is induced by the coglobular object $\cyl(D_{\bullet})\colon \G \rightarrow \wgpd$.
	
	Precomposition with $\iota\colon D_{\bullet}\amalg D_{\bullet} \rightarrow \cyl(D_{\bullet})$ yields a natural map \[p_X=(p_0,p_1)\colon\p X \rightarrow X\times X\]
\end{defi}
Recall that a map of globular sets $f\colon X \rightarrow Y$ is called a trivial fibration if it has the right lifting property with respect to the set of maps $\{S^{n-1}\rightarrow D_n\}_{n\geq 0}$.
Similarly, $f$ is called a fibration if it has the right lifting property with respect to the set of maps $\{\sigma_n\colon D_n \rightarrow D_{n+1}\}_{n\geq 0}$.
\begin{prop}
	The map $p_X=(p_0,p_1)\colon\p X \rightarrow X\times X$ (resp. $p_i\colon \p X \rightarrow X$ for $i=0,1$) is a fibration (resp. trivial fibration) of globular sets.
	\begin{proof}
		Let us first prove the claim about $p_X$. We have to prove it lifts against maps of the form $\sigma_n\colon D_n \rightarrow D_{n+1}$ for $n\geq 0$. This is equivalent to saying that given $(n+1)$-cells $A,B$ in $X$ and an $n$-cylinder in $X$ of the form $C\colon s(A) \curvearrowright s(B)$ we can always extend $C$ to $C'\colon A \curvearrowright B$, so that $s(C')=C$.
		
		We prove this by induction on $n$. If $n=0$ we define  $t(C')=BCA^{-1}$, using the convention that juxtaposition stands for the choice of a composition operation and $()^{-1}$ for the choice of an inverse. To finish this step we need to find a $2$-cell $(BCA^{-1})A \rightarrow BC$, which is certainly possible thanks to the contractibility of $\mathfrak{C}$.
		
		Let $n>0$, and assume the statement holds for every integer $k<n$. We are given an $n$-cylinder $C\colon A \curvearrowright B$ in $ X$, together with $(n+1)$-cells $\Gamma \colon A \rightarrow A'$ and $\Delta\colon B \rightarrow B'$.
		We thus get an $(n-1)$-cylinder $\bar{C}\colon C_t A \curvearrowright BC_s$ in $\Omega (X,s(C_s),t(C_t))$, together with $n$-cells $C_t\Gamma\colon C_t A \rightarrow C_t A', \ \Delta C_s \colon B C_s \rightarrow B'C_s$ in the same groupoid.
		By inductive hypothesis we now obtain an $n$-cylinder $\bar{C'}\colon C_t  \Gamma \curvearrowright \Delta C_s$ in $\Omega (X,s(C_s),t(C_t))$, with source $\overline{C}$, which transposes to give the desired $(n+1)$-cylinder $C'\colon \Gamma \curvearrowright \Delta$ in $X$, whose source is $C$.
		
		We now prove that the map $p_0\colon \p X \rightarrow X$ is a trivial fibration, the other case being entirely similar. This amounts to prove it lifts against all the maps of the form $S^{n-1}\rightarrow D_n$.
		The case $n=0$ is equivalent to proving that, given a $ 0$-cell $x\in X_0$ we can find a $0$-cylinder $C\colon \cyl(D_0) \rightarrow X$, i.e. a $1$-cell in $X$, such that its source is precisely $x$. A possible solution is to take the trivial cylinder on $x$, i.e. $x \circ \mathbf{C}_0$.
		
		If $n=1$, we are given $1$-cells $C,D$ and $\gamma$, and we have to extend this to a $1$-cylinder $\Gamma\colon \gamma \rightarrow \delta$, whose source and target are, respectively, $C$ and $D$. If we set $\delta=D\gamma C^{-1}$ then we are left with providing a $2$-cell $\overline{\Gamma}\colon(D\gamma C^{-1})C \rightarrow D\gamma$, which surely exists thanks to the contractibility of $\mathfrak{C}$.
		
		Now let $n>1$, and assume we have a pair of parallel $ (n-1)$-cylinders $(C,D)$ in $X$, i.e. $\epsilon(C)=\epsilon(D)$ for $\epsilon=s,t$, together with an $n$-cell $\Gamma\colon C_0 \rightarrow D_0$. Notice that, in particular, we have that $C_{\epsilon}$ and $D_{\epsilon}$ are parallel for $\epsilon=0,1$. These data transpose to give a pair of parallel $(n-2)$-cylinders $(\bar{C},\bar{D})$ in $\Omega (X,s(C_s),t(C_t))$. Moreover, we also get an $(n-1)$-cell $C_t\Gamma\colon \bar{C}_0=C_t(C_0) \rightarrow \bar{D}_0= C_t(D_0)$ in $\Omega (X,s(C_s),t(C_t))$.
		
		By inductive hypothesis we thus get an $(n-1)$-cylinder $\chi \colon C_t\Gamma \curvearrowright \epsilon$. By construction, the source (resp. target) of $\epsilon$ are of the form $(\bar{C}_1) C_s$ (resp. $(\bar{D}_1) C_s$). Thanks to Lemma 4.12 in \cite{AR2}, we see that there exists an $n$-cell $\Delta$ in $X$, and an $n$-cell $\epsilon \rightarrow \Delta C_s$ in $\Omega (X,s(C_s),t(C_t))$. We can compose this piece of data with $\chi$ using Lemma \ref{cell +cyl+cell}, getting an $(n-1)$-cylinder $\bar{C'}\colon C_t \Gamma \curvearrowright \Delta C_s$ in $\Omega (X,s(C_s),t(C_s))$, which transposes to give the desired cylinder $C'\colon \Gamma \curvearrowright \Delta$ in $X$, having $C$ as source and $D$ as target.
	\end{proof}
\end{prop}
The codiagonal factorization \[
\bfig
\morphism(0,0)|a|/@{>->}@<0pt>/<750,0>[ D_{\bullet} \coprod D_{\bullet}`\cyl (D_{\bullet});(\iota_0,\iota_1)]
\morphism(750,0)|a|/@{>}@<0pt>/<500,0>[ \cyl (D_{\bullet})`D_{\bullet};\mathbf{C}_{\bullet}]
\efig
\] induces, by applying the functor $\wgpd(\bullet,X)$, a diagonal factorization in $[\G^{op},\mathbf{Set}]$ of the form
\[
\bfig
\morphism(0,0)|a|/@{>->}@<0pt>/<500,0>[X`\p X;c]
\morphism(500,0)|a|/@{>->}@<0pt>/<500,0>[ \p X`X\times X;p]
\efig
\] 
It is important to remark that any sensible structure of $\infty$-groupoid on the globular set $\p X$ will not make $c$ into a map in $\wgpd$, and we will see later on a possible way of fixing this. Also notice that $p$ is a fibration thanks to the previous proposition, and $c$ is easily seen to be a section.
\begin{ex}
	Given a coherator $\mathfrak{C}$, if at some stage of the tower of globular theories witnessing its cellularity, a map $\rho\colon D_1 \rightarrow D_1 \amalg_{ D_0} D_1$ is added, with source given by $i_0\circ \sigma$ and target $i_1\circ \tau$ ($i_j$ being the inclusion on the $j$-th factor), then it is reasonable to define $\cyl(\rho)=\hat{\rho}$, thanks to the freeness of $\mathfrak{C}$, where $\hat{\rho}$ is defined in Definition \ref{rho hat defi}. Indeed, the elementary interpretation that we define in Section 9 respects the source and target operations by default in the case of $1$-dimensional homogeneous operations.
	
	It is straightforward to check that, in general, $(\mathbf{C}_A, \mathbf{C}_B)\circ \hat{\rho} \neq \mathbf{C}_{(A,B)\circ \rho}$, so that in this case $c\colon X \rightarrow \mathbb{P}X$ cannot be made into a strict map of $\infty$-groupoids.
\end{ex}
Let us now endow $\p X$ with part of the structure that any $\infty$-groupoid possesses. We will show how to endow it with a system of composition (see \cite{BAT} for the definition of this concept in the context of globular operads), inverses for each $n$-cell together with suitable coherence cells and identity on $n$-cells for every $n\geq 0$.

To make this more precise, we define the globular theory $\mathfrak{D}$ freely generate by these operations, and then extend the functor $\p$ to $\p \colon \wgpd \rightarrow \mathbf{Mod}(\mathfrak{D})$.
\begin{defi}
	\label{structure systems}
	A system of composition in a globular theory $\mathfrak{D}$ consists of a family of maps $\{\mathbf{c}_n\colon D_n \rightarrow D_n \amalg_{D_{n-1}} D_n\}_{n \geq 1}$ such that $\mathbf{c}_n \circ \sigma=i_1 \circ \sigma$ and $\mathbf{c}_n \circ \tau=i_2 \circ \tau$, where $i_1$ (resp. $i_2$) denotes the colimit inclusion onto the first (resp. second) factor.
	
	A system of identities (with respect to a chosen system of composition) consists of a family of maps $\{\mathbf{id}_n\colon D_{n+1} \rightarrow D_n \}_{n \geq 0}\cup \{\mathbf{l}_n,\mathbf{r}_n\colon D_n \rightarrow D_{n-1}\}_{n\geq 2}$ such that $\mathbf{id}_n \circ \epsilon= 1_{D_n},$ for every $n\geq 0$ and $\epsilon=\sigma, \tau$, $\mathbf{l}_n \circ \sigma=1_{D_{n-1}}, \ \mathbf{l}_n\circ \tau = (1_{D_{n-1}},\tau \circ \mathbf{id}_{n-2} )\circ \mathbf{c}_{n-1}$ and $\mathbf{r}_n \circ \sigma=1_{D_{n-1}}, \ \mathbf{r}_n\circ \tau = (\sigma \circ \mathbf{id}_{n-2},1_{D_{n-1}})\circ \mathbf{c}_{n-1}$.
	
	A system of inverses (with respect to chosen systems of compositions and identities) consists of a family of maps $\{\mathbf{i}_n\colon D_n \rightarrow D_n\}_{n \geq 1} \cup \{\mathbf{k}_n^s,\mathbf{k}_n^t\colon D_n \rightarrow D_{n-1} \}_{n\geq 2}$ such that $\mathbf{i}_n\circ \sigma= \tau, \  \mathbf{i}_n\circ \tau = \sigma$, $\mathbf{k}_n^s\circ \sigma =\sigma \circ \mathbf{id}_{n-2}, \ \mathbf{k}_n^s \circ \tau =(1_{D_{n-1}},\mathbf{i}_{n-1})\circ\mathbf{c}_{n-1}, \ \mathbf{k}_n^t\circ \sigma =\tau \circ \mathbf{id}_{n-2}$ and $\mathbf{k}_n^t\circ \tau =(\mathbf{i}_{n-1},1_{D_{n-1}}) \circ\mathbf{c}_{n-1} $.
	
	If $\mathfrak{D}$ admits a choice of such three systems, given a globular functor $\mathbf{F}\colon \mathfrak{D} \rightarrow \wgpd$ we say that for every $\infty$-groupoid $X$, the $\mathfrak{D}$-model $\wgpd(\mathbf{F},X)$ can be endowed with such systems.
\end{defi}
If we let $\mathfrak{D}$ be the globular theory freely generated by a system of composition, a system of identities and a system of inverses (i.e. given a globular theory $\mathfrak{A}$, globular functors $\mathfrak{D} \rightarrow \mathfrak{A}$ correspond to a choice of such three systems in $\mathfrak{A}$, see \cite{AR1} for a construction of $\mathfrak{D}$) then extending the functor $\p $ to the more structured codomain $\mathbf{Mod}(\mathfrak{D})$ is equivalent to the following extension problem
\[
\bfig
\morphism(0,0)|a|/@{>}@<0pt>/<500,0>[\Theta_0`\wgpd;\cyl]
\morphism(0,0)|a|/@{>}@<0pt>/<0,-400>[\Theta_0` \mathfrak{D};]
\morphism(0,-400)|r|/@{-->}@<0pt>/<500,400>[\mathfrak{D}`\wgpd;\cyl]
\efig
\] which, in turn, is equivalent to defining interpretations $\cyl(f)$ for each of the generators $f$ of $\mathfrak{D}$ described above, satisfying the appropriate equations.
\begin{thm}
	\label{algebraic structure on PX}
	The functor $\cyl\colon\Theta_0 \rightarrow \wgpd$ admits an extension to $\mathfrak{D}$.
	
	Equivalently, given any $\infty$-groupoid $X$, the globular set $\p X$ can be endowed with a system of composition, a system of identities and a system of inverses.
\end{thm}
The proof will be subdivided into several lemmas, the first one addressing the system of compositions. Before proving this, we need a construction representing the composition of an $n$-cylinder with a pair of $(n+1)$-cells attached at the top and bottom, respectively. This is an instance of the vertical composition of degenerate cylinders, as defined in Section \ref{vert comp of deg cyl (section)}.
\begin{lemma}
	\label{cell +cyl+cell}
Given an $\infty$-groupoid $X$, an $n$-cylinder $C\colon A \curvearrowright B$ in $X$ and $(n+1)$-cells $\alpha\colon A' \rightarrow A$ and $\beta\colon B \rightarrow B'$ we can compose these data to get an $n$-cylinder $\beta C \alpha \colon A' \curvearrowright B'$. Moreover, $\epsilon(\beta C \alpha)=\epsilon(C)$ for $\epsilon=s,t$.
\begin{proof}
We prove this by induction on $n$, the case $n=0$ being straightforward. Let's assume $n>0$ and that we have already defined this operation for every $k<n$. We can transpose the data at hand to get an $(n-1)$-cylinder $\overline{C}\colon C_{t_0} A \curvearrowright B C_{s_0}$ in $\Omega\left( X,s^n(A),t^n(B) \right)$ and $n$-cells $C_{t_0}\alpha \colon C_{t_0}A' \rightarrow C_{t_0} A,\ \beta C_{s_0}\colon BC_{s_0} \rightarrow B'C_{s_0}$, where juxtaposition denotes the result of composing using the whiskering $w$'s defined in \ref{w's maps}. By inductive hypothesis we can compose these data to get an $(n-1)$-cylinder $(\beta C_{s_0})\overline{C} (\alpha C_{t_0})\colon C_{t_0}A' \curvearrowright B'C_{s_0}$. Finally, we define $(\beta C \alpha)_{\epsilon_0}=C_{\epsilon_0}$ for $\epsilon=s,t$, and $\overline{\beta C \alpha}=(\beta C_{s_0})\overline{C} (\alpha C_{t_0})$.

The statement on source and target cylinders follows easily from the inductive argument we have just outlined.
\end{proof}
\end{lemma}
This operation also comes endowed with a ``comparison cylinder'', as explained in the following result.
\begin{lemma}
	\label{coherence of cell+cyl+cell}
In the situation of the previous lemma, there exists an $(n+1)$-cylinder $\Gamma_{\beta,C,\alpha}$ in $X$ such that $s(\Gamma_{\beta,C,\alpha})=C$ and $t(\Gamma_{\beta,C,\alpha})=\beta C \alpha$.
\begin{proof}
For sake of simplicity we drop the subscripts of $\Gamma$ in what follows. We prove this result by induction on $n$. The base case $n=0$ is straightforward by contractibility of $\mathfrak{C}$ once we set $\Gamma_0=\alpha^{-1}$ and $\Gamma_1=\beta$.
Let $n>0$, and assume the result holds for each $k<n$. By inductive hypothesis we get an $n$-cylinder $\gamma\colon \overline{C} \rightarrow( C_{t_0}\alpha) \overline{C} (\beta C_{s_0} ) $. If we analyze the source and target of $\gamma_0$ and $\gamma_1$, we see that, thanks to Lemma 4.12 in \cite{AR2}, there exist a pair of $n$-cells $E,F$ and  $(n+1)$-cells $\theta\colon C_{t_0}E\rightarrow \gamma_0,\phi\colon\gamma_1 \rightarrow F C_{s_0}$. We now define $\Gamma_{\epsilon_0}=C_{\epsilon_0}$ for $\epsilon=s,t$, and $\overline{\Gamma}=\phi \gamma \theta$, which concludes the proof.
\end{proof}
\end{lemma}
In what follows, given an $n$-cylinder $F$ in an $\infty$-groupoid $X$, we denote $\epsilon^n(F)$ by $f_{\epsilon}$ for $\epsilon=s,t$.
\begin{lemma}
	Let $\mathfrak{D}_{\mathbf{c}}$ be the globular theory freely generated by a system of compositions. Then there exists an extension of the form
	\[
	\bfig
	\morphism(0,0)|a|/@{>}@<0pt>/<500,0>[\Theta_0`\wgpd;\cyl]
	\morphism(0,0)|a|/@{>}@<0pt>/<0,-400>[\Theta_0` \mathfrak{D}_{\mathbf{c}};]
	\morphism(0,-400)|r|/@{-->}@<0pt>/<500,400>[\mathfrak{D}_{\mathbf{c}}`\wgpd;\cyl]
	\efig
	\]
	\begin{proof}
		We define this extension by induction, fixing an $\infty$-groupoid $X$ and proceeding representably. Firstly, we need to define $\cyl(\mathbf{c}_1)$. This corresponds to $\widehat{\mathbf{c}_1}$ of Definition \ref{rho hat defi}.
		Given a pair of composable $1$-clinders $C,C'$ in $X$, we denote $\p X(c_1)(C,C')$ with $C'\circ C$, and define the top cell (resp. bottom cell) of it to be $C_{0}'\circ C_0$ (resp. $C_{1}'\circ C_1$),composed using $c_1$. We then set $(C'\circ C)_{s}=s(C)$ and $(C'\circ C)_{t}=t(C')$ and declare the $2$-cell $\overline{C'\circ C}$ (i.e. a $0$-cylinder in $\Omega \left( X,s(C_0),t(C'_1)\right)$) to be the composite of
		\[
		\bfig 
\morphism(0,0)|a|/@{>}@<0pt>/<700,0>[C'_{t}(C_0'C_0)`(C'_{t}C'_0)C_0;\cong]
\morphism(700,0)|a|/@{>}@<0pt>/<700,0>[(C'_{t}C'_0)C_0`(C'_1 C'_{s})C_0;\overline{C'}C_0]
\morphism(1400,0)|a|/@{>}@<0pt>/<700,0>[(C'_1 C'_{s})C_0`C'_1 (C'_{s}C_0);\cong]
\morphism(2100,0)|a|/@{>}@<0pt>/<700,0>[C'_1 (C'_{s}C_0)`C'_1 (C_1 C_{s});C'_1 \overline{C}]
\morphism(2800,0)|a|/@{>}@<0pt>/<700,0>[C'_1 (C_1 C_{s})`(C'_1 C_1 )C_{s};\cong]
		\efig 
		\] 
where we have used the fact that $C'_{s_0}=C_{t_0}$, and we have denoted instances of associativity of composition of $1$-cells with ``$\cong$'' and the effect of composing using $c_1$ with juxtaposition.
Given $n>1$, suppose we have defined $\cyl(\mathbf{c}_k)$ for each $k<n$, and denote $(F,G)\circ \cyl(\mathbf{c}_k)$ by $G\circ_{\mathbf{c}_k}F$ for each composable pair of $k$-cylinders in an $\infty$-groupoid $X$. For every $\infty$-groupoid $X$ and every pair of $n$-cylinders $F,G\colon \cyl(D_n) \rightarrow X$ such that $t(f)=s(G)$, we define $\overline{G\circ_{c_n}F}$ to be the following composite, obtained applying Lemma \ref{cell +cyl+cell} to the following piece of data in $\Omega(X,s^n(F_0),t^n(F_1))$:
\[
		\bfig
		\morphism(0,0)|a|/@{>}@<0pt>/<0,-400>[f_t (G_0 F_0)`(f_t G_0) (f_t F_0);D_1]
		\morphism(0,-400)|a|/{@{>}@/^2em/@<0pt>}/<0,-400>[(f_t G_0) (f_t F_0)`(G_1 f_s)( F_1 f_s);\overline{G}\circ_{\mathbf{c}_{n-1}}\overline{F}]
		\morphism(0,-800)|a|/@{>}@<0pt>/<0,-400>[(G_1 f_s)( F_1 f_s)`(G_1 F_1)f_s;D_2]
		\efig 
		\]
		Here, $D_1$ is an $n$-cell obtained by contractibility of $D_n \amalg_{D_{n-1}}D_n \amalg_{ D_0}D_1$, and $G_i$ is defined similarly. 	
		This assignment defines, by the Yoneda lemma, a map \[\cyl(D_n) \rightarrow \cyl(D_n) \plus{\cyl(D_{n-1})} \cyl(D_n)\] which we take as the definition of $\cyl(c_n)$. 
		We also have the following chain of equalities, provided by the inductive hypothesis together with Lemma \ref{cell +cyl+cell}
		 \[\overline{s(G\circ_{\mathbf{c_n}}F)}=s(\overline{G\circ_{\mathbf{c_n}}F})=s(\overline{G}\circ_\mathbf{{c_{n-1}}} \overline{F})=s(\overline{F})=\overline{s(F)}\]
		and 
		\[(G\circ_{\mathbf{c_n}}F)_0=F_0\] which imply that $s(G\circ_{\mathbf{c_n}}F)=s(F)$, and a similar argument can be provided for the target.
		\end{proof}
\end{lemma}
It follows from Lemma \ref{coherence of cell+cyl+cell} that there exists an $n$-cylinder $T_{F,G}$ in $\Omega(X,s^n(F_0),t^n(G_1))$ such that $s(T_{F,G})=\overline{G \circ_{\mathbf{c}_{n}} F}$ and $t(T_{F,G})=\overline{G}\circ_{\mathbf{c}_{n-1}} \overline{F}$.

We now address the problem of definining a system of identities.
\begin{lemma}
	Let $\mathfrak{D}_{\mathbf{id}}$ be the globular theory freely generated by a system of compositions and  identities. Then there exists an extension of the form
	\[
	\bfig
	\morphism(0,0)|a|/@{>}@<0pt>/<500,0>[\Theta_0`\wgpd;\cyl]
	\morphism(0,0)|a|/@{>}@<0pt>/<0,-400>[\Theta_0` \mathfrak{D}_{\mathbf{id}};]
	\morphism(0,-400)|r|/@{-->}@<0pt>/<500,400>[\mathfrak{D}_{\mathbf{id}}`\wgpd;\cyl]
	\efig
	\]
	\begin{proof}
		We only need to define a system of identities. Firstly, set \[\cyl(\mathbf{id}_0)=\mathbf{C}_1\colon \cyl(D_1)\rightarrow D_1= \cyl(D_0)\]
		Let $n>1$, and assume we have already defined $\cyl(\mathbf{id}_k)$ for each $k<n$. Given $F\colon \cyl(D_k) \rightarrow X$, denote the $(k+1)$-cylinder $F\circ \cyl(\mathbf{id}_k)$ by $\mathbf{id}_k(F)$. 
		We have to define, for every $\infty$-groupoid $X$ and every $n$-cylinder $F\colon A \curvearrowright B$ in $X$, an $(n+1)$-cylinder $\mathbf{id}_n(F)\colon \cyl(D_{n+1})\rightarrow X$.
		Define its transpose $\overline{\mathbf{id}_n(F)}$ as the vertical composite of the following diagram
		\[
		\bfig
		\morphism(0,0)|a|/@{>}@<0pt>/<0,-400>[f_t 1_A`1_{f_t A};C_1]
		\morphism(0,-400)|a|/{@{>}@/^2em/@<0pt>}/<0,-400>[1_{f_t A}`1_{B f_s};\mathbf{id}_{n-1}(\overline{F})]
		\morphism(0,-800)|a|/@{>}@<0pt>/<0,-400>[1_{B f_s}`1_B f_s;C_2]
		\efig
		\]
		Here, juxtaposition of cells indicates, as usual, the whiskering operations $w$ defined in the previous section, $C_1$ and $C_2$ are $n$-cells provided by the contractibility of $D_n\amalg_{ D_0} D_1$ and $D_1 \amalg_{ D_0} D_n$ respectively, and the composition operation is the one defined in Lemma \ref{cell +cyl+cell}. 
		
		Having defined identities and binary compositions, we can construct whiskering maps
		\[*_k\colon\cyl(D_k) \rightarrow \cyl(D_k) \plus{\cyl(D_{k-2})}\cyl(D_{k-1})\] \[_k *\colon\cyl(D_k) \rightarrow \cyl(D_{k-1}) \plus{\cyl(D_{k-2})}\cyl(D_k)\]
		by setting \[*_k=(1\plus{\cyl(D_{k-2})} \cyl(\mathbf{id}_{k-1}))\circ \cyl(\mathbf{c}_k)\] and \[_k*=(\cyl(\mathbf{id}_{k-1}) \plus{\cyl(D_{k-2})}1 )\circ \cyl(\mathbf{c}_k)\]
		When no confusion arises, subscripts will be dropped.
		Notice that, thanks to Lemma \ref{coherence of cell+cyl+cell}, for every $k$-cylinder $F$ in $X$, there exists a $(k+1)$-cylinder $\lambda_F$ such that \[s(\lambda_F)=\overline{\mathbf{id}_k(F)} \  \text{and} \ t(\lambda_F)=\mathbf{id}_{k-1}(\overline{F})\]
		These identity cylinders satisfy the required properties thanks to Lemma \ref{cell +cyl+cell}, so we are left with defining the action of $\cyl$ on the maps $\mathbf{l}_n,\mathbf{r}_n$. We will only construct the $\mathbf{l}_n $'s, the remaining bit being similar. The construction of $\mathbf{l}_2$ reduces to definining a $2$-cylinder $\Gamma$ whose source is $C\circ_{\mathbf{c}_{1}}\mathbf{id}_0(s(C))$ and whose target is $C$. We use contractibility of $\mathfrak{C}$ once to find a pair of $2$-cells $\Gamma_0\colon C_0 1_{s(C_0)}  \rightarrow C_0, \ \Gamma_1\colon C_1 1_{s(C_1)}\rightarrow C_1$, and then again to choose a $3$-cell between the following composites
		\[
		\bfig 
	\morphism(0,0)|a|/@{>}@<0pt>/<700,0>[C_t(C_0 1_{s(C_0)})`C_tC_0;C_t\Gamma_0]
\morphism(700,0)|a|/@{>}@<0pt>/<700,0>[C_tC_0`C_1 C_s;C]
	\morphism(0,-500)|a|/@{>}@<0pt>/<1200,0>[C_t(C_0 1_{s(C_0)})`(C_1 1_{s(C_1)})C_s;C\circ_{\mathbf{c}_{1}}\mathbf{id}_0(s(C))]	
	\morphism(1200,-500)|a|/@{>}@<0pt>/<700,0>[(C_1 1_{s(C_1)})C_s`C_1 C_s;\Gamma_1 C_s]	
	\efig 	
		\]
		Given an $(n-1)$-cylinder $F\colon A\curvearrowright B$ in $X$, we let $\alpha\colon A \circ 1_{s(A)}\rightarrow A$ be an instance of unitality of composition in $\mathfrak{C}$, and by contractibility of $D_{n-1}\amalg_{ D_0} D_1 $ we get an $n$-cell $E_1$ whose source is $f_t \alpha$ and whose target is \[(\mathbf{l}_{n-2}(\overline{F})\circ_{\mathbf{c}_{n-1}}(\overline{F} * \lambda_{s(F)})\circ_{\mathbf{c}_{n-1}}T_{1_{s(F),F}})_0\] which coincides with the composite
\[
\bfig 
\morphism(0,0)|a|/@{>}@<0pt>/<700,0>[f_t\left( A1_{s(A)}\right)`\left( f_t A\right)\left(f_t 1_{s(A)}\right);\cong]
\morphism(700,0)|a|/@{>}@<0pt>/<700,0>[\left(f_t A\right)\left(f_t 1_{s(A)}\right)`\left(f_t A\right)\left(1_{f_t A}\right);\cong]
\morphism(1400,0)|a|/@{>}@<0pt>/<500,0>[\left(f_t A\right)\left( 1_{f_t A}\right)`f_t A;\cong]
\efig 
\] where each one of the displayed cells is obtained by contractibility of $\mathfrak{C}$.
		Now, we define $\overline{\mathbf{l}_{n-1}(F)}$ to be the following composite
		\[
		\bfig
		\morphism(0,0)|a|/@{>}@<0pt>/<0,-400>[f_t \alpha`(\mathbf{l}_{n-2}(\overline{F})\circ_{\mathbf{c}_{n-1}}(\overline{F} * \lambda_{s(F)})\circ_{\mathbf{c}_{n-1}}T_{1_{s(F),F}})_0;E_1]
		\morphism(0,-400)|a|/{@{>}@/^2em/@<0pt>}/<0,-400>[(\mathbf{l}_{n-2}(\overline{F})\circ_{\mathbf{c}_{n-1}}(\overline{F} * \lambda_{s(F)})\circ_{\mathbf{c}_{n-1}}T_{1_{s(F),F}})_0`(\mathbf{l}_{n-2}(\overline{F})\circ_{\mathbf{c}_{n-1}}(\overline{F} * \lambda_{s(F)})\circ_{\mathbf{c}_{n-1}}T_{1_{s(F),F}})_1;(\mathbf{l}_{n-2}(\overline{F})\circ_{\mathbf{c}_{n-1}}(\overline{F} * \lambda_{s(F)})\circ_{\mathbf{c}_{n-1}}T_{1_{s(F),F}})]
		\morphism(0,-800)|a|/@{>}@<0pt>/<0,-400>[(\mathbf{l}_{n-2}(\overline{F})\circ_{\mathbf{c}_{n-1}}(\overline{F} * \lambda_{s(F)})\circ_{\mathbf{c}_{n-1}}T_{1_{s(F),F}})_ 1`\alpha' f_s;E_2]
		\efig
		\]
		Here, $\alpha'\colon B \circ 1_{s(B)} \rightarrow B$ is another instance of unitality of composition in $\mathfrak{C}$.
		$E_2$ is obtained similarly to $E_1$, and we compose the diagram using Lemma \ref{cell +cyl+cell}. 
	\end{proof}
\end{lemma}
We are now ready to conclude the proof of Theorem \ref{algebraic structure on PX}:
\begin{proof}
	The only thing left to define is a system of inverses on $\mathbb{P}(X)$, and again we do so by induction.
	Given an $\infty$-groupoid $X$ and a $1$-cylinder $F\colon A \curvearrowright B$ in $X$, we define $\overline{\mathbf{i}_1(F)}$ as the composite
	\[
	\bfig
	\morphism(0,0)|a|/@{>}@<0pt>/<700,0>[f_s A^{-1}`B^{-1}B f_s A^{-1};]
	\morphism(700,0)|a|/@{>}@<0pt>/<1200,0>[B^{-1}B f_s A^{-1}`B^{-1} f_t A A^{-1};B^{-1}(\overline{F})^{-1}A^{-1}]
	\morphism(1900,0)|a|/@{>}@<0pt>/<700,0>[B^{-1} f_t A A^{-1}`B^{-1}f_t;]
	\efig 
	\]
	where the unlabelled cells are obtained by contractibility of $\mathfrak{C}$ and $()^{-1}$ is the action of taking the inverse of a given cell, made possible by the choice of an inverse operation in the contractible globular theory $\mathfrak{C}$.
	
As before, we can construct $\mathbf{k}_2^s,\mathbf{k}_2^t$ using the contractibility of $\mathfrak{C}$, which concludes the proof of the base case.

To address the inductive step, assume given an $n$-cylinder $F\colon A \curvearrowright B$ in $X$ with $n>1$. We define $\overline{\mathbf{i}_n(F)}$ as the  composite of the following diagram, obtained using Lemma \ref{cell +cyl+cell}
	\[
	\bfig
	\morphism(0,0)|a|/@{>}@<0pt>/<0,-400>[f_t A^{-1}`(f_t A)^{-1};M_1]
	\morphism(0,-400)|a|/{@{>}@/^2em/@<0pt>}/<0,-400>[(f_t A)^{-1}`(B f_s)^{-1};\mathbf{i}_{n-1}(\overline{F})]
	\morphism(0,-800)|a|/@{>}@<0pt>/<0,-400>[(B f_s)^{-1}`B^{-1} f_s;M_2]
	\efig 
	\]
	Here, $M_1$ and $M_2$ are $n$-cells obtained by contractibility of $\mathfrak{C}$.
	Again, observe that it follows from Lemma \ref{coherence of cell+cyl+cell} that there exists a cylinder $\mu_F$ whose source is $\overline{\mathbf{i}_n(F)}$ and whose target is $\mathbf{i}_{n-1}(\overline{F})$.
	
	We are now left with constructing $\mathbf{k}^{\epsilon}_{n+1}$ for $\epsilon=s,t$. The two cases being similar, we only construct $\mathbf{k}^s_{n+1}$.
	Let $\beta\colon AA^{-1} \rightarrow 1_{t(A)}$ be an instance of a coherence constraint for inverses, provided by contractibility of $\mathfrak{C}$. The latter also provides a cell $H_1$, whose source is $f_t \beta$ and whose target is given by
	\[((\mathbf{i}_{n}(\lambda_F)) \circ_{\mathbf{c}_{n}}(\mathbf{k}^s_{n-1}(\overline{F}))\circ_{\mathbf{c}_{n}} (\mu_F*\overline{F}) \circ_{\mathbf{c}_{n}}(T_{F^{-1},F}))_0\] which coincides with the following composite
	\[
	\bfig 
\morphism(0,0)|a|/@{>}@<0pt>/<700,0>[f_t (AA^{-1})`(f_t A)(f_t A^{-1});\cong]	
\morphism(700,0)|a|/@{>}@<0pt>/<700,0>[(f_t A)(f_t A^{-1})`(f_t A)(f_t A)^{-1};\cong]	
\morphism(1400,0)|a|/@{>}@<0pt>/<600,0>[(f_t A)(f_t A)^{-1}`1_{s(f_t A)};\cong]
\morphism(2000,0)|a|/@{>}@<0pt>/<500,0>[1_{s(f_t A)}`f_t 1_{s( A)};\cong]
	\efig 
	\]
	Finally, we define $\mathbf{l}^s_{n+1}(F)$ as the composite of the following diagram, using Lemma \ref{cell +cyl+cell}
	\[
	\bfig 
	\morphism(0,0)|a|/@{>}@<0pt>/<0,-400>[f_t \beta`((\mathbf{i}_{n}(\lambda_F)) \circ_{\mathbf{c}_{n}}(\mathbf{k}^s_{n-1}(\overline{F}))\circ_{\mathbf{c}_{n}} (\mu_F*\overline{F}) \circ_{\mathbf{c}_{n}}(T_{F^{-1},F}))_0;H_1]
	\morphism(0,-400)|a|/{@{>}@/^2em/@<0pt>}/<0,-400>[((\mathbf{i}_{n}(\lambda_F)) \circ_{\mathbf{c}_{n}}(\mathbf{k}^s_{n-1}(\overline{F}))\circ_{\mathbf{c}_{n}} (\mu_F*\overline{F}) \circ_{\mathbf{c}_{n}}(T_{F^{-1},F}))_0`((\mathbf{i}_{n}(\lambda_F)) \circ_{\mathbf{c}_{n}}(\mathbf{k}^s_{n-1}(\overline{F}))\circ_{\mathbf{c}_{n}} (\mu_F*\overline{F}) \circ_{\mathbf{c}_{n}}(T_{F^{-1},F}))_1;((\mathbf{i}_{n}(\lambda_F)) \circ_{\mathbf{c}_{n}}(\mathbf{k}^s_{n-1}(\overline{F}))\circ_{\mathbf{c}_{n}} (\mu_F*\overline{F}) \circ_{\mathbf{c}_{n}}(T_{F^{-1},F}))]
	
	\morphism(0,-800)|a|/@{>}@<0pt>/<0,-400>[((\mathbf{i}_{n}(\lambda_F)) \circ_{\mathbf{c}_{n}}(\mathbf{k}^s_{n-1}(\overline{F}))\circ_{\mathbf{c}_{n}} (\mu_F*\overline{F}) \circ_{\mathbf{c}_{n}}(T_{F^{-1},F}))_1`\gamma f_s;H_2]
	\efig 
	\]
	Here, $\gamma\colon BB^{-1}\rightarrow 1_{t(B)}$ is an instance of a coherence constraint for inverses, and $H_2$ is obtained analogously to $H_1$, both being provided by the contractibility of $\mathfrak{C}$.
	\end{proof}
\section{Globular decomposition of Cylinders on globular sums}
Now that we have constructed the coglobular object $\cyl(D_{\bullet})\colon \G \to \wgpd$, we can extend it to a functor $\cyl(D_{\bullet})\colon \Theta_0 \to \wgpd$, since $\wgpd$ is cocomplete (see, for instance, Corollary 5.6.8 in \cite{BOR}).

The goal of this section is, given a globular sum $A \in \Theta_0$, to express $\cyl(A)$ as the colimit of a zig-zag diagram in $\wgpd$, which will be explicitly described. Furthermore, this diagram only consists of globular sums: more precisely, it factors through the Yoneda embedding $y\colon \mathfrak{C} \rightarrow \wgpd$.
\subsection{Zig-zag diagrams}
To begin with, we have to define what a zig-zag is, and record their basic properties.
\begin{defi} \label{zig-zag length n}
	Given a natural number $n$, define a category $\mathbf{I}_n$ as the one associated to the poset $(\{(0,k)\colon 0 \leq k \leq n\}\cup \{(1,m)\colon 0 \leq m \leq n-1\},\prec)$, where the relation is completely described by 
	$$\begin{cases}
	(0,k)\prec (1,k)  \ \forall k \in \{0, \ldots, n-1\} \\
	(0,m) \prec (1,m-1) \ \forall  m \in \{1,\ldots,n\}
	\end{cases} $$
	Notice that, if $k<n$, there is a natural inclusion $\mathbf{I}_k \rightarrow \mathbf{I}_n$.
\end{defi}
Pictorially, $\mathbf{I}_n$ looks like 
\begin{equation}
\label{zig zag picture}
\bfig

\morphism(0,0)|a|/@{>}@<0pt>/<250,-250>[(0,0)`(1,0); ]
\morphism(500,0)|a|/@{>}@<0pt>/<-250,-250>[(0,1)`(1,0);]
\morphism(500,0)|a|/@{>}@<0pt>/<175,-175>[(0,1)`;]
\morphism(1500,0)|a|/@{>}@<0pt>/<250,-250>[(0,n-1)`(1,n-1); ]
\morphism(1500,0)|a|/@{>}@<0pt>/<-175,-175>[(0,n-1)`; ]
\morphism(2000,0)|a|/@{>}@<0pt>/<-250,-250>[(0,n)`(1,n-1); ]

\morphism(1000,0)|a|/@{}@<0pt>/<0,0>[\ldots` ; ]
\morphism(1000,-250)|a|/@{}@<0pt>/<0,0>[\ldots`; ]

\efig
\end{equation}
We have two natural inclusions
$\bfig \morphism(0,0)|a|/@{>}@<3pt>/<500,0>[*`\mathbf{I}_n;(0,0)=\iota_0 ]
\morphism(0,0)|b|/@{>}@<-3pt>/<500,0>[ *`\mathbf{I}_n; (0,n)=\iota_n]
\efig$ for any positive natural number $n$, where $*$ denotes the terminal category.

$\mathbf{I}_n$ is the free-living zig-zag of length $n$, in a sense made precise by the following
\begin{defi}
	A zig-zag of length $n$ in a category $\mathscr{C}$ is a functor 
	$\bfig  \morphism(0,0)|a|/@{>}@<0pt>/<500,0>[\mathbf{I}_n`\mathscr{C};F]\efig$.
	If $F(0,0)=a$ and $F(0,n)=b$ we write $F \colon a \rightsquigarrow b$. 
\end{defi}
We can also define a partial binary operation on zig-zags, which satisfies an associativity property and will be used in the next section.
\begin{defi}
	\label{extend zig-zag}
	Define the category $\mathbf{I}_{m}\bullet \mathbf{I}_{n}$ as the pushout
	\[
	\bfig
	\morphism(0,0)|a|/@{>}@<0pt>/<500,0>[*`\mathbf{I}_n;\iota_n]
	\morphism(0,0)|a|/@{>}@<0pt>/<0,-500>[*`\mathbf{I}_m;\iota_0]
	\morphism(500,0)|a|/@{>}@<0pt>/<0,-500>[\mathbf{I}_n`\mathbf{I}_{m}\bullet \mathbf{I}_{n};]
	\morphism(0,-500)|a|/@{>}@<0pt>/<500,0>[\mathbf{I}_m`\mathbf{I}_{m}\bullet \mathbf{I}_{n}; ]
	\efig
	\] Note that $\mathbf{I}_{m}\bullet \mathbf{I}_{n} \cong \mathbf{I}_{m+n}$.
	
	Given a pair of zig-zags $F \colon \mathbf{I}_n \rightarrow \mathcal{C}$, $G \colon \mathbf{I}_m \rightarrow \mathcal{C}$ such that $F\colon a \rightsquigarrow b$ and $G \colon b \rightsquigarrow c$ we define \[G\bullet F \colon \mathbf{I}_{m}\bullet \mathbf{I}_{n}\rightarrow \mathcal{C}\] as the unique functor making the following diagram commute
	\begin{equation}
	\label{composition of zig-zags}
	\bfig
	\morphism(0,0)|a|/@{>}@<0pt>/<500,0>[*`\mathbf{I}_n;\iota_n]
	\morphism(0,0)|a|/@{>}@<0pt>/<0,-500>[*`\mathbf{I}_m;\iota_0]
	\morphism(500,0)|a|/@{>}@<0pt>/<0,-500>[\mathbf{I}_n`\mathbf{I}_{m}\bullet \mathbf{I}_{n};]
	\morphism(0,-500)|a|/@{>}@<0pt>/<500,0>[\mathbf{I}_m`\mathbf{I}_{m}\bullet \mathbf{I}_{n}; ]
	
	\morphism(500,-500)|l|/@{-->}@<0pt>/<300,-300>[\mathbf{I}_{m}\bullet \mathbf{I}_{n}`\mathcal{C};\exists ! G \bullet F]
	
	\morphism(500,0)|a|/{@{>}@/^2em/}/<300,-800>[\mathbf{I}_{n}`\mathcal{C};F]
	\morphism(0,-500)|l|/{@{>}@/_2em/}/<800,-300>[\mathbf{I}_m`\mathcal{C};G ]
	\efig
	\end{equation}
	Note that $G \bullet F\colon a \rightsquigarrow c$.
\end{defi}
	Obviously, this can be iterated to express $\mathbf{I}_{n_1+n_2+\ldots n_k}$ as an iterated pushout $\mathbf{I}_{n_k}\bullet \ldots\bullet \mathbf{I}_{n_1} $.
\begin{lemma}
	\label{associativity of concatenation}
	Concatenation of zig-zags is associative. More precisely, if we are given $F\colon a \rightsquigarrow b$, $G\colon b \rightsquigarrow c$ and $H\colon c \rightsquigarrow d$ then it holds true that \[H \bullet (G \bullet F)=(H\bullet G)\bullet F\]
\end{lemma}
\begin{defi}
	Let $\mathcal{C}$ be a cocomplete category. Suppose given a zig-zag $F\colon\mathbf{I}_{n} \rightarrow \mathcal{C}$, we define a zig-zag of length $1$ $\tilde{F}\colon \mathbf{I}_1 \rightarrow \mathcal{C}$ by setting $\tilde{F}(0,0)=F(0,0)$, $\tilde{F}(0,1)=F(0,n)$ and $\tilde{F}(1,0)=\colim_{\mathbf{I}_n} F$, where the structural maps are given by the colimit inclusions.
\end{defi}
Given integers $n_i$ for $1\leq i \leq k$, we let $n=\sum_{1}^{k}n_i$. Given a zig-zag $F\colon \mathbf{I}_{n}\cong \mathbf{I}_{n_k}\bullet \ldots\bullet \mathbf{I}_{n_1} \rightarrow \mathcal{C}$, where the target is a cocomplete category, we can consider its restrictions $F_i \colon \mathbf{I}_{n_i} \rightarrow \mathcal{C}$, obtained as in Definition \ref{extend zig-zag}. The next result then holds true, and its proof is simply a matter of applying the universal property of colimits.
\begin{lemma} 
	\label{colimit of zig zags}
	In the situation just described, we have the following isomorphism in $\mathcal{C}$:
	\[
	\colim_{\mathbf{I}_{n}} F \cong \colim_{\mathbf{I}_{k}}\left(\tilde{F}_k\bullet \ldots \bullet \tilde{F}_1\right)
	\]
\end{lemma}
\subsection{Trees and globular sums}
We now need an alternative way of representing globular sums. In \cite{AR1} this is done by associating to any given globular sum $A$ a finite planar tree that uniquely represents it.
\begin{defi}
	Consider the functor category $\text{Ord}_{fin}^{\omega}$, where $\omega$ is viewed as a poset with respect to inclusion, and $ \text{Ord}_{fin}$ is the category of finite linearly ordered sets.
	
	The category $\mathcal{T}$ of finite planar trees is the full subcategory of $\text{Ord}_{fin}^{\omega}$ spanned by the objects $T$ such that $T(0)$ is the terminal object of $\text{Ord}_{fin}$ (i.e. the singleton with its unique ordering) and there exists an $n\in \mathbb{N}$ such that $T(i)=\emptyset$ for each $i\geq n$.
	
	We call the elements of $v(T)=\bigcup_{k \in \omega}T_k$ vertices of $T$, and we say that a vertex $x$ has height $m$, denoted by $\height(x)=m$, if $x\in T_m$. Finally, we set $\height(T)=\max_{x \in v(T)}\height(x)$.
\end{defi}
Explicitly, a finite planar tree $T$ consists of a family of finite linearly ordered sets $(T_i, \leq_i)_{0\leq i \leq n}$ for some $n\in \mathbb{N}$, with $T(0)=\{*\}$, together with order-preserving maps $\iota^T_i\colon T_{i+1} \rightarrow T_i$.
\begin{ex}
	Consider the finite planar tree $T$ given by $T_1=\{ x_1^1< x_2^1\}, \ T_2=\{ x_1^2\}$ and $ T_3=\{ x_1^3< x_2^3< x_3^3 \}$, whose only non trivial structural map $T_2 \rightarrow T_1$ is given by $x_1^2 \mapsto x_2^1$.
	
	Such a tree $T$ can be depicted as
	
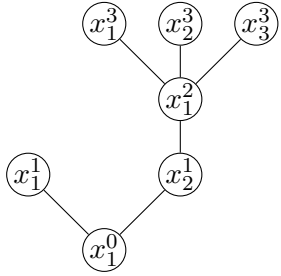
\[
	\begin{tikzpicture}
	\node[circle, draw,
	inner sep=0pt, minimum width=2pt] (A) at (0,0) {$x_1^0$};
	
	\node[circle, draw,
	inner sep=0pt, minimum width=2pt] (B) at (-1,1) {$x_1^1$};
	\node[circle, draw,
	inner sep=0pt, minimum width=2pt] (C) at (1,1) {$x_2^1$};
	
	\node[circle, draw,
	inner sep=0pt, minimum width=2pt] (D) at (1,2) {$x_1^2$};
	
	\node[circle, draw,
	inner sep=0pt, minimum width=2pt] (E) at (0,3) {$x_1^3$};
	\node[circle, draw,
	inner sep=0pt, minimum width=2pt] (F) at (1,3) {$x_2^3$};
	\node[circle, draw,
	inner sep=0pt, minimum width=2pt] (G) at (2,3) {$x_3^3$};
	
	\path [] (A) edge (B);
	\path [] (A) edge (C);
	
	\path [] (C) edge (D);
	
	\path [] (D) edge (E);
	\path [] (D) edge (F);
	\path [] (D) edge (G);
	\end{tikzpicture}
	\]
\end{ex}
\begin{defi}
Given a tree $T$, we can define a relation on the set of vertices $v(T)$ as follows. Consider $x\neq y \in v(T)$, and set
\begin{equation}
\label{prec def}
x\prec y \iff \begin{cases}
\height (x)=\height(y) \ \text{and} \ y<x \ \text{in} \ T_{\height(x)}\\
\height (x)<\height(y) \ \text{and} \ \iota_T^k(y)\leq x \ \text{in} \ T_{\height(x)}\\
\height (x)>\height(y) \ \text{and} \ y < \iota^{k'}_T x \ \text{in} \ T_{\height(y)}
\end{cases}
\end{equation}
where $k=\height(y)-\height(x)$ and $k'=\height(x)-\height(y)$, and $\iota_T^k:T_{\height(y)}\rightarrow T_{\height(x)}, \iota_T^{k'}:T_{\height(x)}\rightarrow T_{\height(y)}$ are composite of the structural maps of $T$.
Clearly, this defines a linear order on $v(T)$.
\end{defi}
For instance, given the tree $T$ of the previous example, the totally ordered set of its vertices is given by 
\[\{x_1^0 \prec x_2^1 \prec x_1^2 \prec x_3^3 \prec x_2^3 \prec x_1^3 \prec x_1^1\}\]
We can associate a tree with every given globular sum $A$, to do so we need the following definition.
\begin{defi}
A given a tree $T$ and a vertex $x\in T$, we say that $x$ is maximal (also called a leaf) if \[\left( \iota^T_{\height(x)}\right)^{-1}(x)=\emptyset\]
Let $\{x_1\prec x_2 \prec \ldots \prec x_k\}$ be the ordered set of maximal vertices of $T$. Let $h_i$ be the height of the highest vertex $y$ such that both $x_i$ and $x_{i+1}$ belong to the fiber of (an iteration of) $\iota^T$ over $y$. $h_i$ is called the height of the region between $x_i$ and $x_{i+1}$.
\end{defi}
It is an easy exercise to prove the following result.
\begin{lemma}
	\label{max vert + reg=tree}
The ordered set $\{x_1\prec x_2 \prec \ldots \prec x_k\}$ of maximal vertices of a given tree $T$ together with their respective heights and the set $\{h_1,\ldots, h_{k-1}\}$ of heights of the regions between them uniquely determine the tree $T$.
\end{lemma}
We now explain how to associate a tree $T^A$ with a given globular sum $A$. Suppose $A$ is defined by the table of dimensions
\[\begin{pmatrix}
i_1 &&i_2 & \ldots&i_{m-1} & &i_m\\
& i'_1 & &\ldots&& i'_{m-1}
\end{pmatrix}\]
We define $T^A$, invoking the previous lemma, by specifying the order $\{x_1\prec x_2 \prec \ldots \prec x_k\}$ of its maximal vertices and the height $\{h_1,\ldots, h_{k-1}\}$ of the regions between them. We let $k=m$ and we impose that $x_i$ has height $i_{m-i+1}$ and $h_i=i'_{m-i}$.

For example, the globular sum $\begin{pmatrix}
2 &&2 && 1 &&2\\
& 1 &&0&& 0
\end{pmatrix}$
can be represented, equivalently, as
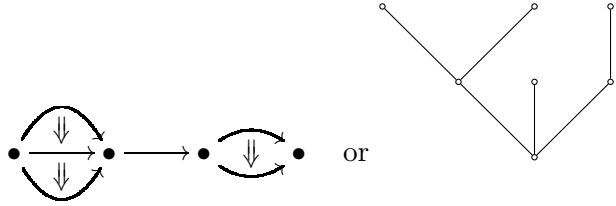
\[
\xymatrix{\bullet \ruppertwocell
	\rlowertwocell
	\ar[r]^-{}
	& \bullet \ar[r]^-{} &\bullet \rtwocell &\bullet\\} \ \ \  \text{or} \ 
\begin{tikzpicture}
\node[circle, draw,
inner sep=0pt, minimum width=2pt] (A) at (0,-3) {};

\node[circle, draw,
inner sep=0pt, minimum width=2pt] (B) at (-1,-2) {};
\node[circle, draw,
inner sep=0pt, minimum width=2pt] (C) at (1,-2) {};
\node[circle, draw,
inner sep=0pt, minimum width=2pt] (F) at (0,-2) {};

\node[circle, draw,
inner sep=0pt, minimum width=2pt] (D) at (-2,-1) {};
\node[circle, draw,
inner sep=0pt, minimum width=2pt] (E) at (0,-1) {};

\node[circle, draw,
inner sep=0pt, minimum width=2pt] (G) at (1,-1) {};

\path [] (A) edge (B);
\path [] (A) edge (C);

\path [] (A) edge (F);

\path [] (B) edge (D);
\path [] (B) edge (E);
\path [] (C) edge (G);
\end{tikzpicture}
\]
\begin{ex}
	\label{suspension of trees}
Given a globular sum $A\in \Theta_0$, we have defined its suspension $\Sigma A$ in Section 2. It is very easy to define the tree $T^{\Sigma A}$ associated with $\Sigma A$ in terms of $T^A$. In fact, suppose $T^A$ consists of the family of finite linearly ordered sets $(T_i,\leq_i)_{0\leq i \leq n}$ for some natural number $n$. Then we have $T^{\Sigma A}_1=\{*\}$ and for every $k>1$:
\[ T^{\Sigma A}_k=T^A_{k-1}\]
Moreover, $\iota^{T^{\Sigma A}}_k=\iota^{T^{A}}_{k-1}$.
\begin{rmk}
	\label{decomposition of trees}
	The decomposition given in Lemma \ref{decomposition of glob sum} has a more geometric interpretation in the language of trees. It corresponds to the elementary fact that any tree can be realized as the glueing at the root of a family of trees all having a single edge at the bottom.
\end{rmk}
For instance, if we let $A$ be the globular sum whose table of dimensions is \[\begin{pmatrix}
2 &&2 && 1 &&2\\
& 1 &&0&& 0
\end{pmatrix}\] then $\Sigma A$ has table of dimensions given by 
\[\begin{pmatrix}
3 &&3 && 2 &&3\\
& 2 &&1&& 1
\end{pmatrix}\]
Moreover, the tree $T^{\Sigma A}$ can be depicted as 
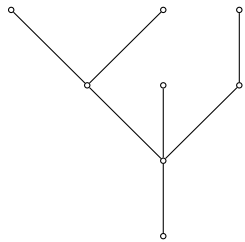
\[\begin{tikzpicture}
\node[circle, draw,
inner sep=0pt, minimum width=2pt] (H) at (0,-4) {};
\node[circle, draw,
inner sep=0pt, minimum width=2pt] (A) at (0,-3) {};

\node[circle, draw,
inner sep=0pt, minimum width=2pt] (B) at (-1,-2) {};
\node[circle, draw,
inner sep=0pt, minimum width=2pt] (C) at (1,-2) {};
\node[circle, draw,
inner sep=0pt, minimum width=2pt] (F) at (0,-2) {};

\node[circle, draw,
inner sep=0pt, minimum width=2pt] (D) at (-2,-1) {};
\node[circle, draw,
inner sep=0pt, minimum width=2pt] (E) at (0,-1) {};

\node[circle, draw,
inner sep=0pt, minimum width=2pt] (G) at (1,-1) {};

\path [] (A) edge (B);
\path [] (A) edge (C);

\path [] (A) edge (F);

\path [] (B) edge (D);
\path [] (B) edge (E);
\path [] (C) edge (G);
\path [] (H) edge (A);
\end{tikzpicture}
\]
\end{ex}
It turns out that the cylinder on a given globular sum $A$ has a quite simple description in terms of trees. In fact, it is the colimit of a suitable zig-zag diagram $\Cyl(A)\colon \mathbf{I}_{n_A} \rightarrow \wgpd$, for an integer $n_A$ that will be defined in what follows. More precisely, this diagram is the composite of a diagram $\mathbf{I}_{n_A} \rightarrow \mathfrak{C}$ followed by the Yoneda embedding.

To begin with, we want to list the globular sums $\{\Cyl(A)(1,k)\}_{0\leq k \leq n_A -1}$, i.e. those appearing on the bottom row (see \eqref{zig zag picture}) of the zig-zag diagram associated with  $\cyl(A)$. These are obtained by considering $n_A-1$ copies of the tree associated with $A$, and to each of these we add a single new branch, following the procedure we now describe.
We start by sticking it at the bottom right and then we traverse the tree going upward and to the left, counterclockwise.

So for example this is what one gets for $A=D_2 \amalg_{D_0} D_1$, whose associated tree is 
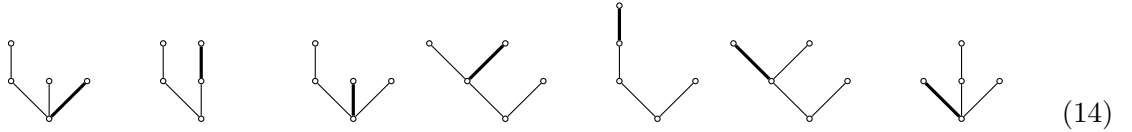
\begin{tikzpicture}
\tikzstyle{every node}=[circle, draw,
inner sep=0pt, minimum width=2pt]
\node[] (A) at (0,-3) {};

\node[] (B) at (-0.5,-2.5) {};
\node[] (C) at (0,-2.5) {};

\node[] (E) at (-0.5,-2) {};

\path [] (A) edge (B);
\path [] (A) edge (C);

\path [] (B) edge (E);

\end{tikzpicture}:

\begin{equation}
\begin{tikzpicture}
\label{list of trees}

\tikzstyle{every node}=[circle, draw,
inner sep=0pt, minimum width=2pt]
\node[] (A) at (0,-3) {};

\node[] (B) at (-0.5,-2.5) {};
\node[] (C) at (0.5,-2.5) {};
\node[] (D) at (0,-2.5) {};

\node[] (E) at (-0.5,-2) {};

\path [] (A) edge (B);
\path [draw=red, very thick] (A) edge (C);
\path [] (A) edge (D);
\path [] (B) edge (E);

\node[] (A') at (2,-3) {};

\node[] (B') at (1.5,-2.5) {};
\node[] (C') at (2,-2) {};
\node[] (D') at (2,-2.5) {};

\node[] (E') at (1.5,-2) {};

\path [] (A') edge (B');
\path [draw=red, very thick] (D') edge (C');
\path [] (A') edge (D');
\path [] (B') edge (E');

\node[] (A'') at (4,-3) {};

\node[] (B'') at (3.5,-2.5) {};
\node[] (C'') at (4,-2.5) {};
\node[] (D'') at (4.5,-2.5) {};

\node[] (E'') at (3.5,-2) {};

\path [] (A'') edge (B'');
\path [draw=red, very thick] (A'') edge (C'');
\path [] (A'') edge (D'');
\path [] (B'') edge (E'');

\node[] (a) at (6,-3) {};

\node[] (b) at (5.5,-2.5) {};
\node[] (c) at (6,-2) {};
\node[] (d) at (6.5,-2.5) {};

\node[] (e) at (5,-2) {};

\path [] (a) edge (b);
\path [draw=red, very thick] (b) edge (c);
\path [] (a) edge (d);
\path [] (b) edge (e);

\node[] (a'') at (8,-3) {};

\node[] (b'') at (7.5,-2.5) {};
\node[] (c'') at (7.5,-2) {};
\node[] (d'') at (8.5,-2.5) {};

\node[] (e'') at (7.5,-1.5) {};

\path [] (a'') edge (b'');
\path [draw=red, very thick] (c'') edge (e'');
\path [] (a'') edge (d'');
\path [] (b'') edge (c'');

\node[] (a') at (10,-3) {};

\node[] (b') at (9.5,-2.5) {};
\node[] (c') at (10,-2) {};
\node[] (d') at (10.5,-2.5) {};

\node[] (e') at (9,-2) {};

\path [] (a') edge (b');
\path [draw=red, very thick] (b') edge (e');
\path [] (a') edge (d');
\path [] (b') edge (c');

\node[] (a''') at (12,-3) {};

\node[] (b''') at (12,-2.5) {};
\node[] (c''') at (12,-2) {};
\node[] (d''') at (12.5,-2.5) {};

\node[] (e''') at (11.5,-2.5) {};

\path [] (a''') edge (b''');
\path [draw=red, very thick] (a''') edge (e''');
\path [] (a''') edge (d''');
\path [] (b''') edge (c''');
\end{tikzpicture}
\end{equation}
On the other hand, the upper row is constant on $A$, i.e. \[\Cyl(A)(0,k)=A \ \forall  \ 0\leq k \leq n_A \]
Let us now formalize what we have said so far.
\begin{defi}
	\label{l(A) defi}
	Given a tree $T$ of height $n$, seen as a family of linearly ordered sets $(T_i)_{0\leq i \leq n}$ together with compatible maps $\iota_k:T_{k+1} \rightarrow T_k$, we define a set of trees $\mathscr{L}(A)$ by considering all the trees obtained from $A$ by adjoining a single edge.
	
	Formally, this means that we consider all possible trees $B$ such that there exists a unique $1 \leq k \leq n+1$ such that $B_k=A_k \cup \{*_B\}$ and $B_i=A_i$ for each $i\neq k$, in such a way that the obvious map $A \rightarrow B$ is a map of trees.
\end{defi}
Note that, by construction, for every $B$ in $\mathscr{L}(A)$ there is a natural inclusion of trees $\chi^A_B\colon A\rightarrow B$.
\begin{lemma}
	Given $B,C \in \mathscr{L}(A)$, the presheaf of sets $B\cup_A C$ inherits the structure of a tree in such a way that the natural inclusions $B\rightarrow B\cup_A C$ and $C\rightarrow B\cup_A C$ are morphisms of such.
	\begin{proof}
		We have to endow each set $(B\cup_A C)_k$ with a linear order, in a compatible way. If the newly added vertices $*_B,*_C$ appear at different heights then this is straightforward: it is inherited from $B$ and $C$.
		If they appear at the same height there exists an $x\in A_{\height(*_B)}$ such that $*_B<_B x <_C *_C$ or $*_C<_C x <_B *_Cb$. In the former case we set $*_B <_{B \cup_A C} *_C$ , in the latter $*_C <_{B\cup_A C} *_B$.
		The rest of the proof is straightforward.
	\end{proof}
\end{lemma}
We now define a relation on the set $\mathscr{L}(A)$, introduced in \ref{l(A) defi}.
\begin{defi}
	
	Given $B\neq C$ in $\mathscr{L}(A)$, set $B\lessdot C$ if and only if $*_B \prec_{B \cup_A C}*_C$, where $\prec$ is the relation defined in \eqref{prec def}.
\end{defi}
\begin{lemma}
	Given a globular sum $A$, the relation on $\mathscr{L}(A)$ just defined is a linear order.
	\begin{proof}
		The only non-trivial thing to check is transitivity. If $B \lessdot C \lessdot D$ then, since $(B\cup C)\cup D=B\cup(C\cup D)$, one has that $*_B<_{B \cup C} *_C <_{C\cup D}*_D$ implies $*_B<_{B \cup C \cup D}*_D$, which in turn implies $*_B <_{B \cup D} *_D$.	
	\end{proof}
\end{lemma}
\begin{lemma}
	\label{tree decomposition with vertices}
Given a tree $T$ and vertices $v_1,\ldots,v_n$ of $T$ such that every leaf of $T$ belongs to the tree $C_i$ above $v_i$ for some $i$, and the $C_i$'s are disjoint, we have the following isomorphism
\[T\cong \Sigma^{m_1}C_1 \plus{D_{h_1}}\Sigma^{m_2}C_2\plus{D_{h_2}}\ldots \plus{D_{h_{n-1}}}\Sigma^{m_n}C_n\] where $m_i=\height(v_i)$, $h_i$ is the height of the region between $v_i$ an $v_{i+1}$ and the maps are the unique maps in $\Theta_0$ that there are between those objects.
\begin{proof}
We argue by induction on the total number $m$ of vertices of the tree $T$ . Assume $v_i\neq T(0)$, i.e. the root, in which case there is nothing to prove. Decompose $T$ as $\Sigma T_1\plus{D_0}\ldots\plus{D_0}\Sigma T_k$, as in Remark \ref{decomposition of trees}. If $k>1$ then we can divide and reorder (if needed) the set of vertices  $\{v_i\}_{1\leq i \leq n}$ in $\{v_j\}_{1\leq j \leq r}\cup \{v_q\}_{r+1\leq q \leq n}$ so that the elements of $\{v_q\}_{r+1\leq q \leq n}$ are precisely those $v_i$'s that belong to $\Sigma T_k$, and thus automatically satisfy the assumption of the lemma for such tree. Therefore, because $h_r=0$ by construction, the statement about the decomposition of the tree $T$ holds true since we can apply the inductive hypothesis to the trees $\Sigma T_1\plus{D_0}\ldots\plus{D_0}\Sigma T_{k-1}$ and $\Sigma T_k$, which have strictly less than $m$ vertices.

If instead $k=1$, then $T=\Sigma T'$ and all the $v_i$'s belong to $T'$. Now, the result follows by induction, since $T'$ has $m-1$ vertices.
\end{proof}
\end{lemma}
In what follows we assume $A$ is a globular sum, decomposed as $A\cong \Sigma \alpha_1 \plus{ D_0} \Sigma \alpha_2\plus{ D_0} \ldots \plus{ D_0} \Sigma \alpha_q$, and we define the maps in the zig-zag diagram associated with $\cyl(A)$.
	\begin{defi}
		\label{upward maps}
	Consider a globular sum $B\in \mathscr{L}(A)$. We define a map $z^A_B \colon A \rightarrow B$ as follows.
		Suppose $B$ is obtained by adjoining a new vertex $*_B$ to $A$, and let $m=\height(*_B)$. Observe that if $m>1$ then $*_B$ is necessarily adjoined to a unique $\Sigma \alpha_{i}$.
		Let $x=\iota^B_{m-1}(*_B)$ and consider $F=(\iota^B_{m-1})^{-1}\{x\}$. Clearly, $*_B \in F$. We now have two possible cases:
		\begin{itemize}
\item If $*_B=\min F$ then $z^A_B\in \Theta_0$ is the unique map whose underlying map of trees is given by $\chi^A_B$ if $n+1>m\geq 1$
 (see Definition 2.3.10 in \cite{AR1} for an explanation on how to associate a map of trees to a map in $\Theta_0$, and for the fact that there is only one such in this case). If $m=n+1$ then $A=\partial B$ and we set $z^A_B=\partial_{\tau}$. 
\item If $*_B\neq \min F$, let $y\in F$ be the predecessor of $*_B$ in $F$, and let $C$ be the subtree of $B$ over $y$. Then we apply Lemma \ref{tree decomposition with vertices} to the tree associated with $A$ and the set of vertices $\{v_1,\ldots,v_n,y\}$ where the $v_i$'s are all the leaves which do not lie above $y$. This allows us to define the map $z^A_B$ by imposing it to be $\Sigma^m({}_{C}w)\colon \Sigma^m C \rightarrow \Sigma^m C \plus{D_m}D_{m+1}$ (as in Definition \ref{whiskering w}) on $\Sigma^m C$ and the identity everywhere else.
	\end{itemize} 

Dually, we define a map $v^A_B\colon A \rightarrow B$ by cases:
\begin{itemize}
	\item If $*_B=\max F$ then $v^A_B\in \Theta_0$ is the unique map whose underlying map of trees is given by $\chi^A_B$ if $n+1>m\geq 1$. If $m=n+1$ then $A=\partial B$ and we set $v^A_B=\partial_{\sigma}$.
	\item If $*_B\neq \max F$, let $y\in F$ be the successor of $*_B$ in $F$, and let $C$ be the subtree of $B$ over $y$. Then we apply Lemma \ref{tree decomposition with vertices} to the tree associated with $A$ and the set of vertices $\{v_1,\ldots,v_n,y\}$ where the $v_i$'s are all the leaves which do not lie above $y$. This allows us to define the map $z^A_B$ by imposing it to be $\Sigma^m({w}_{C})\colon \Sigma^m C \rightarrow D_{m+1} \plus{D_m}\Sigma^m C$ (as in Definition \ref{whiskering w}) on $\Sigma^m C$ and the identity everywhere else.
\end{itemize} 
\end{defi} 
We are now ready to give the following definition
\begin{defi}
Given a globular sum $A$, let $\mathscr{L}(A)=\{A_1\lessdot A_2 \lessdot \ldots \lessdot A_m\}$, so that $m=\vert \mathscr{L}(A) \vert$. We define a functor $\Cyl(A)\colon \mathbf{I}_{\vert \mathscr{L}(A) \vert}\rightarrow \wgpd$ by setting:
\begin{itemize}
	\item $\Cyl(A)(0,k)=A$ for every $0\leq k \leq n$.
	\item $\Cyl(A)(1,q)= A_{q+1}$ for every $0\leq q \leq \vert \mathscr{L}(A) \vert-1$.
	\item $\Cyl(A) \left ( \left ( 0,r \right ) \rightarrow \left ( 1,r-1\right )\right )=z^A_{A_{r-1}}$.
	\item $\Cyl(A) \left ( \left ( 0,r \right ) \rightarrow \left ( 1,r\right )\right )=v^A_{A_r}$.
\end{itemize}
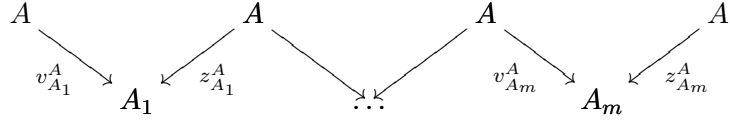
\[\bfig 
\morphism(-800,0)|l|/@{>}@<0pt>/<400,-300>[A`A_1;v^A_{A_1}]
\morphism(0,0)|a|/@{>}@<0pt>/<400,-300>[A`\ldots;]
\morphism(0,0)|r|/@{>}@<0pt>/<-400,-300>[A`A_1;z^A_{A_1}]
\morphism(800,0)|a|/@{>}@<0pt>/<-400,-300>[A `\ldots;]
\morphism(800,0)|l|/@{>}@<0pt>/<400,-300>[A `A_m;v^A_{A_m}]
\morphism(1600,0)|r|/@{>}@<0pt>/<-400,-300>[A`A_m;z^A_{A_m}]
\efig 
\]
\end{defi}
\begin{ex}
	\label{ex cyl decomp}
	Using the trees listed in \eqref{list of trees}, we can write down the full zig-zag diagram corresponding to $\cyl(D_2 \amalg_{D_0} D_1)$.
	\[
	\begin{tikzpicture}
	\matrix (m) [matrix of math nodes,row sep=1.5em,column sep=1.5em,minimum width=2em]
	{
		&  \ \ \ D_2\plus{D_0}D_1 \plus{D_0} D_1 \\
		D_2 \plus{D_0} D_1 & \\
		& D_2 \plus{D_0} D_2 \\
		D_2 \plus{D_0} D_1&\\
		&  \ \ \ D_2 \plus{ D_0} D_1 \plus{D_0} D_1\\
		D_2 \plus{D_0} D_1 &\\
		& \ \ \ D_2 \plus{D_1} D_2 \plus{D_0} D_1\\
		D_2 \plus{D_0} D_1 &\\
		& D_3 \plus{D_0} D_1\\
		D_2 \plus{D_0} D_1 &\\
		& \ \ \ D_2 \plus{D_1}D_2 \plus{D_0}D_1\\
		D_2 \plus{D_0}D_1&\\
		& D_1 \plus{D_0} D_2 \plus{D_0} D_1\\};
	\path[-stealth]
	(m-2-1) edge node [above] {$1\coprod w$} (m-1-2)
	edge  node [above] {$1 \coprod \sigma$} (m-3-2)
	(m-4-1)edge  node [above] {$1 \coprod \tau$} (m-3-2)
	edge  node [above] {$1 \coprod w$} (m-5-2)
	(m-6-1)edge  node [above] {$w \coprod 1$} (m-5-2)
	edge  node [above] {$(i_0,i_2)$} (m-7-2)
	(m-8-1)edge  node [above] {$w \coprod 1$} (m-7-2)
	edge  node [above] {$\sigma \coprod 1$} (m-9-2)
	(m-10-1)edge  node [above] {$\tau \coprod 1$} (m-9-2)
	edge  node [above] {$w \coprod 1$} (m-11-2)
	(m-12-1)edge  node [above] {$(i_1,i_2)$} (m-11-2)
	edge  node [above] {$w \coprod 1$} (m-13-2);

	\tikzstyle{every node}=[circle, draw,
	inner sep=0pt, minimum width=2pt]
	\node[] (A) at (5.5,7) {};
	
	\node[] (B) at (5,7.5) {};
	\node[] (C) at (5.5,7.5) {};
	\node[] (D) at (6,7.5) {};
	
	\node[] (E) at (5,8) {};
	
	\path [] (A) edge (B);
	\path [draw=red, very thick] (A) edge (D);
	\path [] (A) edge (C);
	\path [] (B) edge (E);

	\node[] (A') at (5.5,4.5) {};
	
	\node[] (B') at (5,5) {};
	\node[] (C') at (5.5,5.5) {};
	\node[] (D') at (5.5,5) {};
	
	\node[] (E') at (5,5.5) {};
	
	\path [] (A') edge (B');
	\path [draw=red, very thick] (D') edge (C');
	\path [] (A') edge (D');
	\path [] (B') edge (E');

	\node[] (A'') at (5.5,2) {};
	
	\node[] (B'') at (5,2.5) {};
	\node[] (C'') at (5.5,2.5) {};
	\node[] (D'') at (6,2.5) {};
	
	\node[] (E'') at (5,3) {};
	
	\path [] (A'') edge (B'');
	\path [draw=red, very thick] (A'') edge (C'');
	\path [] (A'') edge (D'');
	\path [] (B'') edge (E'');

	\node[] (a) at (5.5,-0.5) {};
	
	\node[] (b) at (5,0) {};
	\node[] (c) at (5.5,0.5) {};
	\node[] (d) at (6,0) {};
	
	\node[] (e) at (4.5,0.5) {};
	
	\path [] (a) edge (b);
	\path [draw=red, very thick] (b) edge (c);
	\path [] (a) edge (d);
	\path [] (b) edge (e);

	\node[] (a'') at (5.5,-3) {};
	
	\node[] (b'') at (5,-2.5) {};
	\node[] (c'') at (5,-2) {};
	\node[] (d'') at (6,-2.5) {};
	
	\node[] (e'') at (5,-1.5) {};
	
	\path [] (a'') edge (b'');
	\path [draw=red, very thick] (c'') edge (e'');
	\path [] (a'') edge (d'');
	\path [] (b'') edge (c'');

	\node[] (a') at (5.5,-5.5) {};
	
	\node[] (b') at (5,-5) {};
	\node[] (c') at (5.5,-4.5) {};
	\node[] (d') at (6,-5) {};
	
	\node[] (e') at (4.5,-4.5) {};
	
	\path [] (a') edge (b');
	\path [draw=red, very thick] (b') edge (e');
	\path [] (a') edge (d');
	\path [] (b') edge (c');

	\node[] (a''') at (5.5,-8) {};
	
	\node[] (b''') at (5.5,-7.5) {};
	\node[] (c''') at (5.5,-7) {};
	\node[] (d''') at (6,-7.5) {};
	
	\node[] (e''') at (5,-7.5) {};
	
	\path [] (a''') edge (b''');
	\path [draw=red, very thick] (a''') edge (e''');
	\path [] (a''') edge (d''');
	\path [] (b''') edge (c''');
	\end{tikzpicture}
	\]

\end{ex}
We will now prove that the colimit of the zig-zag diagram associated with a globular sum $A$ we have just defined is precisely $\cyl(A)$. To do so we need two preliminary lemmas, which we now present.
\begin{lemma}
	\label{cyl of suspension}
$\cyl(\Sigma B)$ can be expressed as the colimit of the following diagram
\begin{equation}
\label{cyl's of suspensions}
\bfig 
\morphism(-800,0)|a|/@{>}@<0pt>/<400,-300>[\Sigma B`\Sigma B \plus{ D_0} D_1;\iota]
\morphism(0,0)|a|/@{>}@<0pt>/<400,-300>[\Sigma B`\Sigma \cyl(B);\Sigma(\iota_0)]
\morphism(0,0)|a|/@{>}@<0pt>/<-400,-300>[\Sigma B`\Sigma B \plus{ D_0} D_1;{}_{\Sigma B}w]
\morphism(800,0)|a|/@{>}@<0pt>/<-400,-300>[\Sigma B`\Sigma \cyl(B);\Sigma(\iota_1)]
\morphism(800,0)|a|/@{>}@<0pt>/<400,-300>[\Sigma B`D_1 \plus{ D_0}\Sigma B;{w}_{\Sigma B}]
\morphism(1600,0)|a|/@{>}@<0pt>/<-400,-300>[\Sigma B`D_1 \plus{ D_0}\Sigma B;\iota]
\efig 
\end{equation}
\begin{proof}
As we let $B$ vary in $\Theta_0$, we see that the colimit of the zig-zag in the statement defines a globular functor $\Theta_0 \rightarrow \wgpd$, which coincides with $\cyl(\Sigma D_{\bullet})$ on $\G$. Therefore, there exists a natural isomorphism as stated thanks to the universal property of $\Theta_0$.
\end{proof}
\end{lemma}
We record here a categorical lemma that will be used to prove the result immediately after it.
\begin{lemma}
	\label{iterated pushouts}
Given a category with pushouts $\mathcal{C}$ and a diagram in it of the form
\[
\bfig
\morphism(0,0)|a|/@{>}@<0pt>/<-400,-300>[C`A;]
\morphism(0,0)|a|/@{>}@<0pt>/<400,-300>[C`B;]
\morphism(-400,-300)|a|/@{>}@<0pt>/<-400,-300>[A`A';]
\morphism(400,-300)|a|/@{>}@<0pt>/<400,-300>[B`B';]
\efig
\]
we get a pushout square
\[
\bfig
\morphism(0,0)|a|/@{>}@<0pt>/<500,0>[A\plus{C}B`A\plus{C}B';]
\morphism(0,0)|a|/@{>}@<0pt>/<0,-500>[A\plus{C}B`A'\plus{C}B;]
\morphism(500,0)|a|/@{>}@<0pt>/<0,-500>[A\plus{C}B'`A'\plus{C}B';]
\morphism(0,-500)|a|/@{>}@<0pt>/<500,0>[A'\plus{C}B`A'\plus{C}B'; ]
\efig
\] 
\end{lemma}
We want to prove the following result
\begin{prop}
Given a globular sum $A$, there exists a natural isomorphism in $\wgpd$:
\[\colim_{\mathbf{I}_{\vert \mathscr{L}(A) \vert }}\Cyl(A)\cong \cyl(A)\]
\end{prop}
\begin{proof}
In what follows, $\iota$ will denote a colimit inclusion, unless otherwise stated.

We make use of the (unique) decomposition of globular sums described in Lemma \ref{decomposition of glob sum}, which gives \[A\cong \Sigma \alpha_1 \plus{ D_0} \Sigma \alpha_2\plus{ D_0} \ldots \plus{ D_0} \Sigma \alpha_q\]
Therefore, by globularity of the cylinder functor, we have the isomorphism \[\cyl(A) \cong \cyl(\Sigma\alpha_1)\plus{\cyl(D_0)}\cyl(\Sigma\alpha_2)\plus{\cyl(D_0)}\ldots \plus{\cyl(D_0)} \cyl(\Sigma\alpha_q)\]
We can break the ordered set $\mathscr{L}(A)$ into subintervals by taking into consideration the globular sums $A_i$ for which the new edge is joined at the root. More precisely, let $1,m_1,\ldots, m_k,p=\vert \mathscr{L}(A) \vert$ be the ordered sequence of integers such that $A_{m_i}$ is obtained from $A$ by adding a new vertex at height $1$. We then have \[\mathscr{L}(A)=\{A_1\}\cup \{A_2,\ldots, A_{m_1 -1}\}\cup \{A_{m_1}\}\cup\ldots \cup  \{A_{m_{k-1}}\}\cup\{A_{m_{k-1}+1},\ldots A_{m_k -1}\}\cup\{A_{p}\}\]
and a corresponding isomorphism \[\mathbf{I}_{\vert \mathscr{L}(A)\vert}\cong \mathbf{I}_{1} \bullet \mathbf{I}_{p-m_{k-1}-1} \bullet \mathbf{I}_1 \bullet \ldots \bullet \mathbf{I}_{m_1-2} \bullet \mathbf{I}_{1} \] 
This, in turn, induces an isomorphism of diagrams \[\Cyl(A)\cong \mathbf{A_{p}}\bullet \Cyl(A)'_k \bullet \mathbf{A_{m_{k-1}}}\bullet \Cyl(A)'_{k-1}\bullet \ldots \bullet \Cyl(A)'_1 \bullet \mathbf{A_{1}}\] where we define $\Cyl(A)'_i=\Cyl(A)_{\vert \mathbf{I}_{(m_i-1 )- (m_{i-1}+1)+1}}$, and $\mathbf{A_1},\mathbf{A_p},\mathbf{A_{m_i}}$ are zig-zags of length $1$ which we now describe. By definition, $\mathbf{A_{m_i}}(1,0)=A_{m_i}$, and one has that $\mathbf{A_{m_i}}\left((0,0)\rightarrow (1,0)\right)$ is given by: 
\[1\plus{ D_0}w_{\Sigma (\alpha_{i})}\plus{ D_0}1\colon\Sigma(\alpha_1)\plus{ D_0}\ldots\plus{ D_0} \Sigma(\alpha_k)\rightarrow \Sigma(\alpha_1)\plus{ D_0}\ldots\plus{ D_0}\left(D_1 \plus{D_0}\Sigma(\alpha_i)\right)\plus{ D_0}\ldots\plus{ D_0} \Sigma(\alpha_k)\plus{D_0} D_1\]
On the other hand, we have that $\mathbf{A_{m_i}}\left((0,1)\rightarrow (1,0)\right)$
coincides with 
\[1\plus{ D_0} {}_{\Sigma (\alpha_{k-i+1})}w\plus{ D_0}1\colon\Sigma(\alpha_1)\plus{ D_0}\ldots\plus{ D_0} \Sigma(\alpha_k)\rightarrow \Sigma(\alpha_1)\plus{ D_0}\ldots\plus{ D_0}\left( \Sigma(\alpha_{k-i+1})\plus{D_0} D_1\right) \plus{ D_0}\ldots\plus{ D_0} \Sigma(\alpha_k) \]
$\mathbf{A_{1}}\left((0,0)\rightarrow (1,0)\right)$ is the map
\[\iota\colon\Sigma(\alpha_1)\plus{ D_0}\ldots\plus{ D_0}\Sigma(\alpha_k)\rightarrow \Sigma(\alpha_1)\plus{ D_0}\ldots\plus{ D_0} \Sigma(\alpha_k)\plus{D_0} D_1\]
and $\mathbf{A_1}\left((0,1)\rightarrow (1,0)\right)$ is given by
\[1\plus{ D_0}{}_{\Sigma (\alpha_{k})}w\colon\Sigma(\alpha_1)\plus{ D_0}\ldots\plus{ D_0} \Sigma(\alpha_k)\rightarrow \Sigma(\alpha_1)\plus{ D_0}\ldots\plus{ D_0}\left(\Sigma(\alpha_k)\plus{D_0} D_1\right)\]
Finally, $\mathbf{A_p}\left((0,0)\rightarrow (1,0)\right)$ coincides with
\[w_{\Sigma(\alpha_1)}\plus{D_0}1\colon\Sigma(\alpha_1)\plus{ D_0}\ldots\plus{ D_0} \Sigma(\alpha_k)\rightarrow \left( D_1 \plus{D_0}\Sigma(\alpha_1)\right) \plus{ D_0}\ldots\plus{ D_0}\Sigma(\alpha_k)\]
and $\mathbf{A_p}\left((0,1)\rightarrow (1,0)\right)$ is the map 
\[\iota\colon\Sigma(\alpha_1)\plus{ D_0}\ldots\plus{ D_0}\Sigma(\alpha_k)\rightarrow D_1\plus{D_0}\Sigma(\alpha_1)\plus{ D_0}\ldots\plus{ D_0} \Sigma(\alpha_k) \]
We now want to show that 
\begin{equation}
\label{3}
\colim_{\mathbf{I}_{\mathbf{I}_{(m_i-1 )- (m_{i-1}+1)+1}}} \Cyl(A)'_i \cong \Sigma(\alpha_1)\plus{ D_0}\ldots\plus{ D_0}\Sigma(\alpha_{i-1})\plus{ D_0} \Sigma\cyl( \alpha_{i})\plus{ D_0}\Sigma(\alpha_{i+1})\plus{ D_0}\ldots\plus{ D_0} \Sigma(\alpha_{k})
\end{equation}
Firstly, notice that the interval $ \{A_{m_{i-1} +1},\ldots,A_{m_{i}-1}\}$ is isomorphic to $\Sigma(\mathscr{L}(\alpha_{i}))$, i.e. the image of the set $\mathscr{L}(\alpha_{i})$ under the object-part function of the functor $\Sigma$.

By inspection of the maps $z^A_B$ and $v^A_B$ of Definition \ref{upward maps}, we see that the diagram $\Cyl(A)'_i$ coincides with \[\Sigma(\alpha_1)\plus{ D_0}\ldots\plus{ D_0}\Sigma(\alpha_{i-1})\plus{ D_0}\Sigma \circ \Cyl(\alpha_i)\plus{ D_0}\Sigma(\alpha_{i+1})\plus{ D_0}\ldots\plus{ D_0} \Sigma(\alpha_{k})\] 
Using Remark \ref{cocontinuity of UoSIgma}, we see that \eqref{3} holds. Thus, thanks to Lemma \ref{colimit of zig zags}, the colimit of the diagram $\Cyl(A)$ coincides with the colimit of the zig-zag on the left-hand side below, where $\coprod$ denotes the operation $\coprod_{D_0}$. A further application of Lemma \ref{colimit of zig zags} and Lemma \ref{cyl of suspension} proves that this last colimit is in turn isomorphic to the colimit of the right-hand side zig-zag below
\[\bfig 
\morphism(0,0)|a|/@{>}@<0pt>/<400,-300>[A`\Sigma(\alpha_1)\plus{ D_0}\ldots\plus{ D_0} \Sigma(\alpha_k)\plus{D_0} D_1;\iota]
 \morphism(0,-600)|r|/@{>}@<0pt>/<400,300>[A`\Sigma(\alpha_1)\plus{ D_0}\ldots\plus{ D_0} \Sigma(\alpha_k)\plus{D_0} D_1;1\plus{} {}_{\Sigma(\alpha_k)}w]
\morphism(0,-600)|a|/@{>}@<0pt>/<400,-400>[A`\Sigma(\alpha_1)\plus{ D_0}\ldots\plus{ D_0} \Sigma\left(\cyl(\alpha_k)\right); \ \ 1\plus{} \Sigma(\iota_0)]
\morphism(0,-1300)|r|/@{>}@<0pt>/<400,300>[A`\Sigma(\alpha_1)\plus{ D_0}\ldots\plus{ D_0} \Sigma\left(\cyl(\alpha_k)\right);1\plus{}\Sigma(\iota_1)]
\morphism(0,-1300)|r|/@{>}@<0pt>/<400,-300>[A`\ldots;]
\morphism(0,-1900)|r|/@{>}@<0pt>/<400,300>[A`\ldots;]
\morphism(0,-1900)|r|/@{>}@<0pt>/<400,-300>[A`\Sigma\left(\cyl(\alpha_1)\right)\plus{ D_0}\ldots\plus{ D_0} \Sigma(\alpha_k);\Sigma(\iota_0)\plus{} 1]
\morphism(0,-2500)|r|/@{>}@<0pt>/<400,300>[A`\Sigma\left(\cyl(\alpha_1)\right)\plus{ D_0}\ldots\plus{ D_0} \Sigma(\alpha_k);\Sigma(\iota_1)\plus{} 1]
\morphism(0,-2500)|r|/@{>}@<0pt>/<400,-300>[A`D_1 \plus{D_0}\Sigma(\alpha_1)\plus{ D_0}\ldots\plus{ D_0} \Sigma(\alpha_k); \ \ w_{\Sigma(\alpha_1)}\plus{}1]
\morphism(0,-3100)|a|/@{>}@<0pt>/<400,300>[A`D_1 \plus{D_0}\Sigma(\alpha_1)\plus{ D_0}\ldots\plus{ D_0} \Sigma(\alpha_k);\iota]

\morphism(2000,0)|r|/@{>}@<0pt>/<400,-300>[A`\Sigma(\alpha_1)\plus{ D_0}\ldots\plus{ D_0} \cyl\left(\Sigma(\alpha_k)\right);1\plus{}\iota_0]
\morphism(2000,-600)|r|/@{>}@<0pt>/<400,300>[A`\Sigma(\alpha_1)\plus{ D_0}\ldots\plus{ D_0} \cyl\left(\Sigma(\alpha_k)\right);1\plus{} (i\circ w_{\Sigma(
	\alpha_k)})]
\morphism(2000,-600)|r|/@{>}@<0pt>/<400,-300>[A`\ldots;]
\morphism(2000,-1200)|r|/@{>}@<0pt>/<400,300>[A`\ldots;]
\morphism(2000,-1200)|r|/@{>}@<0pt>/<400,-300>[A`\cyl\left(\Sigma(\alpha_1)\right)\plus{ D_0}\ldots\plus{ D_0} \Sigma(\alpha_k); (i\circ {}_{\Sigma(\alpha_1)}w)\plus{}1]
\morphism(2000,-1800)|r|/@{>}@<0pt>/<400,300>[A`\cyl\left(\Sigma(\alpha_1)\right)\plus{ D_0}\ldots\plus{ D_0} \Sigma(\alpha_k);\iota_1 \plus{}1]

\efig 
\]
To finish the proof we now apply Lemma \ref{iterated pushouts} $(k-1)$ times to diagrams of the form
\[
\bfig 
\morphism(0,0)|a|/@{>}@<0pt>/<-1100,-300>[D_1`\Sigma(\alpha_1)\plus{ D_0}\ldots\plus{ D_0}\left( \Sigma(\alpha_{i})\plus{D_0} D_1\right) \plus{ D_0}\ldots\plus{ D_0} \Sigma(\alpha_k) ;]
\morphism(0,0)|a|/@{>}@<0pt>/<1100,-300>[D_1`\Sigma(\alpha_1)\plus{ D_0}\ldots\plus{ D_0}\left( D_1\plus{D_0}\Sigma(\alpha_{i+1}) \right) \plus{ D_0}\ldots\plus{ D_0} \Sigma(\alpha_k) ;]
\morphism(-1100,-300)|a|/@{>}@<0pt>/<0,-400>[\Sigma(\alpha_1)\plus{ D_0}\ldots\plus{ D_0}\left( \Sigma(\alpha_{i})\plus{D_0} D_1\right) \plus{ D_0}\ldots\plus{ D_0} \Sigma(\alpha_k)`\Sigma(\alpha_1)\plus{ D_0}\ldots\plus{ D_0}\cyl\left( \Sigma(\alpha_{i})\right) \plus{ D_0}\ldots\plus{ D_0} \Sigma(\alpha_k);]
\morphism(1100,-300)|a|/@{>}@<0pt>/<0,-400>[\Sigma(\alpha_1)\plus{ D_0}\ldots\plus{ D_0}\left( D_1\plus{D_0}\Sigma(\alpha_{i+1}) \right) \plus{ D_0}\ldots\plus{ D_0} \Sigma(\alpha_k)`\Sigma(\alpha_1)\plus{ D_0}\ldots\plus{ D_0}\cyl\left( \Sigma(\alpha_{i+1})\right) \plus{ D_0}\ldots\plus{ D_0} \Sigma(\alpha_k);]
\efig 
\]
\end{proof}
Given $B\in \mathscr{L}(A)$, we denote by $i_B\colon B \rightarrow \cyl(A)$ the colimit inclusion.
\begin{rmk}
\label{some maps in cyl(A) zig zag}
If we consider the globular sum $A$ as in the previous theorem, then $D_1 \plus{D_0}A$ and $A \plus{D_0} D_1$ both belongs to $\mathcal{L}(A)$ by construction. It is clear from the proof of the previous theorem that the colimit inclusion $i_{D_1 \plus{D_0}A}\colon D_1 \plus{D_0}A \rightarrow \cyl(A)$ is given by $(\cyl(\sigma^{\dim (A)}),\iota_1)$. In a completely analogous manner, $i_{A \plus{D_0} D_1}\colon A \plus{D_0} D_1 \rightarrow \cyl(A)$ is equal to $(\iota_0,\cyl(\tau^{\dim (A)}))$.

If the globular sum $A$ decomposes as $A=S\plus{D_0} T$ then $S\plus{D_0}D_1\plus{D_0}T$ belongs to $\mathcal{L}(A)$, and the colimit inclusion \[i_{S\plus{D_0}D_1\plus{D_0}T}\colon S\plus{D_0}D_1\plus{D_0}T \rightarrow \cyl(A)\cong \cyl(S)\plus{\cyl(D_0)}\cyl(T)\] is given, on each summand respectively, by the composites
\[
\bfig
\morphism(0,0)|a|/@{>}@<0pt>/<500,0>[S`\cyl(S) ;\iota_0]
\morphism(500,0)|a|/@{>}@<0pt>/<1000,0>[\cyl(S)`\cyl(S)\plus{\cyl(D_0)}\cyl(T) ;i]
\morphism(250,-400)|a|/@{>}@<0pt>/<1000,0>[D_1\cong \cyl(D_0)`\cyl(S) \plus{\cyl(D_0)}\cyl(T);i]
\morphism(0,-800)|a|/@{>}@<0pt>/<500,0>[T`\cyl(T) ;\iota_1]
\morphism(500,-800)|a|/@{>}@<0pt>/<1000,0>[\cyl(T)`\cyl(S)\plus{\cyl(D_0)}\cyl(T) ;i]
\efig
\] where we denote with $i$ the obvious colimit inclusions.

Finally, observe that if $B\in \mathcal{L}(A)$ and the new edge is attached at height $m>0$, say to $\Sigma(\alpha_{i})$, then the colimit inclusion $i_B\colon B \rightarrow \cyl(A)$ factors through the natural map \[\Sigma(\alpha_1)\plus{D_0}\ldots \plus{D_0} \Sigma \cyl(\alpha_{i})\plus{D_0}\ldots \plus{D_0}\Sigma(\alpha_k) \rightarrow \cyl(A)\] whose existence is evident from the proof we have just presented, via the map
\[
1\plus{D_0}\Sigma(i_{B'})\plus{D_0}1\colon B\cong \Sigma(\alpha_1)\plus{D_0}\ldots \plus{D_0} \Sigma B'\plus{D_0}\ldots \plus{D_0}\Sigma(\alpha_k) \rightarrow \Sigma(\alpha_1)\plus{D_0}\ldots \plus{D_0} \Sigma \cyl(\alpha_{i})\plus{D_0}\ldots \plus{D_0}\Sigma(\alpha_k)
\] for a unique $B'\in \mathcal{L}(\alpha_{i})$.
\end{rmk}
\section{Operations on cylinders}
\subsection{Overview}
\label{overview}
Consider the globular sum preserving functor
\begin{equation}
\label{cyl} \cyl(\bullet)\colon  \Theta_0 \rightarrow \wgpd
\end{equation}
of which we have just given a more explicit definition.
Constructing an extension of this functor to a cocontinuous endofunctor on $\wgpd$ amounts to endowing \eqref{cyl} with the structure of a co-$\infty$-groupoid.

This means that given the coherator for $\infty$-groupoids $\mathfrak{C}$, we have to find an extension of the form:
\begin{equation}
\label{extension of cyl}
\bfig
\morphism(0,0)|a|/@{>}@<0pt>/<500,0>[\Theta_0`\wgpd;\cyl(\bullet)]
\morphism(0,0)|l|/@{>}@<0pt>/<0,-300>[\Theta_0`\mathfrak{C};]
\morphism(0,-300)|r|/@{-->}@<0pt>/<500,300>[\mathfrak{C}`\wgpd;\cyl(\bullet)]
\efig 
\end{equation}
Thanks to the cellularity property of $\mathfrak{C}$ and to Lemma \ref{univ prop of glob th}, this becomes an inductive process, where we assume we have an operation $\rho\colon D_n \rightarrow A$ in $\mathfrak{C}$, as well as interpretations of its boundary
\[\cyl (\rho \circ \sigma),\cyl (\rho \circ \tau)\colon\cyl(D_{n-1}) \rightarrow \cyl(A)\]
and we need to define a map $\cyl (\rho )\colon\cyl(D_{n}) \rightarrow \cyl(A)$ such that for $\epsilon=\sigma,\tau$: \[\cyl(\rho)\circ \cyl(\epsilon)=\cyl(\rho \circ \epsilon)\colon \cyl(D_{n-1})\rightarrow \cyl(A)\]
Given the fact that we have explained how to decompose of cylinders on globular sums into contractible  pieces,  we may try to use this fact to build maps representing a first approximation of these operations between cylinders.

At his point, we need a technical assumption in order to perform a specific construction that would not be applicable otherwise.
In fact, we assume that $\mathfrak{D}$ is a (homogeneous) coherator for $\infty$-categories, and doing so we manage to define in \ref{rho hat defi}, for every homogeneous operation
$\rho\colon D_n \rightarrow A$ in $\mathfrak{D}$, a map of $\infty$-groupoids \[\hat{\rho}\colon \cyl(D_n)\rightarrow \cyl(A)\] satisfying two properties, that can be expressed in the following commutative diagrams:
\begin{equation}
\label{hat properties}
\bfig
\morphism(-1000,0)|a|/@{>}@<0pt>/<750,0>[D_n \coprod D_n`A \coprod A;\rho \coprod \rho]
\morphism(-1000,-500)|a|/@{>}@<0pt>/<750,0>[\cyl(D_n)`\cyl(A);\hat{\rho}]
\morphism(-250,0)|r|/@{>}@<0pt>/<0,-500>[A \coprod A`\cyl(A);\iota_0 \coprod \iota_1]
\morphism(-1000,0)|l|/@{>}@<0pt>/<0,-500>[D_n \coprod D_n`\cyl(D_n);\iota_0 \coprod \iota_1]

\morphism(500,0)|a|/@{>}@<0pt>/<1000,0>[\cyl(S^{n-1})`\cyl(A);(\widehat{\rho \circ \sigma},\widehat{\rho \circ \tau})]
\morphism(500,0)|a|/@{>}@<0pt>/<0,-500>[\cyl(S^{n-1})`\cyl(D_n);]
\morphism(500,-500)|r|/@{>}@<0pt>/<1000,500>[\cyl(D_n)`\cyl(A);\hat{\rho}]
\efig
\end{equation}
This map is to be thought of as a first approximation of the ``correct'' functorial interpretation $\cyl(\rho)$. It then needs to be modified, more precisely its boundary has to be modified, so that we actually get a functor $\cyl(\bullet)\colon \mathfrak{D} \rightarrow \wgpd$. We will consider an instance of this process in Section 11, where we construct a functor $\mathscr{D}_{\leq 2}\rightarrow \wgpd$ ($\mathscr{D}_{\leq 2}$ being the full subcategory of $\mathscr{D}$ on globular sums of dimension less or equal to $2$).
One then has to generalize this process of adjusting the boundary of the maps $\hat{\rho}$ to all higher dimensions in order to render these interpretations functorial, thus succeeding in extending $\cyl(\bullet)$ to $\mathfrak{D}$. 
Then, provided the following conjecture holds true, the extension problem in \eqref{extension of cyl} can be solved thanks to the constructions performed in Section 6.
\begin{conj}
	A coherator for $\infty$-categories is contractible (i.e. it is a coherator for $\infty$-groupoids) provided it can be endowed with a system of inverses, as in Definition \ref{structure systems}.	
\end{conj}
This conjecture states that Grothendieck $\infty$-groupoids are essentially equivalent to $\infty$-groupoids à la Batanin (see \cite{BAT}, Definition 9.5).

The idea to obtain the map $\hat{\rho}$ is to construct a vertical stack of $(n-1)$-dimensional (possibly degenerate) cylinders in the $\infty$-groupoid $\cyl(A)(x_0,  x_m)$ for an appropriate pair of $0$-cells $(x_0,  x_m)$ in $\cyl(A)$.
We then compose this vertical stack using a vertical composition operation, and
the result is an $(n-1)$-cylinder in the hom-groupoid $\cyl(A)(x_0,x_m)$, which, by construction, transposes to give the desired map \[\bfig \morphism  (0,0)|a|/@{>}@<0pt>/<800,0>[\cyl(D_n)`\cyl(A);\widehat{\rho}] \efig\]
\subsection{Vertical composition of cylinders}
\label{subsect vert comp of cyls}
The goal of this section is to define an operation that performs the vertical composition of a compatible stack of an $m$-tuple of $n$-cylinders.
This operation takes as input a sequence of $n$-cylinders $F_i\colon A_i \curvearrowright A_{i+1}$ in an $\infty$-groupoid $X$, and produces an $n$-cylinder denoted by \[F_m \otimes F_{m-1} \otimes \ldots \otimes F_1 \colon A_1 \curvearrowright A_{m+1}\] It is represented by a map 
\[\bfig
\morphism(750,0)|a|/@{>}@<0pt>/<1000,0>[\cyl (D_n)`\cyl(D_n) \otimes \ldots \otimes \cyl(D_n) ;]
\efig 
\]
where the codomain is defined to be the colimit of the following diagram:

\[ 
\bfig

\morphism(0,-500)|l|/@{>}@<0pt>/<-300,300>[D_n`\cyl(D_n);\iota_1]
\morphism(0,-500)|r|/@{>}@<0pt>/<300,300>[D_n`\cyl(D_n);\iota_0]

\morphism(600,-500)|l|/@{>}@<0pt>/<-300,300>[D_n`\cyl(D_n);\iota_1]  \morphism(600,-500)|r|/@{>}@<0pt>/<300,300>[D_n`\ldots;\iota_0]

\morphism(1200,-500)|l|/@{>}@<0pt>/<-300,300>[D_n`\ldots;\iota_1]  \morphism(1200,-500)|r|/@{>}@<0pt>/<300,300>[D_n`\cyl(D_n);\iota_0]
\efig
\]
Moreover, it will have the property that \[\epsilon(F_m \otimes F_{m-1} \otimes \ldots \otimes F_1 )=\epsilon(F_m) \otimes \epsilon(F_{m-1}) \otimes \ldots \otimes \epsilon(F_1)\] for $\epsilon=s,t$.
To begin with, we have to do some preliminary work.
\begin{lemma}
	Let $\C$ be a category and $X\in \mathrm{ob}(\A)$ an object of $\A$.
	Given an adjunction
	\[\xymatrixcolsep{1pc}
	\vcenter{\hbox{\xymatrix{
				**[l]	\A \xtwocell[r]{}_{G}^{F}{'\perp}& **[r]X \downarrow\A
	}}}
	\]
	there exist, for every $n>0$, functors $G_{n+1}\colon F^n X \downarrow \A \rightarrow \A$ and adjunctions 
	\[\xymatrixcolsep{1pc}
	\vcenter{\hbox{\xymatrix{
				**[l]	\A \xtwocell[r]{}_{G^{n+1}}^{F^{n+1}}{'\perp}& **[r] F^n X \downarrow\A
	}}}
	\]
	where $F^n$ denotes the obvious iteration of $F$.
	\begin{proof}
		The proceeds by induction, the case $n=0$ being valid by assumption.
		
		Let $n>0$ and suppose the lemma holds for $n-1$. We will prove it holds for $n$.
		Given an object $f\colon F^n X \rightarrow B$ in $F^n X \downarrow \A$, consider its transpose $\hat{f}\colon F^{n-1} X \rightarrow GB$ under the adjunction $F \dashv G$.
		Define $G_{n+1}(f) =G_n(\hat{f})$ on objects, and given an arrow $h\colon B \rightarrow C$ under $F^n X$, we set $G_{n+1}(h)\colon=G_n(G(h))$.
		
		The following chain of natural isomorphisms of hom-sets proves the statement:
		\[ \A(A,G_{n+1}(f)) \cong \A(A,G_n(\hat{f})) \cong  F^{n-1}X \downarrow \A (F^n A, \hat{f}) \cong F^n \downarrow \A (F^{n+1} A,f).\qedhere \]\end{proof}
\end{lemma}
Applying this to the adjunction 
\[\xymatrixcolsep{1pc}
\vcenter{\hbox{\xymatrix{
			**[l]\wgpd \xtwocell[r]{}_{\Omega}^{\Sigma}{'\perp}& **[r] S^0 \downarrow\wgpd
}}}
\]
we get adjunctions for every $n>0$ of the form
\[\xymatrixcolsep{1pc}
\vcenter{\hbox{\xymatrix{
			**[l]	\wgpd \xtwocell[r]{}_{\Omega^{n}}^{\Sigma^n}{'\perp}& **[r] S^{n-1} \downarrow\wgpd
}}}
\]
Our next piece of preliminary work will be to use Proposition \ref{direct cof cyl} to construct some morphisms in $\wgpd^{\mathbb{G}}$, by extending suitable maps into contractible objects along cofibrations. The solution to these extension problems will produce cylinders that represent coherent rebracketings of certain composites of globular pasting diagrams in a given $\infty$-groupoid.

For example, given an $\infty$-groupoid $X$ and a map $(f,\alpha,g)\colon D_1\amalg_{ D_0} D_2\amalg_{ D_0} D_1\rightarrow X$, that can be represented as the following pasting diagram labelled by cells of $X$
\[
\xymatrix{\bullet \ar[r]^-{f} &\bullet \rtwocell^h_k{\alpha} &\bullet \ar[r]^-{g} & \bullet}
\]
we can consider two ways of composing this pasting diagram, namely $(g \alpha)f$ and $g(\alpha f)$, where binary composition may be interpreted, for instance, using the maps ${}_{D_2}w,w_{D_2}$. In general these two cells will differ, and also their boundary will, being given respectively by the pairs of parallel $1$-cells $((gh)f,(gk)f)$ and $(g(hf),g(kf))$. Therefore, a comparison between the two $2$-cells cannot be encoded by a $3$-cell, but rather by a $2$-cylinder whose boundary consists of a pair of $1$-cylinders encoding a comparison between the (possibly) different $1$-cells we have just described.

To obtain these cylinders in general, given $m\geq 0$, we consider the following map in $\wgpd^{\mathbb{G}}$
\[ \bfig
\morphism(0,0)|a|/@{>}@<0pt>/<1000,0>[\Sigma^m D_{\bullet} \coprod \Sigma^m D_{\bullet} `\Sigma^m \cyl(D_{\bullet});\Sigma^m(\iota)]
\efig\]
This map belongs to $\mathbb{I}^{\mathbb{G}}$, thanks to Remark \ref{cocontinuity of UoSIgma} and Lemma \ref{preservation of cofs^D}.

In what follows, we assume we have chosen composition operations $\gamma\colon D_n \rightarrow D_1^{\otimes m}\plus{D_0} D_{n}\plus{D_0} D_1^{\otimes k}$ for $k,m,n>0$, which are compatible with the source and target maps, i.e. \[\gamma \circ \epsilon=\left( 1_{D_1^{\otimes m}}\plus{D_0}\epsilon\plus{D_0}1_{D_1^{\otimes k}}\right) \circ \gamma\] There is no risk of confusion in referring to all such maps as $\gamma$, because the codomain uniquely determines such $\gamma$.
\begin{defi}
	Given $q,m,k \geq 0$, thanks to Proposition \ref{glob sums are contractible} we have a pair of trivial fibrations in $\wgpd^{\mathbb{G}}$ \[D_1^{\otimes m}\plus{D_0}\Sigma D_{\bullet}\plus{D_0} D_1^{\otimes k}  \rightarrow * \]
	and 
	\[ D_1^{\otimes q}\plus{D_0} D_{m}\plus{D_{m-1}}\Sigma^{m}D_{\bullet}\plus{D_0}D_1^{\otimes k} \rightarrow *\]
	where the structural maps of the domains are the obvious ones and $*$ denotes the terminal object in $\wgpd^{\mathbb{G}}$.
	For $m,k \neq 0$ define maps 
	\[ \bfig
	\morphism(0,0)|a|/@{>}@<0pt>/<1000,0>[\Sigma D_{\bullet} \coprod \Sigma D_{\bullet} `D_1^{\otimes m}\plus{D_0}\Sigma D_{\bullet}\plus{D_0} D_1^{\otimes k} ;\psi^{m,k}]
	\efig \]
by setting the first component in dimension $n$ to be given by the composite
	\[
	\bfig 
	\morphism(0,0)|a|/@{>}@<0pt>/<1000,0>[D_{n+1}`D_1^{\otimes m-1}\plus{D_0} D_{n+1}\plus{D_0} D_1^{\otimes k} ;\gamma]
	\morphism(1000,0)|a|/@{>}@<0pt>/<2000,0>[D_1^{\otimes m-1}\plus{D_0} D_{n+1}\plus{D_0} D_1^{\otimes k} `D_1^{\otimes m}\plus{D_0} D_{n+1}\plus{D_0} D_1^{\otimes k} ;1_{D_1^{\otimes m-1}}\coprod w \coprod 1_{D_1^{\otimes k}}]	
	\efig 
	\] 
	and the second one to be
		\[
	\bfig 
	\morphism(0,0)|a|/@{>}@<0pt>/<1000,0>[D_{n+1}`D_1^{\otimes m}\plus{D_0} D_{n+1}\plus{D_0} D_1^{\otimes k-1} ;\gamma]
	\morphism(1000,0)|a|/@{>}@<0pt>/<2000,0>[D_1^{\otimes m}\plus{D_0} D_{n+1}\plus{D_0} D_1^{\otimes k-1} `D_1^{\otimes m}\plus{D_0} D_{n+1}\plus{D_0} D_1^{\otimes k} ;1_{D_1^{\otimes m}}\coprod w \coprod 1_{D_1^{\otimes k-1}}]	
	\efig 
	\] 
	This means that given an $\infty$-groupoid $X$ and a map \[(f_1,\ldots,f_m,\alpha,g_1,\ldots,g_k)	\colon D_1^{\otimes m}\plus{D_0} D_{n+1}\plus{D_0} D_1^{\otimes k} \rightarrow X\] we get a pair of $(n+1)$-cells in $X$ of the form $g_k\ldots g_1 (\alpha f_m)f_{m-1}\ldots f_1$ and $g_k\ldots g_2 (g_1 \alpha) f_mf_{m-1}\ldots f_1$, where juxtaposition is the result of composition using the appropriate $\gamma$.
 
	If $m=0$ and $k\neq 0$ define
	\[ \bfig
	\morphism(0,0)|a|/@{>}@<0pt>/<1000,0>[\Sigma D_{\bullet} \coprod \Sigma D_{\bullet} `\Sigma D_{\bullet}\plus{D_0} D_1^{\otimes k} ;\psi^{0,k}]
	\efig \]
	by setting the first component in dimension $n$ to be given by the composite
	\[
	\bfig 
	\morphism(0,0)|a|/@{>}@<0pt>/<600,0>[D_{n+1}`D_{n+1}\plus{D_0} D_{1} ;w]
	\morphism(600,0)|a|/@{>}@<0pt>/<1400,0>[D_{n+1}\plus{D_0} D_{1} `D_{n+1}\plus{D_0} D_1^{\otimes k} ;1_{D_{n+1}}\plus{D_0}\gamma]	
	\efig 
	\] 
	and the second one to be 
	\[
	\bfig 
	\morphism(0,0)|a|/@{>}@<0pt>/<750,0>[D_{n+1}` D_{n+1}\plus{D_0} D_1^{\otimes k-1} ;\gamma]
	\morphism(750,0)|a|/@{>}@<0pt>/<1300,0>[D_{n+1}\plus{D_0} D_1^{\otimes k-1} ` D_{n+1}\plus{D_0} D_1^{\otimes k} ; w \coprod 1_{D_1^{\otimes k-1}}]	
	\efig 
	\] 
	This means that given an $\infty$-groupoid $X$ and a map \[(\alpha,g_1,\ldots,g_k)	\colon  D_{n+1}\plus{D_0} D_1^{\otimes k} \rightarrow X\] we get a pair of $(n+1)$-cells in $X$ of the form $(g_k\ldots g_1) \alpha$ and $g_k\ldots g_2(g_1 \alpha)$, where juxtaposition is the result of composition using the appropriate $\gamma$ or $w$, as described above.
	Finally, if $k=0$ and $m \neq 0$ define
	\[ \bfig
	\morphism(0,0)|a|/@{>}@<0pt>/<1000,0>[\Sigma D_{\bullet} \coprod \Sigma D_{\bullet} ` D_1^{\otimes m} \plus{D_0} \Sigma D_{\bullet};\psi^{m,0}]
	\efig \]
	by setting the first component in dimension $n$ to be given by the composite
	\[
	\bfig 
	\morphism(0,0)|a|/@{>}@<0pt>/<600,0>[D_{n+1}`D_{1}\plus{D_0} D_{n+1} ;w]
	\morphism(600,0)|a|/@{>}@<0pt>/<1400,0>[D_{1}\plus{D_0} D_{n+1} `D_1^{\otimes m}\plus{D_0}D_{n+1}  ;\gamma\plus{D_0}1_{D_{n+1}}]	
	\efig 
	\] 
	and the second one to be
	\[
	\bfig 
	\morphism(0,0)|a|/@{>}@<0pt>/<750,0>[D_{n+1}`D_1^{\otimes m-1}  \plus{D_0} D_{n+1} ;\gamma]
	\morphism(750,0)|a|/@{>}@<0pt>/<1300,0>[D_1^{\otimes m-1}  \plus{D_0} D_{n+1} `D_1^{\otimes m} \plus{D_0} D_{n+1}  ; 1_{D_1^{\otimes m-1}} \plus{D_0} w ]	
	\efig 
	\]This means that given an $\infty$-groupoid $X$ and a map \[(f_1,\ldots,f_m,\alpha)	\colon D_1^{\otimes m}\plus{D_0} D_{n+1}\rightarrow X\] we get a pair of $(n+1)$-cells in $X$ of the form $\alpha (f_m f_{m-1}\ldots f_1)$ and $ (\alpha f_m) f_{m-1}\ldots f_1$, where juxtaposition is the result of composition using the appropriate $\gamma$ or $w$ as described above.
	
	For $m\geq 1$ also define
	\[ \bfig
	\morphism(0,0)|a|/@{>}@<0pt>/<1500,0>[\Sigma^m D_{\bullet} \coprod \Sigma^m D_{\bullet} ` D_1^{\otimes q}\plus{D_0} D_{m+1}\plus{D_{m}}\Sigma^{m}D_{\bullet}\plus{D_0}D_1^{\otimes k} ;\phi^{q,m,k}]
	\efig \]
	where the first component is given by 
	\[
	\bfig 
	\morphism(-200,0)|a|/@{>}@<0pt>/<1000,0>[D_{n+m}`D_{m+1}\plus{D_m} D_{n+m} ;\Sigma^m(w)]
	\morphism(800,0)|a|/@{>}@<0pt>/<1400,0>[D_{m+1}\plus{D_m} D_{n+m} `D_1^{\otimes q}\plus{D_0}D_{m+1}\plus{D_m}D_{n+m}\plus{ D_0}D_1^{\otimes k}  ;f]	
	\morphism(0,500)|l|/@{>}@<0pt>/<800,-500>[D_{m+1} `D_{m+1}\plus{D_m} D_{n+m} ;i]
	\morphism(0,500)|a|/@{>}@<0pt>/<1200,0>[D_{m+1} `D_1^{\otimes q}\plus{D_0}D_{m+1}\plus{ D_0}D_1^{\otimes k} ;\gamma]
	\morphism(1200,500)|r|/@{>}@<0pt>/<1000,-500>[D_1^{\otimes q}\plus{D_0}D_{m+1}\plus{ D_0}D_1^{\otimes k} `D_1^{\otimes q}\plus{D_0}D_{m+1}\plus{D_m}D_{n+m}\plus{ D_0}D_1^{\otimes k}  ;i]
	\morphism(0,-500)|l|/@{>}@<0pt>/<800,500>[D_{n+m} `D_{m+1}\plus{D_m} D_{n+m} ;i]
	\morphism(0,-500)|a|/@{>}@<0pt>/<1200,0>[D_{n+m} `D_1^{\otimes q}\plus{D_0}D_{n+m}\plus{ D_0}D_1^{\otimes k} ;\gamma]
	\morphism(1200,-500)|r|/@{>}@<0pt>/<1000,500>[D_1^{\otimes q}\plus{D_0}D_{n+m}\plus{ D_0}D_1^{\otimes k} `D_1^{\otimes q}\plus{D_0}D_{m+1}\plus{D_m}D_{n+m}\plus{ D_0}D_1^{\otimes k}  ;i]
	\efig 
	\] 
where, with a minor abuse of language, we let $i$ denote the various colimit inclusions and $f$ the map induced by the universal property of colimits.
  
The second component is given by 
\[
\bfig
\morphism(0,0)|a|/@{>}@<0pt>/<800,0>[D_{n+m}`D_1^{\otimes q}\plus{D_0}D_{n+m}\plus{ D_0}D_1^{\otimes k} ;\gamma]
\morphism(800,0)|a|/@{>}@<0pt>/<2400,0>[D_1^{\otimes q}\plus{D_0}D_{n+m}\plus{ D_0}D_1^{\otimes k} `D_1^{\otimes q}\plus{D_0}D_{m+1}\plus{D_m}D_{n+m}\plus{ D_0}D_1^{\otimes k}  ;1\plus{ D_0} \Sigma^m(w)\plus{ D_0}D_1]
\efig
\]
This means that, given a map \[(f_1,\ldots,f_q,\alpha,\beta,g_1,\ldots,g_k)\colon D_1^{\otimes q}\plus{D_0}D_{m+1}\plus{D_m}D_{n+m}\plus{ D_0}D_1^{\otimes k} \rightarrow X\] we get a pair of $(n+m)$-cells $(g_k\ldots g_1 \beta f_q \ldots f_1)(g_k\ldots g_1 \alpha f_q \ldots f_1)$ and $g_k\ldots g_1(\beta \alpha)f_q\ldots f_1$, where juxtaposition stands for the result of composing those cells using the appropriate operations described above. Notice that both these $(n+m)$-cells can be interpreted as $n$-cells in $\Omega^m(X,A,B)$ for appropriate choices of $A,B\colon S^m \rightarrow X$.

Similarly to $\phi^{q,m,k}$, we get a map 
\[ \bfig
\morphism(0,0)|a|/@{>}@<0pt>/<1500,0>[\Sigma^m D_{\bullet} \coprod \Sigma^m _{\bullet} ` D_1^{\otimes q}\plus{D_0}\Sigma^{m}D_{\bullet} \plus{D_{m}}D_{m+1}\plus{D_0}D_1^{\otimes k} ;\theta^{q,m,k}]
\efig \]
with a completely analogous definition on both components.
This time, given a map \[(f_1,\ldots,f_q, \alpha,\beta,g_1,\ldots,g_k)\colon D_1^{\otimes q}\plus{D_0}\Sigma^{m}D_{\bullet} \plus{D_{m}}D_{m+1}\plus{D_0}D_1^{\otimes k} \rightarrow X\] we get back a pair of $(n+m)$-cells $g_k\ldots g_1 (\beta \alpha) f_q \ldots f_1$ and $ (g_k\ldots g_1 \beta f_q \ldots f_1)(g_k\ldots g_1 \alpha f_q \ldots f_1)$, where juxtaposition stands for the result of composing those cells using the appropriate operations described above. Notice that both these $(n+m)$-cells can be interpreted as $n$-cells in $\Omega^m(X,A,B)$ for appropriate choices of $A,B\colon S^m \rightarrow X$.

Let us now define maps $\Psi^{m,k}, \ \Phi^{m,k}$ and $\Theta^{q,m,k}$ by choosing fillers as follows (the existence of which is ensured by Lemma \ref{Reedy construction}): 
	\begin{equation}
	\label{coherence cyls}
\bfig
\morphism(0,0)|a|/@{>}@<0pt>/<1500,0>[\Sigma D_{\bullet} \coprod \Sigma D_{\bullet} `D_1^{\otimes m}\plus{D_0}\Sigma D_{\bullet}\plus{D_0} D_1^{\otimes k} ;\psi^{m,k}]
\morphism(0,0)|a|/@{>}@<0pt>/<0,-500>[\Sigma D_{\bullet} \coprod \Sigma D_{\bullet} `\Sigma \cyl (D_{\bullet});\Sigma(\iota)]
\morphism(0,-500)|r|/@{>}@<0pt>/<1500,500>[\Sigma \cyl (D_{\bullet}) `D_1^{\otimes m}\plus{D_0}\Sigma D_{\bullet}\plus{D_0} D_1^{\otimes k};\Psi^{m,k}]

\morphism(0,-1000)|a|/@{>}@<0pt>/<1500,0>[\Sigma^m D_{\bullet} \coprod \Sigma^m D_{\bullet} `D_1^{\otimes q}\plus{D_0} D_{m+1}\plus{D_{m}}\Sigma^{m}D_{\bullet}\plus{D_0}D_1^{\otimes k};\phi^{q,m,k}]
\morphism(0,-1000)|a|/@{>}@<0pt>/<0,-500>[\Sigma^m D_{\bullet} \coprod \Sigma^m D_{\bullet} `\Sigma^m \cyl (D_{\bullet});\Sigma^m(\iota)]
\morphism(0,-1500)|r|/@{>}@<0pt>/<1500,500>[\Sigma^m \cyl (D_{\bullet}) `D_1^{\otimes q}\plus{D_0} D_{m+1}\plus{D_{m}}\Sigma^{m}D_{\bullet}\plus{D_0}D_1^{\otimes k};\Phi^{m,k}]

\morphism(0,-2000)|a|/@{>}@<0pt>/<1500,0>[\Sigma^m D_{\bullet} \coprod \Sigma^m D_{\bullet} ` D_1^{\otimes q}\plus{D_0}\Sigma^{m}D_{\bullet} \plus{D_{m}}D_{m+1}\plus{D_0}D_1^{\otimes k};\theta^{q,m,k}]
\morphism(0,-2000)|a|/@{>}@<0pt>/<0,-500>[\Sigma^m D_{\bullet} \coprod \Sigma^m D_{\bullet} `\Sigma^m \cyl (D_{\bullet});\Sigma^m(\iota)]
\morphism(0,-2500)|r|/@{>}@<0pt>/<1500,500>[\Sigma^m \cyl (D_{\bullet}) ` D_1^{\otimes q}\plus{D_0}\Sigma^{m}D_{\bullet} \plus{D_{m}}D_{m+1}\plus{D_0}D_1^{\otimes k};\Theta^{q,m,k}]
\efig
\end{equation}
\end{defi}
For example, this means that given a map \[(f_1,\ldots,f_m,\alpha,g_1,\ldots,g_k)	\colon D_1^{\otimes m}\plus{D_0} D_{n+1}\plus{D_0} D_1^{\otimes k} \rightarrow X\] we get an $n$-cylinder in $\Omega(X,s(F_1),t(g_k))$ of the form: \[\Psi^{m,k}(f_1,\ldots,f_m,\alpha,g_1,\ldots,g_k)\colon g_k\ldots g_1 (\alpha f_m)f_{m-1}\ldots f_1 \curvearrowright g_k\ldots g_2 (g_1 \alpha) f_mf_{m-1}\ldots f_1\] and similarly for the other cases.

We are now ready to define vertical composition of cylinders. Let us list the inductive hypotheses we need to define it on $(n+1)$-cylinders:
\begin{itemize}
	\item existence of a coglobular vertical composition of $n$-cylinders (i.e. compatible with source and target maps) and defined as described in what follows;
	\item given \[\begin{cases} \ m \geq 0, \ g\colon S^m \rightarrow X, \ g_{|S^0}=(b,c) \ \text{and an} \  n\text{-cylinder} \ F\colon A\curvearrowright B \ \text{in} \ \Omega^{m+1}(X,g), \\
	\text{1-cells} \ 	h_i \colon  b_i \rightarrow b_{i+1} \ \text{in} \ X, \ 1\leq i \leq q, b_{q+1}=b\\
	\text{1-cells}\ 	f_j \colon c_j \rightarrow c_{j+1} \ \text{in} \ X, \ 1\leq j \leq k, c_{1}=c	
	\end{cases}\]
	there is an $n$-cylinder \[f_k \ldots f_1Fh_q\ldots h_1\colon f_k \ldots f_1 A h_q\ldots h_1\curvearrowright f_k \ldots f_1 B h_q\ldots h_1\]
	in $\Omega^{m+1}(X,f_k \ldots f_1*g*h_q\ldots h_1)$, where we define (for $m >0$ and $g=(g_0,g_1)$) \[f_k \ldots f_1*g*h_q\ldots h_1 =(f_k \ldots f_1 g_0 h_q\ldots h_1,f_k \ldots f_1 g_1 h_q\ldots h_1)\] and we let juxtaposition represent the result of composing cells using the appropriate operation $\gamma$.
\end{itemize} 
If $m=0$ then $g=(b,c)\colon S^0 \rightarrow X$, and we define
\[f_k \ldots f_1*g*h_q\ldots h_1   =(s(h_1),t(f_k))\]
Let us now consider the case $n=0$. A vertical stack of $0$-cylinders is nothing but a string of composable 1-cells, which we compose using the appropriate choice of $\gamma$.
Explicitly, given $0$-cylinders $(F^i\colon x_i \curvearrowright x_{i+1})_{i \in \{1,\ldots,n\}}$, i.e. $1$-cells $F^i\colon x_i \rightarrow x_{i+1}$ in $X$, we define \[F^n \otimes \ldots \otimes \ F^1  =F^n\ldots F^1\colon x_1 \curvearrowright x_{n+1}\]
With the purpose of addressing the second point in the inductive hypotheses, note that a $0$-cylinder $F\colon A\curvearrowright B \ \text{in} \ \Omega^{m+1}(X,g)$ is just an $(m+2)$-cell in $X$.
Therefore, we simply define the required $0$-cylinder by
\[f_k \ldots f_1Fh_q\ldots h_1\colon f_k \ldots f_1 A h_q\ldots h_1\curvearrowright f_k \ldots f_1 B h_q\ldots h_1\] where, as usual, juxtaposition means the result of composing those cells using the appropriate operation $\gamma$

Turning to the inductive step, given $p>0$ and $(n+1)$-cylinders $F^i\colon A_i \curvearrowright A_{i+1}$ in $X$ for $1\leq i \leq p$ we want to define an $(n+1)$-cylinder \[F^p\otimes \ldots \otimes F^1 \colon A_1 \curvearrowright A_{p+1}\] in a way that is compatible with the already defined composition on lower dimensions.
We can express the cylinders $F^i$ in an equivalent way, by considering them as $n$-cylinders \[\bar{F}^i\colon A_{i+1}F_{s_0}^i \curvearrowright F_{t_0}^i A_i\]  in $X(s(F_{s_0}^i),t(F_{t_0}^i))$.
We define the $(n+1)$-cylinder $F^p\otimes \ldots \otimes F^1$ by setting \[(F^p\otimes \ldots \otimes F^1)_{\epsilon_0}={F^p}_{\epsilon_0} \ldots  {F^1}_{\epsilon_0}\] for $\epsilon=s,t$, using $\gamma$ to compose these $1$-cells, and defining $\overline{F^p\otimes \ldots \otimes F^1}$ to be the vertical composition of the following sequence of $n$-cylinders in $X(s(F^1_{s_0}),t(F^p_{t_0}))$

\[
\vcenter{\hbox{\xymatrix{
			(F^{p}_{t_0}\ldots  F^{1}_{t_0}) A_0 \ar@/^3mm/[d]^-{\Psi^{0,p}(F^{p}_{t_0},\ldots,  F^{1}_{t_0},A_0)}\\F^{p}_{t_0}\ldots  F^{1}_{t_0}(F_{t_0}^0 A_0)
			\ar@/^3mm/[d]^-{F^p F^{p-1}_{s_0}\ldots  F^{1}_{s_0}}\\F^{p}_{t_0}\ldots  F^{1}_{t_0}(A_1 F^1_{s_0})
			\ar@/^3mm/[d]^-{\Psi^{1,p-1}(F^{p}_{t_0},\ldots, F^{1}_{t_0},A_{1},F^{1}_{s_0})}\\F^{p}_{t_0}\ldots  (F^{1}_{t_0}A_1)	 F^1_{s_0}
			\ar@/^3mm/[d]^-{}\\\ldots
			\ar@/^3mm/[d]^-{\Psi^{p-1,1}(F^{p}_{t_0},A_p,F^{p-1}_{s_0},\ldots,   F^{1}_{s_0})}\\
			(F^{p}_{t_0}A_{p})F^{p-1}_{s_0}\ldots F^{1}_{s_0}
			\ar@/^3mm/[d]^-{F^pF^{p-1}_{t_0}\ldots  F^{1}_{t_0}}\\(A_{p+1}F^{p}_{s_0})F^{p-1}_{s_0}\ldots F^{1}_{s_0}
			\ar@/^3mm/[d]^-{\Psi^{p,0}(A_{p+1},F^{p}_{s_0},\ldots ,F^{1}_{s_0})}\\A_{p+1}(F^{p}_{s_0}\ldots F^{1}_{s_0})
}}}
\]

Let us now address the second part of the inductive hypothesis.

The data are the following:
\[\begin{cases} m \geq 0, \ g\colon S^m \rightarrow X, \ g_{|S^0}=(b,c), \ \text{an} \  (n+1)\text{-cylinder} \ F\colon A\curvearrowright B \ \text{in} \ \Omega^{m+1}(X,g),  \\
 \text{1-cells in} \ h_i\colon b_i \rightarrow b_{i+1} \ \text{in} \ X, \ 1\leq i \leq q, b_{q+1}=b\\
\text{1-cells in} \ f_j\colon c_j \rightarrow c_{j+1} \ \text{in} \ X, \ 1\leq j \leq k, c_{k+1}=c	
\end{cases}\]
View $F$ as an $n$-cylinder 
\[\bar{F}\colon BF_{s_0}\curvearrowright F_{t_0}A \  \text{in} \ \Omega(\Omega^{m+1}(X,f),s^{n+1}A,s^{n+1}B)\cong \Omega^{m+2}(X,\phi) \] where we set $\phi\colon=(s^{n+1}A,s^{n+1}B)\colon S^{m+1}\rightarrow X$.

By inductive hypothesis we can construct an $n$-cylinder
\[f_k\ldots f_1 \bar{F} h_q \ldots h_1\colon f_k\ldots f_1 (BF_{s_0}) h_q \ldots h_1\curvearrowright f_k\ldots f_1 (F_{t_0}A) h_q \ldots h_1\]
in $\Omega^{m+2}(X,f_k \ldots f_1*\phi*h_q \ldots h_1)$, which is isomorphic to \[ \Omega(\Omega^{m+1}(X,f_k \ldots f_1*g*h_q \ldots h_1),f_k \ldots f_1*A_{s_0}*h_q \ldots h_1,f_k \ldots f_1*B_{t_0}*h_q \ldots h_1) \]
Finally, we define $f_k\ldots f_1 F h_q \ldots h_1$ by setting \[(f_k\ldots f_1 F h_q \ldots h_1)_{\epsilon_0}=f_k\ldots f_1 F_{\epsilon_0} h_q \ldots h_1 \] for $\epsilon=s,t$, using $\gamma$ to compose these $1$-cells, and defining $\overline{f_k\ldots f_1 F h_q \ldots h_1}$ to be the vertical composition of the following sequence of $n$-cylinders in $\Omega^{m+2}(X,f_k \ldots f_1*\phi*h_q \ldots h_1)$:
\[
\vcenter{\hbox{\xymatrix{
			(f_k\ldots f_1 B h_q \ldots h_1)(f_k\ldots f_1 F_{s_0} h_q \ldots h_1)\ar@/^3mm/[d]^-{\Phi^{q,m+2,k}(f_k,\ldots, f_1,B,F_{s_0},h_q,\ldots h_1)}\\f_k\ldots f_1 (BF_{s_0}) h_q \ldots h_1
			\ar@/^3mm/[d]^-{f_k\ldots f_1 \bar{F} h_q \ldots h_1}\\f_k\ldots f_1 (F_{t_0}A) h_q \ldots h_1
			\ar@/^3mm/[d]^-{\Theta^{q,m+2,k}(f_k,\ldots, f_1,F_{t_0},A,h_q,\ldots h_1)}\\(f_k\ldots f_1 F_{t_0} h_q \ldots h_1)(f_k\ldots f_1 A h_q \ldots h_1)
}}}
\]
This completes the induction.
\subsection{Naive elementary interpretation of operations}
We now describe how to define $\hat{\rho}\colon \cyl(D_n) \rightarrow \cyl(A)$ in a naive way, that will satisfy the first condition in \eqref{hat properties} but not the second. The next section will then address and solve the problem of also satisfying the second condition.
Let $p=\vert \mathscr{L}(A)\vert$, and $B_1, \ldots, B_p$ the ordered list $\mathscr{L}(A)$. As in the previous section, we will obtain this cylinder as the vertical composite of a stack of $p \ (n-1)$-cylinders in $\Omega \left(\cyl(A),a,b\right)$, where $a=\cyl(\partial_{\sigma}^m) \circ \sigma$ and $b=\cyl(\partial_{\tau}^m)\circ \tau$.

First, we specify a sequence of $n$-cells $(\alpha_0, \ldots, \alpha_p)$ in $\cyl(A)$ that appear as top and bottom cells of the cylinders in the stack. The first and the last are thus forced by the requirement that the vertical composite $(n-1)$-cylinder in the hom-$\infty$-groupoid between $a$ and $b$ is the transpose of an actual $n$-cylinder in $\cyl(A)$ satisfying the first of the conditions expressed in \eqref{hat properties}. Therefore, $\alpha_0$ must be given by the composite
\[
\bfig 
\morphism(0,0)|a|/@{>}@<0pt>/<600,0>[D_n` D_n \plus{ D_0} D_1;{}_{D_n}w]
\morphism(600,0)|a|/@{>}@<0pt>/<1100,0>[D_n \plus{ D_0} D_1` \cyl(A) ;(\iota_0 \circ \rho, \cyl(\partial^m_\tau))]
\efig 
\]
\[\alpha_0=(\iota_0 \circ \rho,\cyl(\partial_{\tau}^m)) \circ w\] where $m=\dim(A)$.
Similarly, $\alpha_p$ must be defined to be the composite
\[
\bfig 
\morphism(0,0)|a|/@{>}@<0pt>/<600,0>[D_n` D_1 \plus{ D_0} D_n;w_{D_n}]
\morphism(600,0)|a|/@{>}@<0pt>/<1100,0>[D_1\plus{ D_0} D_n` \cyl(A) ;(\cyl(\partial^m_\sigma),\iota_1 \circ \rho )]
\efig 
\]
For $1\leq i \leq p-1$, we define $\alpha_i\colon D_n \rightarrow \cyl (A)$ as the following composite
\begin{equation}
\label{alpha cells}
\bfig
\morphism(0,0)|a|/@{>}@<0pt>/<300,0>[D_n` A;\rho]
\morphism(300,0)|a|/@{>}@<0pt>/<300,0>[A` B_i;z^A_{B_i}]
\morphism(600,0)|a|/@{>}@<0pt>/<500,0>[B_i` \cyl (A);i_{B_i}]
\efig 
\end{equation}
Notice that they all transpose under the adjunction $\Sigma \dashv \Omega$ to give $(n-1)$-cells $\alpha_0, \ldots,\alpha_p$ in $\Omega \left(\cyl(A),a,b\right)$, where $a=\cyl(\partial_{\sigma}^m) \circ \sigma$ and $b=\cyl(\partial_{\tau}^m)\circ \tau$. In addition, by construction, $\alpha_{i-1}$ and $\alpha_i$ factor through $\Omega (B_i,a,b)$, where, for every $B_i \in \mathscr{L}(A)$, we denote the endpoints of this globular sum (i.e. the $0$-cells $\partial_{\sigma}^{\dim(B_i)},\partial_{\tau}^{\dim(B_i)}\colon D_0 \rightarrow B_i$) with $a$ and $b$, committing a slight abuse of language since they are all sent to $a$ and $b$ in $\cyl (A)$ by the maps $i_{B_i}$).
\begin{ex}
	Let $A=D_2 \amalg_{ D_0}D_1$ and consider a homogeneous map  $\rho\colon D_2\rightarrow A$. This operation may represent, for instance, the whiskering of a $2$-cell with a $1$-cell.
	
	Thanks to Example \ref{ex cyl decomp}, if we follow the algorithm explained above we find that the cells in the list we have to provide are given by the transposes of:
	
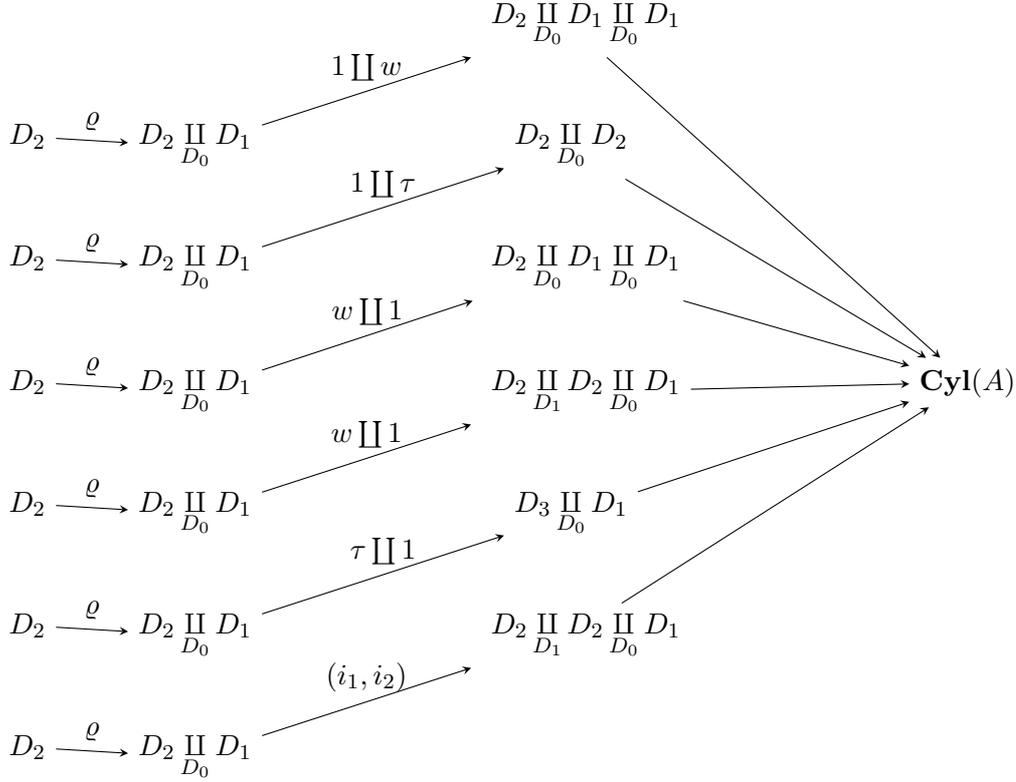
\[
	\begin{tikzpicture}
	\matrix (m) [matrix of math nodes,row sep=2em,column sep=2.5em,minimum width=1.5em]
	{
		&	& & \ \ \ D_2\plus{D_0}D_1 \plus{D_0} D_1 \\
		D_2&	D_2 \plus{D_0} D_1 & 
		&	 D_2 \plus{D_0} D_2 \\
		D_2&	D_2 \plus{D_0} D_1&
		&	  \ \ \ D_2 \plus{ D_0} D_1 \plus{D_0} D_1\\
		D_2&	D_2 \plus{D_0} D_1 &\
		&	 \ \ \ D_2 \plus{D_1} D_2 \plus{D_0} D_1 & & &\cyl(A)\\
		D_2&	D_2 \plus{D_0} D_1 &
		&	 D_3 \plus{D_0} D_1\\
		D_2&	D_2 \plus{D_0} D_1 &
		&	 \ \ \ D_2 \plus{D_1}D_2 \plus{D_0}D_1\\
		D_2&	D_2 \plus{D_0}D_1&\\};
	\path[-stealth]
	(m-2-2) edge node [above] {$1\coprod w$} (m-1-4)
	(m-3-2)edge  node [above] {$1 \coprod \tau$} (m-2-4)
	
	(m-4-2)edge  node [above] {$w \coprod 1$} (m-3-4)
	
	(m-5-2)edge  node [above] {$w \coprod 1$} (m-4-4)
	
	(m-6-2)edge  node [above] {$\tau \coprod 1$} (m-5-4)
	
	(m-7-2)edge  node [above] {$(i_1,i_2)$} (m-6-4)
	
	(m-2-1) edge node [above] {$\rho$} (m-2-2)
	(m-3-1)edge  node [above] {$\rho$} (m-3-2)
	
	(m-4-1)edge  node [above] {$\rho$} (m-4-2)
	
	(m-5-1)edge  node [above] {$\rho$} (m-5-2)
	
	(m-6-1)edge  node [above] {$\rho$} (m-6-2)
	
	(m-7-1)edge  node [above] {$\rho$} (m-7-2)
	
	(m-1-4) edge node [above] {} (m-4-7)
	(m-2-4)edge  node [above] {} (m-4-7)
	
	(m-3-4)edge  node [above] {} (m-4-7)
	
	(m-4-4)edge  node [above] {} (m-4-7)
	
	(m-5-4)edge  node [above] {} (m-4-7)
	
	(m-6-4)edge  node [above] {} (m-4-7);
	\end{tikzpicture}\]
	together with the transpose of the composite 
	\[\bfig  
	\morphism(0,0)|a|/@{>}@<0pt>/<500,0>[D_2`D_2 \plus{ D_0} D_1;w]
	\morphism(500,0)|a|/@{>}@<0pt>/<800,0>[D_2 \plus{ D_0} D_1`A\plus{ D_0}D_1;\rho \plus{ D_0}1]
	\morphism(1300,0)|a|/@{>}@<0pt>/<850,0>[A \plus{ D_0} D_1`\cyl(A);(\iota_0, \cyl(\partial_{\tau}^2))]
	\efig \]
	as the first cell of the list, and of the composite
	
	\[\bfig  
	\morphism(0,0)|a|/@{>}@<0pt>/<500,0>[D_2`D_2 \plus{ D_0} D_1;w]
	\morphism(500,0)|a|/@{>}@<0pt>/<800,0>[D_2 \plus{ D_0} D_1`D_1\plus{ D_0}A;1 \plus{ D_0}\rho]
	\morphism(1300,0)|a|/@{>}@<0pt>/<850,0>[D_1 \plus{ D_0} A`\cyl(A);(\cyl(\partial_{\sigma}^2),\iota_1)]
	\efig \]
	as the last.
\end{ex}
We can now define a stack of $(n-1)$-cylinders $C_i^{\rho}\colon \alpha_{i-1} \curvearrowright \alpha_{i}$ in $\Omega  \left(\cyl(A),a,b\right)$, for $1\leq i \leq p$, such that $C_i^{\rho}$ factors through $\Omega (B_i ,a,b)$.

We do so by solving the following lifting problems, thanks to Propositions \ref{omega preserves contractibles} and \ref{boundary of deg cyl's}
\[ \bfig
\morphism(0,0)|a|/@{>}@<0pt>/<1000,0>[D_{n-1}\coprod D_{n-1}`\Omega (B_i ,a,b);(\alpha_{i-1},\alpha_{i})]
\morphism(0,0)|a|/@{>}@<0pt>/<0,-500>[D_{n-1}\coprod D_{n-1}`\cyl(D_{n-1});\partial]
\morphism(0,-500)|r|/@{-->}@<0pt>/<1000,500>[\cyl(D_{n-1})`\Omega (B_i,a,b);]
\morphism(1000,0)|a|/@{>}@<0pt>/<800,0>[\Omega (B_i ,a,b)`\Omega \left(\cyl (A),a,b\right);\Omega( i_{B_i})]
\morphism(0,-500)|r|/@{>}@<0pt>/<1800,500>[\cyl(D_{n-1})`\Omega\left(\cyl (A),a,b\right);C_i^{\rho}]
\efig \]
Finally, set $\hat{\rho}$ as the $n$-cylinder induced by the transpose of $C_p^{\rho} \otimes \ldots \otimes C_1^{\rho}\colon \alpha_0 \curvearrowright \alpha_p$, viewed as a map $\Sigma \cyl(D_{n-1}) \rightarrow \cyl(A)$. It is now straightforward to check that, in general, this definition satisfies only the first condition of \eqref{hat properties}, essentially because we have no control on what happens to the boundary of $\hat{\rho}$ in relation to $\widehat{\rho\circ \epsilon}$ for $\epsilon=\sigma, \tau$.
\section{Degenerate Cylinders}
In this section we want to define cylinders whose iterated source or target are degenerate, in a suitable sense.
We will also extend the operation of vertical compositions to this more general setting, as it will be needed later to construct the ``correct'' elementary interpretation of a homogeneous operation, i.e. the one satisfying both conditions expressed in \eqref{hat properties}.
\begin{defi}
	\label{free deg cyl}
	Let $n>0$. Define the $\infty$- groupoid $\cyl^0_{-1} (D_n)=\cyl^0 (D_n)$ as the colimit of the following diagram:
	\[
	\bfig
	\morphism(-1000,-250)|a|/@{>}@<0pt>/<650,250>[D_n` D_n \plus{D_0} D_1;w]
	\morphism(-1000,-250)|l|/@{>}@<0pt>/<650,-250>[D_n` \Sigma \cyl(D_n-1);\Sigma (\iota_0)]
	\efig
	\]
	Similarly, define $\cyl_0^{-1} (D_n)=\cyl_0 (D_n)$ as the colimit of the following diagram:
	\[
	\bfig
	\morphism(-1000,-250)|a|/@{>}@<0pt>/<650,250>[D_n` D_1 \plus{D_0} D_n;w]
	\morphism(-1000,-250)|l|/@{>}@<0pt>/<650,-250>[D_n` \Sigma \cyl(D_n-1);\Sigma (\iota_1)]
	\efig
	\]
	Also set $\cyl_0^0 (D_n)=\Sigma \cyl(D_{n-1})$.
	
	Finally, given $0 < p,q <n$ with $\vert p-q \vert \leq 1$, define inductively \[\cyl^p_{q}(D_n)= \Sigma \cyl^{p-1}_{q-1}(D_{n-1})\] 
	We call $\cyl^p_q$ the n-cylinder with degenerate $p$-source and degenerate $q$-target.
\end{defi}
It is clear that all these cylinders come equipped with maps $\iota_0, \iota_1\colon  D_n \rightarrow \cyl^p_{q}(D_n)$.
\begin{defi}
	\label{deg cyl}
	Given a pair of $n$-cells $\alpha,\beta$ in $X$ and integers $0\leq p,q <n$ as above, an $n$-cylinder in $X$ from $\alpha$ to $\beta$ with degenerate $p$-source and degenerate $q$-target is a map \[C\colon \cyl^p_{q}(D_n) \rightarrow X\] such that $C \circ \iota_0=\alpha$ and $C \circ \iota_1 = \beta$.
	We will denote it by $C\colon \alpha \curvearrowright_q^p \beta$
\end{defi}
\begin{rmk}
	Notice that a cylinder $C\colon \alpha \curvearrowright_q^p \beta$ exists only if $s^p(\alpha) = s^p(\beta)$ and $t^q(\alpha)=t^q(\beta)$.
\end{rmk}
To describe these data explicitly, we need to distinguish between cases.\\
If $p=0$ and $q=-1$ then it consists of:
\begin{itemize}
	\item a $1$-cell $c\colon t^n (\alpha) \rightarrow t^n(\beta)$;
	\item an $(n-1)$-cylinder $\bar{C}\colon c \alpha \curvearrowright \beta$ in $\Omega \left( X,s^n (\alpha), t^n (\beta)\right)$.
	
\end{itemize}
If $p=-1$ and $q=0$ then it consists of:
\begin{itemize}
	\item a $1$-cell $c\colon s^n (\alpha) \rightarrow s^n(\beta)$;
	\item an $(n-1)$-cylinder $\bar{C}\colon \alpha \curvearrowright \beta c$ in $\Omega \left(X,s^n (\alpha),t^n (\beta)\right)$.
\end{itemize}
If $p,q >0$ then it consists of:
\begin{itemize}
	
	\item an $(n-1)$-cylinder with degenerate $(p-1)$-source and degenerate $(q-1)$-target $\bar{C}\colon \alpha \curvearrowright^{p-1}_{q-1} \beta$ in $\Omega \left(X,s^n (\alpha),t^n (\beta)\right)$.
\end{itemize}
\begin{defi}
	Let $p,q\geq -1$ be integers such that $\vert p-q \vert \leq 1$. We define the category $\G_q^p$ as the full subcategory of $\G$ generated by:
	\begin{itemize}
		\item $\G_{\geq p+1}$ if $p=q$;
		\item $\G_{\geq p}$ and $\tau\colon q\rightarrow p$ if $q=p-1$;
		\item $\G_{\geq q}$ and $\sigma\colon p\rightarrow q$ if $p=q-1$;
	\end{itemize}
	Clearly, the direct category structure on $\G$ restricts to one on $\G^p_q$, and we can extend the previous construction to a functor   \[\cyl^p_q(D_{\bullet})\colon \G_q^p \rightarrow \wgpd\]
	Given such $p,q$ we also construct a functor $B_q^p\colon\G^p_q \rightarrow \wgpd$ by defining 
	\[\bfig
	\morphism(-500,-250)|a|/@{}@<0pt>/<750,0>[B_q^p(n)\cong \colim` ;  ]
	
	\morphism(0,0)|a|/@{>}@<0pt>/<500,0>[D_p` D_n;\sigma ]
	\morphism(0,0)|a|/@{>}@<0pt>/<0,-500>[D_p`D_{n};\sigma]
	\morphism(500,-500)|a|/@{>}@<0pt>/<0,500>[D_{q}` D_n;\tau]
	\morphism(500,-500)|a|/@{>}@<0pt>/<-500,0>[D_q` D_n;\tau]
	\efig
	\]
	where we set $D_{-1}=\emptyset$, the initial object of $\wgpd$.
	
	For each such pair of integers we get a natural transformation \[B_q^p \rightarrow \cyl_q^p (D_{\bullet}) \]  induced by $\iota_0,\iota_1:D_n \rightarrow  \cyl^p_q (D_n)$.
\end{defi}
\begin{ex}
	A $1$-cylinder with degenerate $0$-source in $X$, represented by a map $C:\cyl^0(D_1) \rightarrow X$, consists of specifying the following data
	\[ 
	\bfig
	\morphism(0,400)|a|/@{>}@<0pt>/<400,0>[a` b;\alpha]
		\morphism(0,400)|a|/@{=}@<0pt>/<0,-400>[a` a;]
	\morphism(0,0)|b|/@{>}@<0pt>/<400,0>[a` c;\beta]
	\morphism(400,400)|r|/@{>}@<0pt>/<0,-400>[b` c;g]
	
	\morphism(200,300)|a|/@{=>}@<0pt>/<0,-200>[`;C]
	\efig
	\] by which we mean a $2$-cell $C\colon g \alpha \rightarrow \beta$.
	In this case, the pair $(\alpha,\beta)$ represents the natural map $(B^0)_1 \rightarrow \cyl^0 (D_1)$.
\end{ex}
\subsection{Boundary of degenerate cylinders}
As we did for normal cylinders, we will construct a map of diagrams indexed by a direct category, with codomain $\cyl^p_q(D_{\bullet})$, whose latching maps will represent the inclusion of the boundary of a degenerate cylinder.
This construction will be fundamental to perform inductive constructions involving cylinders.
\begin{prop}
	\label{boundary of deg cyl's}
	The map \[B_q^p \rightarrow \cyl_q^p (D_{\bullet}) \]  is a direct cofibration in $\wgpd^{\G^p_q}$.
	\begin{proof}
		If $p,q\geq0$, we have \[B_q^p \cong \Sigma B_{q-1}^{p-1} \ \text{and} \ \cyl_q^p \cong \Sigma \cyl_{q-1}^{p-1}\] and the map $B_q^p \rightarrow \cyl_q^p$ results from applying $\Sigma$ to $B_{q-1}^{p-1} \rightarrow \cyl_{q-1}^{p-1}$. Therefore, since $\Sigma$ preserves cofibrations, it is enough to prove the result for $p=0,q=-1$ and $p=-1,q=0$.
		
	Let's consider $\cyl^0(D_{\bullet})$.
		Consider the following cocartesian square (where $n>1$):
		\[
		\bfig

		\morphism(0,0)|a|/@{>}@<0pt>/<0,-500>[S^{n-1}` \partial \cyl(D_{n-1});]
		\morphism(0,0)|a|/@{>}@<0pt>/<750,0>[S^{n-1}` D_n;]
		\morphism(0,-500)|a|/@{>}@<0pt>/<750,0>[\partial \cyl(D_{n-1})` \cyl(D_{n-1});\hat{L}_{n-1}(\iota)]
		\morphism(750,0)|a|/@{>}@<0pt>/<0,-500>[D_n` \cyl(D_{n-1});]
		\efig
		\]
		If we apply the functor $\Sigma$ to it, we get the following cocartesian square thanks to Remark \ref{cocontinuity of UoSIgma}:
		
		\[
		\bfig
		\morphism(0,0)|a|/@{>}@<0pt>/<0,-500>[S^{n}`\Sigma\partial \cyl(D_{n-1});]
		\morphism(0,0)|a|/@{>}@<0pt>/<1200,0>[S^{n}` D_{n+1};]
		\morphism(0,-500)|a|/@{>}@<0pt>/<1200,0>[\Sigma \partial \cyl(D_{n-1})` \Sigma \cyl(D_{n-1});\Sigma\left(\hat{L}_{n-1}(\iota)\right)]
		\morphism(1200,0)|a|/@{>}@<0pt>/<0,-500>[D_{n+1}` \Sigma \cyl(D_{n-1});]
		\efig
		\]
		From this, we derive the following diagrams where both squares are cocartesian:
		
		\[
		\bfig
		\morphism(0,0)|a|/@{>}@<0pt>/<750,0>[D_n` \Sigma \partial \cyl(D^{n-1}); ]
		\morphism(0,0)|a|/@{>}@<0pt>/<0,-500>[D_n`D_{n} \plus{D_0} D_{1};w]
		\morphism(0,-500)|a|/@{>}@<0pt>/<750,0>[D_{n} \plus{D_0} D_{1}` \partial \cyl^0(D_{n});]
		\morphism(750,0)|a|/@{>}@<0pt>/<0,-500>[\Sigma \partial \cyl(D^{n-1})` \partial \cyl^0(D_{n});]
		
		\morphism(750,0)|a|/@{>}@<0pt>/<1200,0>[\Sigma \partial \cyl(D^{n-1})`\Sigma  \cyl(D^{n-1});\Sigma (\hat{L}_{n-1}(\iota)) ]
		\morphism(750,-500)|a|/@{>}@<0pt>/<1200,0>[\partial \cyl^0(D_{n})`   \cyl^0(D^{n});i]
		\morphism(1950,0)|a|/@{>}@<0pt>/<0,-500>[\Sigma  \cyl(D^{n-1})` \cyl^0(D^{n});]
		
		\efig
		\]
		It is clear that the $n$-th latching map of the map $B^0 \rightarrow \cyl^0(D_{\bullet})$ is given by \[i\colon \partial \cyl^0(D_{n}) \rightarrow \cyl^0(D_{n})\] and is thus a cofibration. Therefore the natural map $B^0 \rightarrow \cyl^0(D_{\bullet})$ is a direct cofibration of $\G^0$-diagrams in $\wgpd$.
		A similar argument shows that also $B_0 \rightarrow \cyl_0(D_{\bullet})$ is a direct cofibration of $\G_0$-diagrams in $\wgpd$.		
	\end{proof}
\end{prop}
\begin{defi}
	Given non-negative integers $p,q$ such that $\vert p-q \vert\leq 1$, we call $\partial \cyl^p_q (D_n)$ the boundary of the $n$-cylinder with degenerate $p$-source and $q$-target.
	
	Given an $\infty$-groupoid $X$ and an $n$-cylinder with degenerate $p$-source and $q$-target in $X$, represented by a map $C:\cyl^p_q (D_n) \rightarrow X$, we call the boundary of $C$ the map we get by precomposing $C$ with the boundary inclusion 
	\[
	\bfig
	\morphism(0,0)|a|/@{>->}@<0pt>/<600,0>[ \partial \cyl^p_q (D_n)`    \cyl^p_q (D_n);] \morphism(600,0)|a|/@{>}@<0pt>/<500,0>[ \cyl^p_q (D_n)`   X;C]
	\efig
	\]
\end{defi}
From the previous analysis we see that, given an $\infty$-groupoid $X$, specifying the boundary of a degenerate $n$-cylinder $\cyl^p_q(D_n)\rightarrow X$ is equivalent to providing the following data:
\begin{itemize}
	\item a pair of parallel $(n-1)$-cylinders $C\colon A\curvearrowright^{p-1}_{q-1} B,D\colon	A'\curvearrowright^{p-1}_{q-1} B'$ in $X$;
	\item a pair of $n$-cells $\alpha\colon A \rightarrow A', \beta \colon B \rightarrow B'$ in $X$.
\end{itemize}
\subsection{Vertical composition of degenerate cylinders}
\label{vert comp of deg cyl (section)}
We now want to define an operation of vertical composition that generalizes the one we already have to the case of a vertical stack of (possibly) degenerate cylinders.

To do so, assume given a $k$-tuple of pairs of integers $((p_i,k_i))_{1\leq i \leq k}$, with $\vert p_i - q_i \vert \leq 1$ for each $1\leq i \leq k$.
This operation is represented by a map:
\begin{equation}
\label{vert comp deg cyl}
\bfig
\morphism(750,0)|a|/@{>}@<0pt>/<1250,0>[\cyl_q^p (D_n)`\cyl_{q_1}^{p_1}(D_n) \otimes \ldots \otimes \cyl_{q_k}^{p_k}(D_n) ;]

\efig 
\end{equation} 
where the codomain is defined to be the colimit of the following diagram:

\[ 
\bfig

\morphism(0,-500)|l|/@{>}@<0pt>/<-300,300>[D_n`\cyl_{q_1}^{p_1}(D_n);\iota_1]
\morphism(0,-500)|r|/@{>}@<0pt>/<300,300>[D_n`\cyl_{q_2}^{p_2}(D_n);\iota_0]

\morphism(600,-500)|l|/@{>}@<0pt>/<-300,300>[D_n`\cyl_{q_2}^{p_2}(D_n);\iota_1]  \morphism(600,-500)|r|/@{>}@<0pt>/<300,300>[D_n`\ldots;\iota_0]

\morphism(1200,-500)|l|/@{>}@<0pt>/<-300,300>[D_n`\ldots;\iota_1]  \morphism(1200,-500)|r|/@{>}@<0pt>/<300,300>[D_n`\cyl_{q_k}^{p_k}(D_n);\iota_0]
\efig
\]
and $p=\min \{p_i\}_{1\leq i \leq k},\ q=\min \{q_i\}_{1\leq i \leq k}$.

We will adapt the construction we already have for the case $p_i=q_i=-1$.
Therefore, we firstly neeed to define whiskerings of degenerate cylinders with 1-cells.
We have to treat the different cases separately, the inductive hypothesis for the general case being:
\begin{itemize}
	\item existence of a coglobular vertical composition of possibly degenerate  $(n-1)$-cylinders (i.e. compatible with source and target maps) as in \eqref{vert comp deg cyl} and defined as described in what follows;
	\item given
	\[\begin{cases} m \geq 0, \ g\colon S^m \rightarrow X, \ g_{|S^0}=(b,c) \ \text{an}\ n\text{-cylinder} \ C\colon A\curvearrowright^p_q B \ \text{in} \ \Omega^{m+1}(X,g)   \\
	1\text{-cells}\ 	h_i\colon b_i \rightarrow b_{i+1} \  \text{in}\ X, \ 1\leq i \leq q,\ b_{q+1}=b\\
		1\text{-cells} \ 	f_j\colon c_j \rightarrow c_{j+1} \  \text{in}\ X, \ 1\leq j \leq k, \ c_{1}=c	
	\end{cases}\]
	there is an $n$-cylinder \[f_k \ldots f_1Fh_q\ldots h_1\colon f_k \ldots f_1 A h_q\ldots h_1\curvearrowright^p_q f_k \ldots f_1 B h_q\ldots h_1\]
	in $\Omega^{m+1}(X,f_k \ldots f_1*g*h_q\ldots h_1)$, where we define (for $m >0$) \[f_k \ldots f_1*g*h_q\ldots h_1 =(f_k \ldots f_1 g_0 h_q\ldots h_1,f_k \ldots f_1 g_1 h_q\ldots h_1)\] if $g=(g_0,g_1)$.
\end{itemize}
Once we have constructed the whiskerings of the second point, the rest of the proof follows just by adapting the one for normal cylinders, omitting the use of the $\Psi$'s when no rebracketing is needed.

Let us start with the case $p=0, \ q=-1$.
By definition, $C$ induces an $(n-1)$-cylinder \[\bar{C}\colon cA \curvearrowright B \ \text{in} \ \Omega^{m+2} (X,\phi)\] where $\phi=(s^n(A),t^n(B))$ and $c=\cyl(\tau^n)(C)$.
Because we already know how to whisker normal cylinders with $1$-cells, we get an $(n-1)$-cylinder \[f_k\ldots f_1 \bar{C} h_q \ldots h_1\colon f_k\ldots f_1 (cA) h_q \ldots h_1\curvearrowright f_k\ldots f_1 B h_q \ldots h_1\]
We now define \[f_k\ldots f_1 C h_q \ldots h_1\colon f_k\ldots f_1 A h_q \ldots h_1 \curvearrowright^0 f_k\ldots f_1 B h_q \ldots h_1\] as the vertical composition of the following cylinders in $\Omega^{m+2}(X,f_k \ldots f_1*\phi*h_q \ldots h_1)$
\[
\vcenter{\hbox{\xymatrix{
			(f_k\ldots f_1 c \  h_q \ldots h_1)(f_k\ldots f_1 A h_q \ldots h_1)\ar@/^3mm/[d]^-{\Phi^{q,m+2,k}(f_k,\ldots, f_1,c,A,h_q,\ldots h_1)}\\f_k\ldots f_1 (cA) h_q \ldots h_1
			\ar@/^3mm/[d]^-{f_k\ldots f_1 \bar{C} h_q \ldots h_1}\\f_k\ldots f_1 B h_q \ldots h_1
}}}
\]
In a completely analogous way, if $p=-1, \ q=0$, we obtain \[f_k\ldots f_1 C h_q \ldots h_1\colon f_k\ldots f_1 A h_q \ldots h_1 \curvearrowright_0 f_k\ldots f_1 B h_q \ldots h_1\] in $\Omega^{m+2}(X,f_k \ldots f_1*\phi*h_q \ldots h_1)$.

The case $C\colon A\curvearrowright_0^0 B \ \text{in} \ \Omega^{m+1}(X,g)$ is even simpler, because we simply define the whiskering as 
\[\overline{f_k\ldots f_1 C h_q \ldots h_1}=f_k\ldots f_1\bar{C} h_q \ldots h_1 \]
This conclude the base case of the induction.

Finally, let us consider the case $C\colon A\curvearrowright_{p-1}^p B \ \text{in} \ \Omega^{m+1}(X,g)$ with $p>0$ (the remaining case $C\colon A\curvearrowright_p^{p-1} B \ \text{in} \ \Omega^{m+1}(X,g)$ is treated similarly).
By definition, we have a cylinder \[\bar{C}\colon A \curvearrowright_{p-2}^{p-1} B \ \text{in} \ \Omega^{m+2}(X,\phi) \] where $\phi=(s^n(A),t^n(B))$.
By inductive hypothesis, we obtain 

\[f_k\ldots f_1 \bar{C} h_q \ldots h_1\colon f_k\ldots f_1 A h_q \ldots h_1\curvearrowright_{p-2}^{p-1} f_k\ldots f_1 B h_q \ldots h_1\]
so that we can set 
\[\overline{f_k\ldots f_1 C h_q \ldots h_1}=f_k\ldots f_1\bar{C} h_q \ldots h_1 \]
Given vertically composable (possibly degenerate) $n$-cylinders $C_1,\ldots, C_k$ in an $\infty$-groupoid $X$, we denote by $C_1 \otimes \ldots \otimes C_k$ the $n$-cylinder in $X$ that results as their vertical composition.

\subsection{Elementary interpretation of operations}
In this section we finally define, for every homogeneous operation $\rho:D_m \rightarrow A$ (so that $m\geq n=\dim(A)$) in a fixed coherator for $\infty$-categories $\mathfrak{D}$, a map
\[\hat{\rho}\colon \cyl(D_m) \rightarrow \cyl(A)\] satisfying both properties depicted in \eqref{hat properties}. In what follows, we will implicitly assume that a map $\mathfrak{D}\rightarrow \mathfrak{C}$ has been chosen once and for all (its existence is ensured by cellularity of $\mathfrak{D}$ and contractibility of $\mathfrak{C}$), and we identify maps in the domain with their image in the codomain as there is no harm in doing so.
To achieve the goal we set for this section, given $\epsilon=\sigma, \tau$ we need a description of the map 
\[\cyl(\partial_{\epsilon})\colon \cyl(\partial A) \rightarrow \cyl(A) \] in terms of the globular decomposition of both its domain and its target, where $\partial_{\epsilon}:\partial A \rightarrow A$ are the maps in $\Theta_0$ defined in \ref{partial defi}.

By construction, we know that the bottom row (see \eqref{zig zag picture}) of the globular decomposition of $\cyl(A)$ is obtained by sticking a new branch at the bottom right of the tree associated with $A$ and then letting this new branch traverse the tree counterclockwise.

Let $n=\dim(A)$, let us first explain how to get the list of trees appearing on the bottom row of the globular decomposition of $\cyl(\partial A)$ starting from the one associated with $\cyl (A)$.
We have three possible cases for the newly added branch in each tree belonging to the set $\mathcal{L}(A)$:
\begin{itemize}
\item it is attached at height $k\geq n$;
\item it is attached at height $k<n-1$;
\item it is attached at height $n-1$.
\end{itemize}
We discard all the trees in the first class, and we keep all the trees associated with the ones belonging to the second class, chopping off everything above height $k=n-1$.
In the third case, if we focus on the strip between height $n-1$ and $n$, the newly added branch has to appear in a certain corolla.
If the newly added branch is at the far left of the corolla (this will be important when we give an alternative description of $\cyl(\partial_{\sigma})$ in Proposition \ref{cocone partial}) it belongs to, then we chop off everything above dimension $n-1$ except this new branch, and we keep the resulting tree. If not, we discard the tree.

By doing so we get a new list that can be easily proven to correspond exactly to the one for $\partial A$, and for which we have maps in $\Theta_0$ between each tree in the list for $\cyl(\partial A)$ and the corresponding one in the list for $\cyl(A)$, induced by the source maps or appropriate colimit inclusions.

This can be reformulated in the following way. Consider the two diagrams of $\infty$-groupoids $\Cyl(\partial A):I_{\vert \mathscr{L}(\partial A)\vert} \rightarrow \wgpd$ and $\Cyl(A):I_{\vert \mathscr{L}(A)\vert } \rightarrow \wgpd$. The previous analysis essentially specifies a cocone under $\Cyl(\partial A)$, whose vertex is $\cyl(A)$, such that the map induced on the colimit is precisely $\cyl(\partial_{\sigma})$.
In fact, the trees that we keep (after having suitably modified them if prescribed by the algorithm described above) come equipped with maps into the tree in the decomposition of $\cyl(A)$ that they are associated to, and these maps are all induced by source maps or colimit inclusions.
After having postcomposed these maps with the colimit inclusions we get the cocone we are after.

A similar analysis, replacing every occurence of ``left'' with ``right'', gives an analogous result for the map $\cyl(\partial_{\tau}):\cyl(\partial A) \rightarrow \cyl(A)$.
\begin{ex}
Let's have a look at a specific example to clarify this argument.

Consider the globular sum given by $A=D_2 \amalg_{D_1} D_2 \amalg_{ D_0} D_1$. We have $\partial A= D_1 \amalg_{ D_0} D_1$.\\
According to the abovementioned rule, to describe the map $\cyl(\partial_{\sigma}):\cyl(\partial A) \rightarrow \cyl(A)$ we have to consider each of the trees on the bottom row, and check where the new edge is.
We have to discard all those in which this special edge is attached at height $n=2$, and keep those in which it is attached at height $n=0$, chopping off everything above height $n=1$.
Moreover, whenever it is attached at height $n=1$, we select only those in which the newly added edge is at the far left of the corolla it belongs to, and we chop everything above height $n=1$ except for this edge.

The list appearing  on the bottom row in the zig-zag expressing the globular decomposition of $\cyl(A)$ is given by:
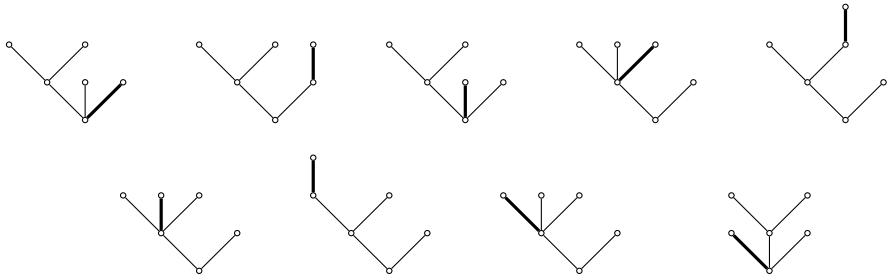
\[
\begin{tikzpicture}

\tikzstyle{every node}=[circle, draw,
inner sep=0pt, minimum width=2pt]
\node[] (A) at (0,7) {};

\node[] (B) at (-0.5,7.5) {};
\node[] (C) at (0,7.5) {};
\node[] (D) at (0.5,7.5) {};

\node[] (E) at (-1,8) {};
\node[] (F) at (0,8) {};

\path [] (A) edge (B);
\path [draw=red, very thick] (A) edge (D);
\path [] (A) edge (C);
\path [] (B) edge (E);
\path [] (B) edge (F);

\node[] (A') at (2.5,7) {};

\node[] (B') at (3,7.5) {};
\node[] (D') at (2,7.5) {};

\node[] (F') at (2.5,8) {};
\node[] (E') at (1.5,8) {};
\node[] (C') at (3,8) {};

\path [] (A') edge (B');
\path [draw=red, very thick] (B') edge (C');
\path [] (A') edge (D');
\path [] (D') edge (E');
\path [] (D') edge (F');

\node[] (A'') at (5,7) {};

\node[] (B'') at (4.5,7.5) {};
\node[] (C'') at (5,7.5) {};
\node[] (D'') at (5.5,7.5) {};

\node[] (E'') at (4,8) {};
\node[] (F'') at (5,8) {};

\path [] (A'') edge (B'');
\path [draw=red, very thick] (A'') edge (C'');
\path [] (A'') edge (D'');
\path [] (B'') edge (E'');
\path [] (B'') edge (F'');

\node[] (a) at (7.5,7) {};

\node[] (b) at (7,7.5) {};
\node[] (c) at (7.5,8) {};
\node[] (d) at (8,7.5) {};

\node[] (e) at (6.5,8) {};
\node[] (f) at (7,8) {};

\path [] (a) edge (b);
\path [draw=red, very thick] (b) edge (c);
\path [] (a) edge (d);
\path [] (b) edge (e);
\path [] (b) edge (f);

\node[] (a'') at (10,7) {};

\node[] (b'') at (9.5,7.5) {};
\node[] (d'') at (10.5,7.5) {};

\node[] (c'') at (10,8) {};
\node[] (f'') at (9,8) {};

\node[] (e'') at (10,8.5) {};

\path [] (a'') edge (b'');
\path [draw=red, very thick] (c'') edge (e'');
\path [] (a'') edge (d'');
\path [] (b'') edge (c'');
\path [] (b'') edge (f'');

\node[] (a') at (1.5,5) {};

\node[] (b') at (1,5.5) {};
\node[] (d') at (2,5.5) {};

\node[] (c') at (1.5,6) {};
\node[] (e') at (1,6) {};
\node[] (f') at (0.5,6) {};

\path [] (a') edge (b');
\path [draw=red, very thick] (b') edge (e');
\path [] (a') edge (d');
\path [] (b') edge (c');
\path [] (b') edge (f');

\node[] (a''') at (4,5) {};

\node[] (b''') at (4.5,5.5) {};
\node[] (c''') at (3.5,5.5) {};

\node[] (d''') at (3,6) {};
\node[] (e''') at (4,6) {};

\node[] (f''') at (3,6.5) {};

\path [] (a''') edge (b''');
\path [draw=red, very thick] (d''') edge (f''');
\path [] (c''') edge (d''');
\path [] (a''') edge (c''');
\path [] (c''') edge (e''');

\node[] (U) at (6.5,5) {};

\node[] (V) at (7,5.5) {};
\node[] (W) at (6,5.5) {};

\node[] (X) at (5.5,6) {};
\node[] (Y) at (6.5,6) {};
\node[] (Z) at (6,6) {};

\path [] (U) edge (V);
\path [draw=red, very thick] (W) edge (X);
\path [] (U) edge (W);
\path [] (W) edge (Y);
\path [] (W) edge (Z);

\node[] (u) at (9,5) {};

\node[] (v) at (9.5,5.5) {};
\node[] (w) at (8.5,5.5) {};
\node[] (x) at (9,5.5) {};

\node[] (y) at (8.5,6) {};
\node[] (z) at (9.5,6) {};

\path [] (u) edge (v);
\path [draw=red, very thick] (u) edge (w);
\path [] (u) edge (x);
\path [] (x) edge (y);
\path [] (x) edge (z);

\end{tikzpicture}\]
For example, let's consider the sixth tree of this list, namely: 
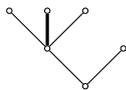
\[
\begin{tikzpicture}
\tikzstyle{every node}=[circle, draw,
inner sep=0pt, minimum width=2pt]

\node[] (a') at (5.5,-5.5) {};

\node[] (b') at (5,-5) {};
\node[] (c') at (5.5,-4.5) {};
\node[] (d') at (6,-5) {};

\node[] (e') at (5,-4.5) {};
\node[] (f') at (4.5,-4.5) {};

\path [] (a') edge (b');
\path [draw=red, very thick] (b') edge (e');
\path [] (a') edge (d');
\path [] (b') edge (c');
\path [] (b') edge (f');
\end{tikzpicture}\]
The corolla which the special edge belongs to is $\begin{tikzpicture}
\tikzstyle{every node}=[circle, draw,
inner sep=0pt, minimum width=2pt]

\node[] (b') at (5,-5) {};
\node[] (c') at (5.5,-4.5) {};
\node[] (e') at (5,-4.5) {};
\node[] (f') at (4.5,-4.5) {};

\path [draw=red, very thick] (b') edge (e');
\path [] (b') edge (c');
\path [] (b') edge (f');
\end{tikzpicture}$, and the red edge is not at the far left of it, so we discard this tree.

Proceding as described by the rule, we are left with the following list of trees, identified with (respectively) the first, second, third, eighth and ninth tree of the previous list
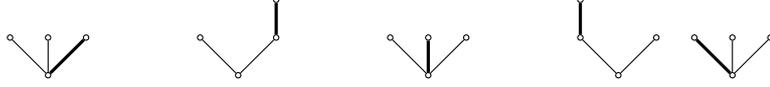
\[
\begin{tikzpicture}
\tikzstyle{every node}=[circle, draw,
inner sep=0pt, minimum width=2pt]

\node[] (b') at (0,0) {};
\node[] (c') at (0.5,0.5) {};
\node[] (e') at (0,0.5) {};
\node[] (f') at (-0.5,0.5) {};

\path [draw=red, very thick] (b') edge (c');
\path [] (b') edge (e');
\path [] (b') edge (f');

\node[] (b) at (2.5,0) {};
\node[] (c) at (3,0.5) {};
\node[] (e) at (3,1) {};
\node[] (f) at (2,0.5) {};

\path [draw=red, very thick] (c) edge (e);
\path [] (b) edge (c);
\path [] (b) edge (f);

\node[] (b'') at (5,0) {};
\node[] (c'') at (5.5,0.5) {};
\node[] (e'') at (5,0.5) {};
\node[] (f'') at (4.5,0.5) {};

\path [draw=red, very thick] (b'') edge (e'');
\path [] (b'') edge (c'');
\path [] (b'') edge (f'');

\node[] (b-) at (7.5,0) {};
\node[] (c-) at (8,0.5) {};
\node[] (e-) at (7,1) {};
\node[] (f-) at (7,0.5) {};

\path [draw=red, very thick] (f-) edge (e-);
\path [] (b-) edge (c-);
\path [] (b-) edge (f-);

\node[] (b'-) at (9,0) {};
\node[] (c'-) at (9.5,0.5) {};
\node[] (e'-) at (9,0.5) {};
\node[] (f'-) at (8.5,0.5) {};

\path [draw=red, very thick] (b'-) edge (f'-);
\path [] (b'-) edge (e'-);
\path [] (b'-) edge (c'-);
\end{tikzpicture}
\]
Clearly, this is the list of trees appearing on the bottom row of the globular decomposition of $\cyl(\partial A)=\cyl(D_1 \amalg_{ D_0} D_1)$.
\end{ex}
Let's make this precise: we start by defining a map of sets $\phi^{\sigma}_A\colon \mathscr{L}(\partial A)\rightarrow\mathscr{L}(A)$ (resp. $\phi^{\tau}_A\colon \mathscr{L}(\partial A)\rightarrow\mathscr{L}(A)$) by sending $B\in \mathscr{L}(\partial A)$ to $\phi^{\sigma}_A(B)$ (resp. $\phi^{\tau}_A(B)$), whose underlying presheaf of sets is given by adjoining a new vertex to $A$ in the fiber over $x\in A_{m-1}$ if $B$ is obtained by adding a vertex to $\partial A$, in the fiber over such $x$.

To extend the linear order, we consider two possible cases. First, if $m<n=\dim(A)$ then one has $(\iota^{\partial A}_{m-1})^{-1}\{x\}=(\iota^{A}_{m-1})^{-1}\{x\}$. Therefore, we endow the fiber over $x$ in $\phi^{\sigma}_A(B)$ (resp. $\phi^{\tau}_A(B)$) with the linear order transported from the one in $(\iota^{\partial A}_{m-1})^{-1}\{x\}$.
If $m=n$, we impose that the newly adjoined vertex in $\phi^{\sigma}_A(B)$ (resp. $\phi^{\tau}_A(B)$) is the least element (resp. the greatest element) in the fiber over $x$ .

Given any $B\in \mathscr{L}(\partial A)$, let us define a map $j^{\sigma}_B:B \rightarrow \phi^{\sigma}_A(B)$ (resp. $j^{\tau}_B:B \rightarrow \phi^{\tau}_A(B)$) in $\Theta_0$. Following the notation in the previous paragraph, if $m<n$, we define $j^{\sigma}_B$ to be $\partial_{\sigma}\colon B \rightarrow \phi^{\sigma}_A(B)$, and $j^{\tau}_B$ to be $\partial_{\tau}\colon B \rightarrow \phi^{\tau}_A(B)$.
If $m=n$, then the maps $j^{\sigma}_B:B \rightarrow \phi^{\sigma}_A(B)$ and $j^{\tau}_B:B \rightarrow \phi^{\tau}_A(B)$ are induced by the universal property of pushouts as depicted in the diagrams below, where the front and back faces of the cubes are cocartesian
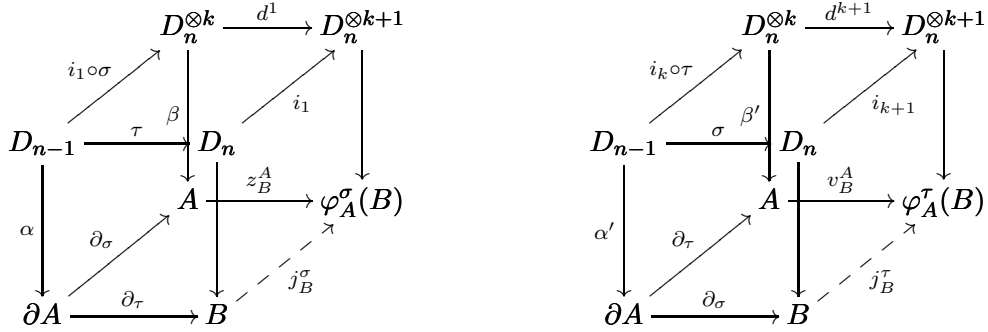
\[
\bfig 
	\morphism(-1200,-400)|a|/@{>}@<0pt>/<600,0>[D_{n-1}`D_n;\tau]
		\morphism(-1200,-400)|a|/@{>}@<0pt>/<0,-600>[D_{n-1}`\partial A;\alpha]
	\morphism(-1200,-1000)|a|/@{>}@<0pt>/<600,0>[\partial A`B;\partial_{\tau}]
	\morphism(-600,-400)|a|/@{>}@<0pt>/<0,-600>[D_{n}`B;] 
	
\morphism(-700,0)|a|/@{>}@<0pt>/<600,0>[D_n^{\otimes k}`D_n^{\otimes k+1};d^1]
\morphism(-700,0)|a|/@{>}@<0pt>/<0,-600>[D_n^{\otimes k}` A;\beta]
\morphism(-700,-600)|a|/@{>}@<0pt>/<600,0>[ A`\phi^{\sigma}_{A}(B);z^A_B]
\morphism(-100,0)|a|/@{>}@<0pt>/<0,-600>[D_n^{\otimes k+1}`\phi^{\sigma}_{A}(B);]

\morphism(-1200,-1000)|a|/@{>}@<0pt>/<500,400>[\partial A`A;\partial_{\sigma}]
	\morphism(-1200,-400)|a|/@{>}@<0pt>/<500,400>[D_{n-1}`D_n^{\otimes k};i_1\circ \sigma]
		\morphism(-600,-400)|r|/@{>}@<0pt>/<500,400>[D_{n}`D_n^{\otimes k+1};i_1] 
			\morphism(-600,-1000)|r|/@{-->}@<0pt>/<500,400>[B`\phi^{\sigma}_{A}(B);j^{\sigma}_B] 

\morphism(800,-400)|a|/@{>}@<0pt>/<600,0>[D_{n-1}`D_n;\sigma]
\morphism(800,-400)|a|/@{>}@<0pt>/<0,-600>[D_{n-1}`\partial A;\alpha']
\morphism(800,-1000)|a|/@{>}@<0pt>/<600,0>[\partial A`B;\partial_{\sigma}]
\morphism(1400,-400)|a|/@{>}@<0pt>/<0,-600>[D_{n}`B;] 

\morphism(1300,0)|a|/@{>}@<0pt>/<600,0>[D_n^{\otimes k}`D_n^{\otimes k+1};d^{k+1}]
\morphism(1300,0)|a|/@{>}@<0pt>/<0,-600>[D_n^{\otimes k}` A;\beta']
\morphism(1300,-600)|a|/@{>}@<0pt>/<600,0>[ A`\phi^{\tau}_{A}(B);v^A_B]
\morphism(1900,0)|a|/@{>}@<0pt>/<0,-600>[D_n^{\otimes k+1}`\phi^{\tau}_{A}(B);]

\morphism(800,-1000)|a|/@{>}@<0pt>/<500,400>[\partial A`A;\partial_{\tau}]
\morphism(800,-400)|a|/@{>}@<0pt>/<500,400>[D_{n-1}`D_n^{\otimes k};i_k\circ \tau]
\morphism(1400,-400)|r|/@{>}@<0pt>/<500,400>[D_{n}`D_n^{\otimes k+1};i_{k+1}] 
\morphism(1400,-1000)|r|/@{-->}@<0pt>/<500,400>[B`\phi^{\tau}_{A}(B);j^{\tau}_B] 
\efig 
\]
Here, $D_n^{\otimes k}=\Sigma^{n-1}(D_1^{\otimes k})$, $d^r$ is the map that skips the $r$-th summand, $i_k$ denotes the inclusion of the $k$-th summand, $\alpha,\alpha'$ represent the target (resp.source) of the leaf we are adjoining to $\partial A$ and $\beta,\beta'$ are the corresponding inclusion of the fiber of $A$ over $x$.
\begin{prop}
	\label{cocone partial}
Given a globular sum $A$ with $\dim(A)=n>0$, we have the following commutative square for each $B\in \mathscr{L}(\partial A)$ and $\epsilon=\sigma, \tau$:
\[
\bfig
	\morphism(0,0)|a|/@{>}@<0pt>/<750,0>[B` \phi^{\epsilon}_A(B);j^{\epsilon}_B]
	\morphism(0,0)|a|/@{>}@<0pt>/<0,-500>[B` \cyl(\partial A);i_B]
	\morphism(0,-500)|a|/@{>}@<0pt>/<750,0>[\cyl(\partial A)` \cyl(A);\cyl(\partial_{\epsilon})]
	\morphism(750,0)|r|/@{>}@<0pt>/<0,-500>[\phi^{\epsilon}_A(B)`\cyl(A);i_{\phi^{\epsilon}_A(B)}]
\efig
\]
\begin{proof}
We only prove the case $\epsilon=\sigma$, the other one being entirely dual. The statement is clear if $A=D_1^{\otimes m}$, in which case $\partial A=D_0$.
Otherwise, let $A=\Sigma A_1 \amalg_{D_0} \ldots \amalg_{ D_0} \Sigma A_k$, as in \eqref{decomposition of glob sum}, and assume the result holds for all the $A_i$'s.
Define $A'_i$ to be $A_i$ if $\dim(A_i)<n-1$, or $\partial A_i$ otherwise.
Let us subdivide the set $\mathscr{L}(\partial A)$ as the union of the set of globular sums for which the new edge is joined at the root, and the sets of the form \[\{\Sigma A'_1 \plus{D_0} \ldots\plus{D_0} \Sigma B_j \plus{ D_0} \ldots \plus{ D_0}\Sigma A'_k\}_{1\leq i \leq k, B_j \in \mathscr{L}(A'_i)}\]
If $B$ belongs to the first set, i.e. $B=\Sigma A'_1 \amalg_{ D_0}\ldots \Sigma A'_i \amalg_{ D_0}D_1 \amalg_{ D_0} \Sigma A'_{i+1} \amalg_{ D_0} \Sigma A'_k$ for some $0\leq i \leq k+1$, then thanks to Remark \ref{some maps in cyl(A) zig zag} one has the following commutative square
\[
\bfig
\morphism(0,0)|a|/@{>}@<0pt>/<1500,0>[B`\Sigma A_1 \plus{ D_0}\ldots \Sigma A_i \plus{ D_0}D_1 \plus{ D_0} \Sigma A_{i+1} \plus{ D_0} \Sigma A_k;\partial_{\sigma}]
\morphism(0,0)|a|/@{>}@<0pt>/<0,-500>[B` \cyl(\partial A);i_B]
\morphism(0,-500)|a|/@{>}@<0pt>/<1500,0>[\cyl(\partial A)` \cyl(A);\cyl(\partial_{\sigma})]
\morphism(1500,0)|r|/@{>}@<0pt>/<0,-500>[\Sigma A_1 \plus{ D_0}\ldots \Sigma A_i \plus{ D_0}D_1 \plus{ D_0} \Sigma A_{i+1} \plus{ D_0} \Sigma A_k`\cyl(A);i_{\phi_A(B)}]
\efig
\]
To conclude the proof for this case, just observe that the upper horizontal map coincides with $j^{\sigma}_B\colon B \rightarrow \phi^{\sigma}_A(B)$.

Next, suppose $B=\Sigma A'_1 \amalg_{D_0} \ldots\amalg_{D_0} \Sigma B_j \amalg_{ D_0} \ldots \amalg_{ D_0}\Sigma A'_k$ for some $B_j \in \mathscr{L}(A'_i)$.
By construction, we have that the natural transformation $\cyl(\partial_{\sigma})$ restricted to the sub zig-zag \[\Sigma A'_1\plus{ D_0}\ldots\plus{ D_0} \Sigma A'_{i-1}\plus{ D_0}\Sigma \Cyl (A'_i) \plus{ D_0} \Sigma A'_{i+1}\plus{ D_0} \ldots\plus{ D_0} \Sigma A'_k\] coincides with \[\Sigma  \partial'_1 \plus{ D_0}\ldots \plus{ D_0} \Sigma \cyl(\partial'_i)\plus{D_0} \Sigma \partial'_{i+1}\plus{ D_0}\ldots\plus{ D_0}\Sigma  \partial'_k \]
where we set $\partial'_i$ to be $\partial_{\sigma}:\partial A_i=A'_i \rightarrow A_i$ if $\dim (A_i)=n-1$, and the identity otherwise.

Thanks to Remark \ref{some maps in cyl(A) zig zag} and the inductive hypothesis,we get the following commutative square 
\[
\bfig
\morphism(0,0)|a|/@{>}@<0pt>/<2800,0>[B`\Sigma A_1 \plus{D_0} \ldots\plus{D_0} \Sigma\phi^{\sigma}_{A}(B_j) \plus{ D_0} \ldots \plus{ D_0}\Sigma A_k;\Sigma \partial'_1\plus{ D_0}\ldots\plus{ D_0}\Sigma j_{B_j}\plus{ D_0} \Sigma \partial'_{i+1}\plus{ D_0} \ldots \plus{ D_0} \Sigma \partial'_k]
\morphism(0,0)|a|/@{>}@<0pt>/<0,-500>[B` \cyl(\partial A);i_B]
\morphism(0,-500)|a|/@{>}@<0pt>/<2800,0>[\cyl(\partial A)` \cyl(A);\cyl(\partial_{\sigma})]
\morphism(2800,0)|r|/@{>}@<0pt>/<0,-500>[\Sigma A_1 \plus{D_0} \ldots\plus{D_0} \Sigma\phi^{\sigma}_{A}(B_j) \plus{ D_0} \ldots \plus{ D_0}\Sigma A_k`\cyl(A);i]
\efig
\]
We conclude by observing that \[\Sigma A_1 \plus{D_0} \ldots\plus{D_0} \Sigma\phi^{\sigma}_{A}(B_j) \plus{ D_0} \ldots \plus{ D_0}\Sigma A_k=\phi_A(B)\] and
\[\Sigma \partial'_1\plus{ D_0}\ldots\plus{ D_0}\Sigma j_{B_i}\plus{ D_0} \Sigma \partial'_{i+1}\plus{ D_0} \ldots \plus{ D_0} \Sigma \partial'_k=j_B\]
\end{proof}
\end{prop}
We also record here the following result, for future use (and generalization):
\begin{prop}
	\label{compatibility of j's}
	The following square commutes for $\epsilon=\sigma,\tau$
	\[\bfig
	\morphism(0,0)|a|/@{>}@<0pt>/<500,0>[\partial A` A;\partial^A_{\epsilon}]
	\morphism(0,0)|a|/@{>}@<0pt>/<0,-500>[\partial A` B;z^{\partial A}_B]
	\morphism(0,-500)|a|/@{>}@<0pt>/<500,0>[B` \phi^{\epsilon}_A(B);j^{\epsilon}_B]
	\morphism(500,0)|r|/@{>}@<0pt>/<0,-500>[A`\phi^{\epsilon}_A(B);z^A_{\phi^{\epsilon}_A(B)}]
	\efig
	\]
	\begin{proof}
		Let$*_B$ be the vertex adjoined to $\partial A$ to get $B$, and set $m=\height(*_B)$.
		
		If $*_B=\min F$, with $F=(\iota^B_{m-1})^{-1}\{x\}$ for $x=\iota^B_{m-1}(*_B)$, then either $n>m \geq 1$, in which case the square in the statement is given by
		\[
		\bfig 
		\morphism(0,0)|a|/@{>}@<0pt>/<500,0>[\partial A` A;\partial_{\sigma}]
		\morphism(0,0)|a|/@{>}@<0pt>/<0,-500>[\partial A` B;]
		\morphism(0,-500)|a|/@{>}@<0pt>/<500,0>[B` \phi^{\epsilon}_A(B);\partial_{\sigma}]
		\morphism(500,0)|r|/@{>}@<0pt>/<0,-500>[A`\phi^{\epsilon}_A(B);]
		\efig 
		\]
		where the unlabelled maps are the natural inclusions discussed in the definition of the maps $z^A_B$, or $m=n$, in which case the square becomes
		\[
		\bfig 
		\morphism(0,0)|a|/@{>}@<0pt>/<500,0>[\partial A` A;\partial_{\sigma}]
		\morphism(0,0)|a|/@{>}@<0pt>/<0,-500>[\partial A` B;\partial_{\tau}]
		\morphism(0,-500)|a|/@{>}@<0pt>/<500,0>[B` \phi^{\epsilon}_A(B);j^{\sigma}_B]
		\morphism(500,0)|r|/@{>}@<0pt>/<0,-500>[A`\phi^{\epsilon}_A(B);z^A_{\phi^{\sigma}_A(B)}]
		\efig 
		\]
	The first square commutes for obvious reasons, and the second one is the bottom face of the cube appearing in the definition of the maps $j^{\sigma}_B$, and therefore commutes as well.	
		
		
If $*_B\neq \min F$, then $m<n$ and so the square is given by 		\[
\bfig 
\morphism(0,0)|a|/@{>}@<0pt>/<500,0>[\partial A` ;\partial_{\sigma}]
\morphism(0,0)|a|/@{>}@<0pt>/<0,-500>[\partial A`B;z^A_B]
\morphism(0,-500)|a|/@{>}@<0pt>/<500,0>[B`\phi^{\epsilon}_A(B);\partial_{\sigma}]
\morphism(500,0)|r|/@{>}@<0pt>/<0,-500>[A`\phi^{\epsilon}_A(B);z^A_{\phi^{\sigma}_A(B)}]
\efig 
\]
If we let $F'=(\iota^{\phi^{\sigma}_A(B)}_{m-1})^{-1}\{x'\}$ for $x'=\iota^{\phi^{\sigma}_A(B)}_{m-1}j_B^{\sigma}(*_B)$, we see that $*_{\phi^{\sigma}_A(B)}=j_B^{\sigma}(*_B)\neq \min F'$, so that the commutativity of the square above follows immediately from the globularity of the generalized whiskering $w$'s.
	\end{proof}
\end{prop}
We now describe how to extend the results presented thus far, in order to study the maps $\cyl(\partial_{\sigma}^k):\cyl(\partial^k A) \rightarrow \cyl (A)$ for $k>1$.
This time too, if we want to obtain the list of trees appearing on the bottom row of the globular decomposition of $\cyl(\partial^k A)$ we start with those in the decomposition of $\cyl(A)$.

Firstly, we have to discard all the trees in which the new edge has been attached at height $m\geq n-k+1$. If it has been attached at height $m<n-k$, we keep the trees after having chopped off everything above height $n-k$.
Finally, if the new edge is attached at height $m=n-k$, we consider the strip comprised between height $m=n-k$ and $m'=n-k+1$. The new edge belongs to a corolla in here, and we only keep the trees in which it is at the far left of the corolla it belongs to. Again, in this case, we chop everything above height $n-k$, except for the newly added edge.

Each of the trees (representing a globular sum) we thus obtain comes equipped with a map (induced by a $k$-fold iteration of the source maps or by a colimit inclusion) towards the one appearing in the decomposition of $\cyl(A)$ which it is associated to.
In this way one gets a cocone under the diagram $\Cyl(\partial^k A)\colon I_m \rightarrow \wgpd$, whose vertex is $\cyl(A)$, such that the map induced on the colimit is precisely $\cyl(\partial_{\sigma}^k)$.

A similar argument, replacing every occurence of ``left'' with ``right'', yields an analogous result for the map $\cyl(\partial_{\tau}^k)\colon \cyl(\partial^k A) \rightarrow \cyl (A)$.
To make this precise we have to generalize the work already done for $k=1$.
\begin{defi}
Let A be a globular sum, and $k>0$ a positive integer. Define (for $\epsilon=\sigma, \tau$) the map of sets $(\phi^{\epsilon}_{A})^k \colon \mathscr{L}(\partial^k A) \rightarrow \mathscr{L}(A)$ inductively, by setting $(\phi^{\epsilon}_{A})^1=\phi^{\epsilon}_A$ and $(\phi^{\epsilon}_{A})^k=\phi^{\epsilon}_A \circ (\phi^{\epsilon}_{A})^{k-1} $ for $k>1$.
\end{defi}
Unraveling the previous definition we see that $(\phi^{\epsilon}_{A})^k=\phi^{\epsilon}_A \circ \phi^{\epsilon}_{\partial A} \circ \ldots \circ \phi^{\epsilon}_{\partial^k A}$.
We also define, given $B \in \mathscr{L}(\partial^k A)$, a map $(j^{\epsilon}_B)^k\colon B \rightarrow (\phi^{\epsilon}_{A})^k(B)$ by setting $(j^{\epsilon}_B)^1=j^{\epsilon}_B$ and $(j^{\epsilon}_B)^k=j^{\epsilon}_{\phi_{\partial A}^{k-1}(B)} \circ (j^{\epsilon}_B)^{k-1}$ if $k>1$.

The following result is an immediate consequence of Proposition \ref{cocone partial}.
\begin{prop}
Given a positive integer $k$ and a globular sum $A$ with $\dim(A)=n\geq k$, we have the following commutative square for each $B\in \mathscr{L}(\partial^k A)$:
\[
\bfig
\morphism(0,0)|a|/@{>}@<0pt>/<750,0>[B` (\phi^{\epsilon}_{A})^k(B);(j^{\epsilon}_B)^k]
\morphism(0,0)|a|/@{>}@<0pt>/<0,-500>[B` \cyl(\partial^k A);i_B]
\morphism(0,-500)|a|/@{>}@<0pt>/<750,0>[\cyl(\partial^k A)` \cyl(A);\cyl(\partial_{\epsilon}^k)]
\morphism(750,0)|r|/@{>}@<0pt>/<0,-500>[(\phi^{\epsilon}_{A})^k(B)`\cyl(A);i_{(\phi^{\epsilon}_{A})^k(B)}]
\efig
\]
\end{prop}
It is easy to give an explicit description of $(\phi^{\epsilon}_{A})^k(B)$, similar to what we did for the case $k=1$.
\begin{lemma}
Consider a globular sum $B\in \mathscr{L}(\partial^k A)$, so that $B$ is obtained by adjoining a new vertex $*_B$ to $\partial^k A$. Let $m=\height (*_B)$, define $x=\iota_{B} (*_B)$ and $F=\left(\iota_{\partial^k A}\right)^{-1}(x)$.

If $1\leq m < n-k+1$ $(\phi^{\epsilon}_A)^k(B)$ is obtained by adding a new vertex to $A$ in the fiber $(\iota_A)^{-1}(x)\cong F$, with the linear order inherited from $F$.

If $m=n-k+1$, then  $(\phi^{\epsilon}_A)^k(B)$ is obtained by adding a new vertex to A in the fiber $(\iota_A)^{-1}(x)$, where we extend the linear order by imposing that the newly added vertex is the least element in this fiber if $\epsilon=\sigma$, and the greatest if $\epsilon=\tau$.
\begin{proof}
We prove this lemma by induction, the case $k=1$ being already proven. We also assume $\epsilon=\sigma$, the other case being entirely dual, and we drop the superscripts since we have just clarified any possible ambiguity. Let $k>1$ and assume the claim holds for any globular sum with dimension greater than $k-1$, and for every integer $k'<k$. By definition, $\phi_A^k(B)=\phi_A \left( \phi_{\partial A}^{k-1}(B) \right)$.

If $1\leq m \leq n-k$ then $m \leq (n-1)-(k-1)$, so that, by inductive hypothesis, $\phi_{\partial A}^{k-1}(B)$ is obtained by adding a new vertex to the fiber of $\partial A$ over $x$ (this fiber coincides with $F$, and the order is the transported one). Because $m\leq n-k<n$, $\phi_{A}$ sends $\phi_{\partial A}^{k-1}(B)$ to a tree obtained by adding a new vertex to $A$, over $x$, with the order induced once again by that of $F$.

If $m=n-k+1=(n-1)-(k-1)+1$ then $\phi_{\partial A}^{k-1}(B)$ is obtained from $\partial A$ by adding a new least element to its fiber over $x$. Since $k>1$, this fiber is the same as that of $A$ over $x$, and $\phi_{A}$ sends $\phi_{\partial A}^{k-1}(B)$ to the tree obtained by adding a new vertex over $x$ to $A$, with the order transported from that of $\phi_{\partial A}^{k-1}(B)$. Therefore, the newly added vertex is going to be the minimum in the fiber over $x$, which concludes the proof. 
\end{proof}
\end{lemma}
We also generalize Proposition \ref{compatibility of j's}
\begin{prop}
Let $k$ be a positive integer and $A$ be a globular sum with $\dim(A)=n\geq k$.
Given $B\in \mathscr{L}(\partial^k A)$ the following square commutes for $\epsilon=\sigma,\tau$
\[\bfig
\morphism(0,0)|a|/@{>}@<0pt>/<500,0>[\partial^k A` A;\partial_{\epsilon}^k]
\morphism(0,0)|a|/@{>}@<0pt>/<0,-500>[\partial^k A` B;z^{\partial^k A}_B]
\morphism(0,-500)|a|/@{>}@<0pt>/<500,0>[B` (\phi^{\epsilon}_A)^k(B);j^k_B]
\morphism(500,0)|r|/@{>}@<0pt>/<0,-500>[A`(\phi^{\epsilon}_A)^k(B);z^A_{(\phi^{\epsilon}_A)^k(B)}]
\efig
\]
\begin{proof}
	The square is obtained by gluing together two squares, as displayed below
	\[\bfig
\morphism(0,0)|a|/@{>}@<0pt>/<800,0>[\partial^k A` \partial A;\partial_{\epsilon}^k]
\morphism(0,0)|a|/@{>}@<0pt>/<0,-500>[\partial^k A` B;z^{\partial^k A}_B]
\morphism(0,-500)|a|/@{>}@<0pt>/<800,0>[B`(\phi^{\epsilon}_{\partial A})^{k-1}(B);(j_B^{\epsilon})^{k-1}]
\morphism(800,0)|r|/@{>}@<0pt>/<0,-500>[\partial A`(\phi^{\epsilon}_{\partial A})^{k-1}(B);z^{\partial A}_{(\phi^{\epsilon}_{\partial A})^{k-1}(B)}]

\morphism(800,0)|a|/@{>}@<0pt>/<1000,0>[\partial A` A;\partial_{\epsilon}]

\morphism(800,-500)|a|/@{>}@<0pt>/<1000,0>[(\phi^{\epsilon}_{\partial A})^{k-1}(B)` (\phi^{\epsilon}_A)^k(B);j^{\epsilon}_{\phi^{k-1}_{\partial A}(B)}]
\morphism(1800,0)|r|/@{>}@<0pt>/<0,-500>[A`(\phi^{\epsilon}_A)^k(B);z^A_{(\phi^{\epsilon}_A)^k(B)}]
	\efig
	\]
	Therefore, the claim follows by induction from the case treated in Proposition \ref{compatibility of j's}.
\end{proof}
\end{prop}
We now get back to the task of defining $\hat{\rho}\colon \cyl(D_n) \rightarrow \cyl
(A)$, for a homogeneous map $\rho\colon D_n \rightarrow A$, with $m=\dim(A)$. The goal is to define it as the vertical composition of a suitable stack of $(n-1)$-cylinders in the $\infty$-groupoid $\Omega\left(\cyl(A),a,b\right)$. Therefore, we just need to define this vertical stack so that its vertical composition has the desired properties.
We assume this construction has been performed for every $k<n$, in the same way as in what follows.
So far we have $(n-1)$-cells $\alpha_0, \ldots, \alpha_p$, as in \eqref{alpha cells}. We now build possibly degenerate $(n-1)$-cylinders $C_i^{\rho}\colon \alpha_{i-1} \curvearrowright_{q_i}^{r_i} \alpha_{i}$ in $\Omega \left(\cyl(A),a,b\right)$ for $1\leq i \leq p$, such that $C_i^{\rho}$ factors through $\Omega (B_{i},a,b)$.

We do so by defining the boundary of each of these cylinders, and then we extend this piece of data to actual cylinders by applying Propositions \ref{omega preserves contractibles} and \ref{boundary of deg cyl's} to diagrams of the form
\[ \bfig
\morphism(0,0)|a|/@{>}@<0pt>/<800,0>[\partial \cyl^{r_i}_{q_i}(D_{n-1})`\Omega (B_{i},a,b);]
\morphism(0,0)|a|/@{>}@<0pt>/<0,-500>[\partial \cyl^{r_i}_{q_i}(D_{n-1})`\cyl^{r_i}_{q_i}(D_{n-1});\partial]
\morphism(0,-500)|r|/@{-->}@<0pt>/<800,500>[\cyl^{r_i}_{q_i}(D_{n-1})`\Omega (B_{i},a,b);]
\morphism(800,0)|a|/@{>}@<0pt>/<1000,0>[\Omega (B_{i},a,b)`\Omega \left(\cyl (A),a,b\right);\Omega( i_{B_i})]
\morphism(0,-500)|r|/@{>}@<0pt>/<1800,500>[\cyl^{r_i}_{q_i}(D_{n-1})`\Omega\left(\cyl (A),a,b\right);C_i^{\rho}]
\efig \]
To begin with, we have to define $r_i$ and $q_i$ for every $0\leq i \leq p-1$.
$B_i$ is obtained, by construction, by adding a new vertex $*_{B_i}$ to $A$. Letting $d=\height(*_{B_i})$, we have two possibilities: either $*_{B_i}$ is the least (resp. greatest) element in the fiber $(\iota_{B_i})^{-1}(\iota_{B_i}(*_{B_i}))$ or it is not.
In the first case we set $r_i=d-2$ (resp. $q_i=d-2$), otherwise $r_i=d-1$ (resp. $q_i=d-1$).

If $n>m$ then we define the source (resp. the target) of $C_i^{\rho}$ to be $C_i^{\rho \circ \sigma}$ (resp. $C_i^{\rho \circ \tau}$), which have already been defined, since $\rho \circ \epsilon$ is a homogeneous map with domain $D_{n-1}$. It is easy to check that this is well defined, since the indexing set for the $i$'s is given in both cases by $\mathscr{L}(A)$.

If $n=m$ then, for $\epsilon=\sigma, \tau$, there exists a unique factorization of $\rho \circ \epsilon$, due to the homogeneity of the coherator for $\infty$-categories $\mathfrak{D}$, of the form:
\[
\bfig
\morphism(0,0)|a|/@{>}@<0pt>/<500,0>[D_{n-1}` D_n;\epsilon]
\morphism(0,0)|a|/@{>}@<0pt>/<0,-500>[D_{n-1}` \partial A;\rho_{\epsilon}]
\morphism(0,-500)|a|/@{>}@<0pt>/<500,0>[\partial A` A;\partial^A_{\epsilon}]
\morphism(500,0)|r|/@{>}@<0pt>/<0,-500>[D_n`A;\rho]
\efig
\] where $\rho_{\epsilon}$ is a homogeneous map.
If $r_i<n-2$ (resp. $q_i<n-2$), i.e. the source (resp. the target) is not collapsed, then either $d<n$ and $*_{B_i}$ is the least (resp. greatest)  element in the fiber $(\iota_{B_i})^{-1}(\iota_{B_i}(*_{B_i}))$ or $d<n-1$. In both cases there exists a unique $1\leq k \leq \vert \mathscr{L}(\partial A)\vert$ such that $B_i=\phi^{\sigma}_A (E_k)$ (resp. $B_i=\phi^{\tau}_A (E_k)$), where $E_k$ is the $k$-th element of $\mathscr{L}(\partial A)$.
We now define the source (resp. target) of $C_i^{\rho}$ to be $C_k^{\rho_{\sigma}}$ (resp. $C_k^{\rho_{\tau}}$)
This assignment is well defined, as the commutative squares below ensure that the cylinder $C_k^{\rho_{\sigma}}$ (resp. $C_k^{\rho_{\tau}}$) has, as top and bottom cells, the source (resp. target) of the top and bottom cells of $C_i^{\rho}$ 
\[\bfig
\morphism(0,0)|a|/@{>}@<0pt>/<500,0>[D_{n-1}` D_n;\epsilon]
\morphism(0,0)|a|/@{>}@<0pt>/<0,-500>[D_{n-1}` \partial A;\rho_{\epsilon}]
\morphism(0,-500)|a|/@{>}@<0pt>/<500,0>[\partial A` A;\partial^A_{\epsilon}]
\morphism(500,0)|r|/@{>}@<0pt>/<0,-500>[D_n`A;\rho]

\morphism(0,-500)|l|/@{>}@<0pt>/<0,-500>[\partial A` E_k;z^{\partial A}_{E_k}]

\morphism(500,-500)|r|/@{>}@<0pt>/<0,-500>[A` \phi^{\epsilon}_A(E_k);z^A_{\phi^{\epsilon}_A(E_k)}]
\morphism(0,-1000)|a|/@{>}@<0pt>/<500,0>[E_k`\phi^{\epsilon}_A(E_k);j_{E_k}]
\efig
\]
The vertical composition of the stack of cylinders given by: 
\[C^{\rho}_1\otimes \ldots \otimes C^{\rho}_p\colon \cyl^{r_1}_{q_1}(D_{m-1})\otimes \ldots \otimes \cyl^{r_p}_{q_p}(D_{m-1}) \rightarrow \Omega\left(\cyl(A),a,b\right)\]
produces an $(m-1)$-cylinder $C^{\rho}\colon \alpha_0 \curvearrowright \alpha_p$ in $\Omega\left(\cyl(A),a,b\right)$, since $\min\{r_{i}\}_{1\leq i \leq p}=-1$ and $\min\{q_{i}\}_{1\leq i \leq p}=-1$ by construction.
\begin{defi}
	\label{rho hat defi}
Let $\rho\colon D_n \rightarrow A$ be a homogeneous operation in $\mathfrak{D}$. Using the notation established so far, we let $\hat{\rho}\colon \cyl(D_n) \rightarrow \cyl(A)$ be the $n$-cylinder consisting of the following piece of data:
\begin{itemize}
\item $\hat{\rho}_{\epsilon_0}=\cyl(\partial^{n}_{\epsilon})\colon D_1\cong\cyl(D_0) \rightarrow \cyl(A)$
\item $\overline{\hat{\rho}}$ is given by $C^{\rho}\colon\cyl(D_{n-1})\rightarrow \Omega \left(\cyl(A),a,b\right)$
\end{itemize}
\end{defi}
We are now left with checking the compatibility with the coglobular structure, i.e. we have to prove that, for $\epsilon=\sigma, \tau$, $\hat{\rho}\circ \cyl(\epsilon)=\widehat{\rho \circ \epsilon}$ if $n>m$ and $\hat{\rho}\circ \cyl(\epsilon)=\cyl(\partial_{\epsilon})\circ \hat{\rho_{\epsilon}}$ if $n=m$.

The first case is straightforward by construction, since the operation of vertical composition is compatible with the coglobular structure.

The proof of the second case is accomplished by using the following lemma, which essentially says that the collapsed pieces of the boundaries do not contribute to the result of the vertical composition.
\begin{lemma}
	\label{collapsed boundary gives no contribution}
Let q be a positive integer, and suppose given a sequence of $n$-cylinders $C_i\colon \alpha_{i} \curvearrowright^{p_i}_{q_i} \alpha_{i+1}$ in an $\infty$-groupoid X. Consider the ordered set $\{p_i\}_{1\leq i \leq q}$, where $p_i<p_j$ if and only if $i<j$, and let $\{\bar{p}_{i_1}, \ldots, \bar{p}_{i_k}\}$ be the (ordered) subset spanned by those $p_i<n-1$.
Then the cylinders $(C_{i_j}\circ \cyl(\sigma))_{1\leq j \leq k}$ are again composable, and moreover we have
\[(C_1 \otimes \ldots \otimes C_q)\circ \cyl(\sigma)=(C_{i_1}\circ \cyl(\sigma))\otimes \ldots \otimes  (C_{i_k}\circ \cyl(\sigma))\]
An analogous result holds true if we replace p with q and $\sigma$ with $\tau$.
\begin{proof}
The fact that the $C_{i_j}$ are again composable is obvious.
We prove the second statement by induction on $n$, the base case $n=1$ being straightforward. We know that $C_1 \otimes \ldots \otimes C_q$ is obtained by transposing the result of vertically composing  a stack obtained from whiskerings of the $(n-1)$-cylinders $\bar{C}_i$ with appropriate 1-cells together with $(n-1)$-cylinders of the form $\psi^{c,d}$ that witness the rebracketing of cells when needed (as explained in Section \ref{vert comp of deg cyl (section)} ). Whenever $p_i=n-1$ these $\Psi$'s do not appear, so that the claim follows from the inductive assumption and the coglobularity of the remaining $\Psi$'s.
\end{proof}
\end{lemma}
We conclude this section with an extension of Definition \ref{rho hat defi}. Note that a general map $\phi\colon A \rightarrow B$ in $\mathfrak{D}$ is homogeneous if the homogeneous-globular factorizations $D_{i_k}\rightarrow B_k \rightarrow B$ of the composites $D_{i_k}\rightarrow A \rightarrow B$ for every $i_k$ in the table of dimensions of $A$ are such that the following isomorphism holds 
\[\colim_k B_k \cong B\]
\begin{defi}
	\label{ext of hat defi}
If $\phi\colon A\rightarrow B$ is a homogeneous map in $\mathfrak{D}$, we can obtain an elementary interpretation of it which still satisfies the properties expressed in \eqref{hat properties} simply by coglobularity of the construction recorded in Definition \ref{rho hat defi}. Indeed, we can consider the induced homogeneous maps $\phi_k\colon D_{i_k} \rightarrow B_k$ and define $\widehat{\phi}\colon \cyl(A) \rightarrow \cyl(B)$ as the map induced by passing to the colimit the family of maps $\hat{\phi_k}\colon \cyl(D_{i_k})\rightarrow \cyl(B_k)$.

For a general map $h\colon A \rightarrow B$ in $\mathfrak{D}$, we factor $h$ as $h=i\circ \rho$, using homogeneity of $\mathfrak{D}$, with $\rho \colon A \rightarrow C$ homogeneous map and $i\colon C \rightarrow B$ globular map, i.e. a map in $\Theta_0$. Now set $\hat{h}=\cyl(i)\circ \hat{\rho}$ to get the desired result, where we have used the fact that we do have a functor $\cyl\colon\Theta_0 \rightarrow \wgpd$.
\end{defi}
\section{Modifications}
If we consider the framework described in Section \ref{overview}, we see that setting $\cyl(\rho)=\hat{\rho}$ for every map $\rho\colon D_n \rightarrow A$ in $\mathfrak{C}$ does not produce a well defined functor, since it does not necessarily respect the operadic composition in $\mathfrak{C}$.

For instance, suppose the table of dimensions of $A$ is given by 
\[\begin{pmatrix}
j_1 &&j_2 & \ldots&j_{m-1} & &j_m\\
& j'_1 & &\ldots&& j'_{m-1}
\end{pmatrix}\] 
and assume we are given operations $\phi_i \colon D_{j_i} \rightarrow B_i$ for each $1\leq i \leq m$, that are compatible in the sense that the factorizations depicted in the following commutative diagram exist for all such $i$:
\[
\bfig 
\morphism(0,0)|a|/@{>}@<0pt>/<1500,0>[D_{j_i}` B_i;\phi_i]
\morphism(0,500)|l|/@{>}@<0pt>/<0,-500>[D_{j'_{i-1}}` D_{j_i};\sigma_{j'_{i-1}}^{j_i}]
\morphism(0,500)|a|/@{>}@<0pt>/<1500,0>[D_{j'_{i-1}}`\partial^{j_i -j'_{i-1}}B_i=\partial^{j_{i-1}-{j'}_{i-1}}B_{i-1} ;]
\morphism(1500,500)|r|/@{>}@<0pt>/<0,-500>[\partial^{j_i -j'_{i-1}}B_i=\partial^{j_{i-1}-{j'}_{i-1}}B_{i-1} `B_i;\partial^{j_i -j'_{i-1}}_{\sigma}]

\morphism(1500,500)|r|/@{>}@<0pt>/<0,500>[\partial^{j_i -j'_{i-1}}B_i=\partial^{j_{i-1}-{j'}_{i-1}}B_{i-1} `B_{i-1};\partial^{j_{i-1}-{j'}_{i-1}}_{\tau}]
\morphism(0,500)|l|/@{>}@<0pt>/<0,500>[D_{j'_{i-1}}` D_{j_{i-1}};\tau_{j'_{i-1}}^{j_{i-1}}]
\morphism(0,1000)|a|/@{>}@<0pt>/<1500,0>[D_{j_{i-1}}` B_{i-1};\phi_{i-1}]
\efig 
\]
We can now consider $ (\phi_1,\ldots,\phi_m)\circ \rho\colon D_n \rightarrow B$, where $B$ is the globular sum obtained from glueing the $B_i$'s, as explained in Prop. 2.5.7 of \cite{AR1}.

We have two ways of interpreting this using the construction we have just explained: one is $\left((\phi_1,\ldots,\phi_m)\circ \rho\right)^{\wedge}$, by which we mean the elementary interpretation of $(\phi_1,\ldots,\phi_m)\circ \rho$, and the other one is $(\hat{\phi_1},\ldots \hat{\phi_m})\circ \hat{\rho}  $ (which is well defined thanks to globularity of the elementary interpretation). These two maps need not coincide, in fact they almost never do. In order to try to remedy to this problem we introduce modifications of cylinders.

Given an $\infty$-groupoid $X$, a modification in $X$ between $n$-cylinders $\Theta\colon C \Rightarrow D$ will be defined inductively to consist of a pair of $2$-cells $\Theta_s \colon s^n(C)\rightarrow s^n (D)$, $\Theta_t \colon t^n (D) \rightarrow t^n(C)$ together with a modification of $(n-1)$-cylinders in $\Omega (X,x,y)$  \[\bar{\Theta}\colon\Upsilon(\iota_0 C,\Theta_t)\otimes \bar{C}\otimes\Gamma(\Theta_s,\iota_1 C)\Rightarrow \bar{D}\]
where $x=s^n(C)\circ \sigma, \ y=t^n(C)\circ \tau$, and $\Gamma,\Upsilon$ are cylinders we define below.
\begin{ex}
	\label{unraveling modifications}
	Before we formally give the definition of modification, we give an example of what modifications look like in low dimensions. 
	If $n=0$ then a modification is simply a 2-cell.
	If $n=1$ then we can depict $C$ and $D$ as, respectively
	\[
	\bfig
	\morphism(0,0)|a|/@{>}@<0pt>/<300,0>[a`b;\alpha]
	\morphism(0,0)|l|/@{>}@<0pt>/<0,-300>[a`c;f]
	\morphism(300,0)|r|/@{>}@<0pt>/<0,-300>[b`d;g]
	\morphism(0,-300)|b|/@{>}@<0pt>/<300,0>[c`d; \beta]
	\morphism(250,-75)|a|/@{=>}@<0pt>/<-150,-150>[`; \Gamma]

	\morphism(600,0)|a|/@{>}@<0pt>/<300,0>[a`b;\alpha]
	\morphism(600,0)|l|/@{>}@<0pt>/<0,-300>[a`c;f']
	\morphism(900,0)|r|/@{>}@<0pt>/<0,-300>[b`d;g']
	\morphism(600,-300)|b|/@{>}@<0pt>/<300,0>[c`d; \beta]
	\morphism(850,-75)|a|/@{=>}@<0pt>/<-150,-150>[`; \Delta]
	\efig\] 
	Therefore, a modification $\Theta\colon C \Rightarrow D$ corresponds to the data of a pair of $2$-cells $S\colon f \rightarrow f'$, $T\colon g' \rightarrow g$ in $X$ and a modification $\bar{\Theta}\colon\Upsilon(\iota_0 C,\Theta_t)\otimes \bar{C}\otimes\Gamma(\Theta_s,\iota_1 C)\Rightarrow \bar{D}$, which is easily seen to correspond to a $3$-cell $\tilde{\Theta}\colon (\beta \Theta_s)(\Gamma(\Theta_t \alpha)) \rightarrow \Delta$ in $X$, where we denote by juxtaposition the result of the appropriate operations $w$ involved in the definition. Notice that if $f=f'$ and $g=g'$, then a modification $\Theta\colon C \Rightarrow D$ such that $\Theta_s$ and $\Theta_t$ are identities can be equivalently thought of as a $3$-cell between the $2$-cells $\Gamma$ and $\Delta$.
\end{ex}
If we think of $\cyl(D_n)$ to be an instance of a (yet to be defined) Gray tensor product $D_1 \otimes D_n$, then $\mathbf{M}_n$, i.e. the free $\infty$-groupoid on a modification of $n$-cylinders, is to be thought of as the tensor product $D_2 \otimes D_n$.
In fact, we will construct a coglobular object $\M_{\bullet}\colon \G \rightarrow \wgpd$, that will induce a functor $\mathbb{P}_2\colon \wgpd \rightarrow \left[\G^{op}, \mathbf{Set}\right]$ defined by $\p_2(X)_n=\wgpd(\M_n,X)$.
This coglobular object will also come provided with a map $\Xi=(\Xi_0,\Xi_1)\colon \cyl(D_{\bullet})\ast\cyl(D_{\bullet}) \rightarrow \M_{\bullet}$, where the domain denotes the colimit of the diagram
\[\bfig
\morphism(0,0)|a|/@{>}@<0pt>/<500,0>[D_{\bullet}` \cyl(D_{\bullet});\iota_0 ]
\morphism(0,0)|a|/@{>}@<0pt>/<0,-500>[D_{\bullet}`\cyl(D_{\bullet});\iota_1]
\morphism(500,-500)|a|/@{>}@<0pt>/<0,500>[D_{\bullet}`\cyl(D_{\bullet});\iota_1]
\morphism(500,-500)|a|/@{>}@<0pt>/<-500,0>[D_{\bullet}`\cyl(D_{\bullet});\iota_0]
\efig\]
Just like in the case of cylinders, this map will be proven to be a direct cofibration.

As a preliminary step, we need to construct some cylinders that witness specific coherences between certain whiskerings, that will be used in the definition of modifications, just as we did in \eqref{coherence cyls}.
We define these cylinders by solving the following lifting problems in $\left[\G,\wgpd\right]$:
\[
\bfig
\morphism(0,0)|a|/@{>}@<0pt>/<900,0>[\Sigma(D_{\bullet}\coprod D_{\bullet})` D_2 \plus{ D_0}\Sigma D_{\bullet};(a,b)]
\morphism(0,0)|a|/@{>}@<0pt>/<0,-400>[\Sigma(D_{\bullet}\coprod D_{\bullet})` \Sigma\cyl(D_{\bullet});\Sigma(\iota_0,\iota_1) ]
\morphism(0,-400)|r|/@{-->}@<0pt>/<900,400>[\Sigma\cyl(D_{\bullet})` D_2 \plus{ D_0}\Sigma D_{\bullet};\Gamma]

\morphism(2000,0)|a|/@{>}@<0pt>/<900,0>[\Sigma(D_{\bullet}\coprod D_{\bullet})`\Sigma D_{\bullet} \plus{ D_0} D_2;(a',b') ]
\morphism(2000,0)|a|/@{>}@<0pt>/<0,-400>[\Sigma(D_{\bullet}\coprod D_{\bullet})` \Sigma\cyl(D_{\bullet});\Sigma(\iota_0,\iota_1) ]
\morphism(2000,-400)|r|/@{-->}@<0pt>/<900,400>[\Sigma\cyl(D_{\bullet})`\Sigma D_{\bullet} \plus{ D_0} D_2;\Upsilon]
\efig
\]
where we define $a=(\sigma_1 \amalg_{ D_0}1)\circ w, \ b=(\tau_1 \amalg_{ D_0}1)\circ w, \ a'=(1\amalg_{ D_0}\sigma)\circ w, \ b'=(1\amalg_{ D_0}\tau)\circ w$.
In words, given an $\infty$-groupoid $X$, these produce $(n-1)$-cylinders $\Gamma(c,B)\colon Bs(c) \curvearrowright Bt(c)$ in $\Omega (X,s^2(c),t^n(B))$ (resp. $\Upsilon(A,d)\colon s(d)A \curvearrowright t(d)A$ in $\Omega (X,s^n(A),t^2(d))$) out of an $n$-cell $B$ and  $2$-cell $c$ (resp. an $n$-cell $A$ and a $2$-cell $d$) in $X$ which are suitably compatible .

We start with defining $M_0=D_2$ and $(\Xi)_0$ to be simply the boundary inclusion $S^1 \rightarrow D_2$.
Assuming we have defined $M_{\bullet}\colon \G_{\leq n-1 } \rightarrow \wgpd$ together with a direct cofibration of $(n-1)$-truncated coglobular objects $\Xi\colon \cyl(D_{\bullet})\ast\cyl(D_{\bullet}) \rightarrow \M_{\bullet}$, we set $\M_n$ to be the colimit of the following diagram of $\infty$-groupoids
\[
\bfig
\morphism(0,0)|a|/@{>}@<0pt>/<0,400>[\Sigma \cyl(D_{n-1})` \Sigma \M_{n-1};\Sigma(\Xi_{0}) ]
\morphism(0,0)|a|/@{>}@<0pt>/<0,-400>[\Sigma \cyl(D_{n-1})` \Sigma \left( \cyl(D_{n-1}) \otimes \cyl(D_{n-1}) \otimes \cyl(D_{n-1}) \right);\Sigma(c) ]
\morphism(800,0)|r|/@{>}@<0pt>/<-800,-400>[\Sigma \cyl(D_{n-1})` \Sigma \left( \cyl(D_{n-1}) \otimes \cyl(D_{n-1}) \otimes \cyl(D_{n-1}) \right);\Sigma(i_3) ]
\morphism(-800,0)|l|/@{>}@<0pt>/<800,-400>[\Sigma \cyl(D_{n-1})` \Sigma \left( \cyl(D_{n-1}) \otimes \cyl(D_{n-1}) \otimes \cyl(D_{n-1}) \right);\Sigma(i_1) ]
\morphism(-800,0)|a|/@{>}@<0pt>/<0,400>[\Sigma \cyl(D_{n-1})` D_n \plus{ D_0}D_2;\Upsilon_{n-1} ]
\morphism(800,0)|a|/@{>}@<0pt>/<0,400>[\Sigma \cyl(D_{n-1})` D_2 \plus{ D_0}D_n;\Gamma_{n-1} ]
\efig
\]
where $c$ denotes the vertical composition of a stack of three $(n-1)$-cylinders and $i_k$ is the inclusion on the $k$-th cylinder of the stack.
The reason we chose this ``biased'' definition, in contrast with the ``unbiased'' one we gave of cylinders is to simplify the results of the following section.

Define $\M_{\sigma_0}\colon \M_0 \rightarrow \M_1$ to be the composite $D_2 \rightarrow D_2 \amalg_{ D_0}D_1\rightarrow \M_1$, both maps being given by colimit inclusions. Analogously, we set $\M_{\tau_0}\colon \M_0 \rightarrow \M_1$ to be the composite $D_2 \rightarrow D_1 \amalg_{ D_0}D_2\rightarrow \M_1$.

Now suppose $n>2$, and define $\M_{\epsilon_{n-1}}\colon \M_{n-1} \rightarrow \M_n$ (for $\epsilon=\sigma, \tau$) as the map obtained by applying the colimit functor to the natural transformation between the defining diagrams for $\M_{n-1}$ and $\M_n$ induced by $\epsilon_{n-1}$, $\cyl(\epsilon_{n-2})$ and $\M_{\epsilon_{n-2}}$.

We define $\Xi=(\Xi_0,\Xi_1)\colon \cyl(D_{n})\ast\cyl(D_n) \rightarrow \M_n$ by setting $\Xi_0$ to be induced by the following cocone
\[
\bfig
\morphism(0,0)|a|/@{>}@<0pt>/<1000,0>[\Sigma \cyl(D_{n-1})` \Sigma \cyl(D_{n-1})^{\otimes 3};\Sigma(i_1)]
\morphism(-400,400)|l|/@{>}@<0pt>/<400,-400>[D_n` \Sigma \cyl(D_{n-1});\iota_0]
\morphism(-400,400)|a|/@{>}@<0pt>/<600,0>[D_n` D_n \plus{ D_0}D_1;w]
\morphism(200,400)|a|/@{>}@<0pt>/<800,0>[D_n \plus{ D_0}D_1` D_n \plus{ D_0}D_2;1 \coprod \tau]
\morphism(1000,0)|a|/@{>}@<0pt>/<800,0>[\Sigma \cyl(D_{n-1})^{\otimes 3}`  \M_{n};]
\morphism(1000,400)|a|/@{>}@<0pt>/<800,-400>[D_n \plus{ D_0}D_2` \M_{n};]

\morphism(-400,-400)|l|/@{>}@<0pt>/<400,400>[D_n` \Sigma \cyl(D_{n-1});\iota_1]
\morphism(-400,-400)|a|/@{>}@<0pt>/<600,0>[D_n` D_1 \plus{ D_0}D_n;w]
\morphism(200,-400)|a|/@{>}@<0pt>/<800,0>[D_1 \plus{ D_0}D_n` D_2 \plus{ D_0}D_n;\sigma \coprod 1]
\morphism(1000,-400)|a|/@{>}@<0pt>/<800,400>[D_2 \plus{ D_0}D_n`  \M_{n};]
\efig
\]
Next, we set $\Xi_1$ to be induced by the following cocone
\[
\bfig
\morphism(0,0)|a|/@{>}@<0pt>/<1000,0>[\Sigma \cyl(D_{n-1})` \Sigma \M_{n-1};\Sigma(\Xi_1)]
\morphism(-400,400)|l|/@{>}@<0pt>/<400,-400>[D_n` \Sigma \cyl(D_{n-1});\iota_0]
\morphism(-400,400)|a|/@{>}@<0pt>/<600,0>[D_n` D_n \plus{ D_0}D_1;w]
\morphism(200,400)|a|/@{>}@<0pt>/<800,0>[D_n \plus{ D_0}D_1` D_n \plus{ D_0}D_2;1 \coprod \tau]
\morphism(1000,0)|a|/@{>}@<0pt>/<800,0>[\Sigma\M_{n-1}` \M_{n};]
\morphism(1000,400)|a|/@{>}@<0pt>/<800,-400>[D_n \plus{ D_0}D_2` \M_{n};]

\morphism(-400,-400)|l|/@{>}@<0pt>/<400,400>[D_n` \Sigma \cyl(D_{n-1});\iota_1]
\morphism(-400,-400)|a|/@{>}@<0pt>/<600,0>[D_n` D_1 \plus{ D_0}D_n;w]
\morphism(200,-400)|a|/@{>}@<0pt>/<800,0>[D_1 \plus{ D_0}D_n` D_2 \plus{ D_0}D_n;\sigma \coprod 1]
\morphism(1000,-400)|a|/@{>}@<0pt>/<800,400>[D_2 \plus{ D_0}D_n` \M_{n};]
\efig
\]
In both cases the unlabeled maps denote the colimit inclusions.
\begin{defi}
	Given an $\infty$-groupoid $X$ and a map $\Theta\colon M_n \rightarrow X$ such that $C=\Theta \circ \Xi_0$ and $D=\Theta \circ \Xi_1$ we say that $\Theta$ is a modification between the $n$-cylinders $C$ and $D$. Notice that, by construction, $C\circ \iota_k=D\circ \iota_k$ for $k=0,1$. 
	
	We will also denote this by $\Theta\colon C \Rightarrow D$ or, pictorially, by
	\[\bfig
	\morphism(0,0)|a|/{@{>}@/^1.5em/}/<1000,0>[\cyl(D_n)`X;C]
	\morphism(0,0)|b|/{@{>}@/^-1.5em/}/<1000,0>[\cyl(D_n)`X;D ]
	\morphism(550,150)|r|/@{=>}@<0pt>/<0,-250>[`; \Theta]
	
	\efig
	\]
	$\Theta \circ \M_{\sigma}$ is called the source of $\Theta$, and it is denoted by $s(\M)$. Similarly, $\Theta \circ \M_{\tau}$ is called the target of $\Theta$, and it is denoted by $t(\M)$.
	
	Given two modifications $\Theta_1,\Theta_2$ such that $\epsilon(\Theta_1) =\epsilon(\Theta_2)$ for $\epsilon=\sigma, \tau$, we say that $\Theta_1$ and $\Theta_2$ are parallel.
\end{defi}

\begin{lemma}
	The map of coglobular objects $\Xi\colon \cyl(D_{\bullet})\ast\cyl(D_{\bullet})\rightarrow \M_{\bullet}$ is a direct cofibration.
	\begin{proof}
		We will prove by induction on $n$ that the $n$-th latching map $\hat{L}_n(\Xi)$ is a cofibration of $\infty$-groupoids.
		For $n=0$ this is just $(\Xi)_0$, i.e. the boundary inclusion $S^1 \rightarrow D_2$, and therefore it is a cofibration.
		
		Assume by induction that $\hat{L}_{k}(\Xi)$ is the pushout of the boundary inclusion $S^{k+1} \rightarrow D_{k+2}$ for each $0\leq q \leq n-1$ and let's prove the same holds true for $k=n$.
		We do this representably, as follows: let $X$ be an $\infty$-groupoid, and $C,D\colon \cyl(D_n)\rightarrow X$ be two $n$-cylinders in $X$, such that $C\circ\iota_k=D\circ\iota_k$ for $k=0,1$. Assume given a pair of parallel modifications $\Theta \colon s(C) \Rightarrow s(D)$, $\Psi \colon t(C) \Rightarrow t(D)$.
		To extend this to a modification $C \Rightarrow D$	we have to give a modification of $(n-1)$-cylinders $\Upsilon(\iota_0 C,\Theta_t)\otimes \bar{C}\otimes\Gamma(\Theta_s,\iota_1 C)\Rightarrow \bar{D}$ in $X(a,b)$ where $\Theta_s, \Theta_t$ are the $2$-cells that are part of the data of both $C$ and $D$, and $a=s^2(S)$, $b=t^2(T)$. Notice that we already have the source and target of this modification, so that (by inductive hypothesis), this extension amounts to filling in an $n$-sphere $X(a,b)$.
		Upon transposing along the suspension-loop space adjunction we see that the original extension problem is equivalent to extending along the boundary inclusion $S^{n+1}\rightarrow D_{n+2}$, which concludes the proof.
	\end{proof}
\end{lemma}
Thanks to this lemma, it is straightforward to prove the next result:
\begin{lemma}
	\label{modifications in contractible groupoids}
	Let $X$ be a contractible $\infty$-groupoid. Given a pair of $n$-cylinders $C,D\colon A \curvearrowright B$ in $X$, there exists a modification $\Theta\colon C \Rightarrow D$ in $X$.
\end{lemma}
We also record here, for future use, the following lemma
\begin{lemma}
	\label{extension of modifications}
	Given an $\infty$-groupoid $X$, an $n$-cylinder $C\colon A \curvearrowright B$ in $X$, a pair of parallel $(n-1)$-cylinders $D_{s},D_t\colon\cyl(D_{n-1}) \rightarrow X$ and parallel modifications $\Theta_1 \colon s(C)\Rightarrow D_s, \ \Theta_2\colon t(C)\Rightarrow D_t$ there exists an $n$-cylinder $D\colon \cyl(D_n)\rightarrow X$ such that $s(D)=D_s, \ t(D)=D_t$ and a modification $\Theta \colon C \Rightarrow D$ such that $s(\Theta)=\Theta_1$ and $t(\Theta)=\Theta_2$.
	\begin{proof}
		We prove this statement by induction, the base case being $n=1$. We can use the $2$-cells $\Theta_1\colon s(C)\rightarrow D_s$ and $\Theta_2\colon t(C)\rightarrow D_t$ and define the $2$-cell filling $D$ to be $(B\Theta_1)(C(\Theta_2^{-1}A))$. Clearly, it is possible to extend $(\Theta_1,\Theta_2^{-1})$ to a modification $\Theta\colon C \Rightarrow D$, thanks to the contractibility of $\mathfrak{C}$.
		
		Now let $n>1$ and assume the statement holds true for every integer $k < n$.
		The pair of parallel $(n-2)$-cylinders $\bar{D}_s,\bar{D}_t$ in $\Omega (X,x,y)$ (where $x=s^n(C)\sigma, \ y=t^n(C)\circ \tau$), the $(n-1)$-cylinder $ \Upsilon(\iota_0 C,\Theta_t)\otimes \bar{C}\otimes\Gamma(\Theta_s,\iota_1 C)$ in $\Omega(X,x,y)$ and the modifications $\bar{\Theta}_1, \bar{\Theta}_2$ satisfy the assumptions of the lemma for $k=n-1$. Therefore, we get an $(n-1)$-cylinder $\bar{D}$ and a modification $\bar{\Theta}\colon\Upsilon(\iota_0 C,\Theta_t)\otimes \bar{C}\otimes\Gamma(\Theta_s,\iota_1 C) \Rightarrow \bar{D}$, both in $\Omega (X,x,y)$, which concludes the proof.
	\end{proof}
\end{lemma}
The content of the previous lemma can be pictorially represented by the following extension problem
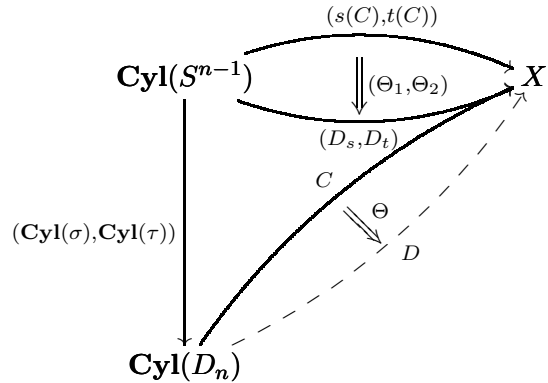
\[
\bfig
\morphism(0,0)|a|/{@{>}@/^1.5em/}/<1200,0>[\cyl(S^{n-1})`X;(s(C),t(C))]
\morphism(0,0)|b|/{@{>}@/^-1.5em/}/<1200,0>[\cyl(S^{n-1})`X;(D_s,D_t )\ \ \ \ ]
\morphism(600,125)|r|/@{=>}@<0pt>/<0,-250>[`; (\Theta_1,\Theta_2)]
\morphism(0,0)|a|/{@{>}@//}/<0,-1000>[\cyl(S^{n-1})`\cyl(D_n);(\cyl(\sigma),\cyl(\tau))]
\morphism(0,-1000)|a|/{@{>}@/^1.2em/}/<1200,1000>[\cyl(D_n)`X;C]
\morphism(0,-1000)|b|/{@{-->}@/^-1.2em/}/<1200,1000>[\cyl(D_n)`X;D]
\morphism(525,-425)|r|/@{=>}@<0pt>/<150,-150>[`; \Theta]
\efig
\]
\begin{rmk}
	\label{modif are invert}
Note that all modifications are ``invertible'' in a sense that can be made precise, but here we content ourselves with the weaker statement that given $n$-cylinders $C,D$ in an $\infty$-groupoid $X$, there exists a modification $\Theta \colon C \Rightarrow D$ if and only if there exists a modification $\Theta'\colon D \Rightarrow C$.
\end{rmk}
\section{Low dimensional operations in $\p X$}
In this section we use the tools developed so far to endow the underlying globular set of $\p X$ with all the operations  $\p X(\rho)$ for $\rho\colon D_n \rightarrow A$ in a homogeneous coherator for $\infty$-categories $\mathfrak{D}$, with $n\leq 2$. This is enough, for instance, to endow the $2$-globular set obtained from $\p X$ by identifying $2$-cells connected by a $3$-cell and keeping the same $0$ and $1$-cells with the structure of a bicategory with weak inverses (i.e. an unbiased bigroupoid, or a weak 2-groupoid à la Batanin, see \cite{BAT}), thanks to the results of Section 6.

Thanks to Theorem \ref{algebraic structure on PX}, we already have inverses, so we only need to interpret all the homogeneous operations of dimension $n\leq 2$ since we have already defined $\cyl(\bullet)$ on globular maps. More precisely, we have to define $\cyl(\rho)\colon \cyl(D_n) \rightarrow \cyl(A)$ for every homogeneous map $\rho \colon D_n \rightarrow A$ with $n\leq 2$ in $\mathfrak{D}$. Notice that this forces $m=\dim(A)\leq 2$.
\begin{rmk}
	\label{tower for D}
Thanks to Proposition \ref{fact of maps of gpds} and a similar argument as the one used in the proof of the lemma below, it is not hard to show that one can assume, without loss of generality, that the coherator $\mathfrak{D}$ has been obtained in the following manner: there is a functor $\mathfrak{D}_{\bullet}\colon \omega \rightarrow \mathbf{GlTh}$ as in Definition \ref{cell glob th}, with $\mathfrak{D}_{n+1}= \mathfrak{D}_{n}[X]$, where $X=\{(h_1,h_2)\colon D_n \rightarrow A\}$, i.e. all the $(n+1)$-dimensional ``basic'' operations of $\mathfrak{D}$ are added at the $(n+1)$st step.
Therefore, we may rephrase the goal of this section in terms of constructing an extension of the form:
\[
\bfig
\morphism(0,0)|a|/@{>}@<0pt>/<500,0>[\Theta_0`\wgpd;\cyl]
\morphism(0,0)|a|/@{>}@<0pt>/<0,-400>[\Theta_0` \mathfrak{D}_{2};]
\morphism(0,-400)|r|/@{-->}@<0pt>/<500,400>[\mathfrak{D}_{2}`\wgpd;\cyl]
\efig
\]
\end{rmk}
We will make use of the following fact, whose proof we only sketch not to disrupt the flow of this section
\begin{lemma}
Given a cellular globular theory $\mathfrak{D}$, the inclusion $\Theta_0 \rightarrow \mathfrak{D}$ induces isomorphisms 
\[\Theta_0(D_0,A)\cong \mathfrak{D}(D_0,A)\]
\begin{proof}
It is enough to prove that the unit map $\eta_A\colon A \rightarrow U\circ F A$ of the adjunction 
\[\xymatrixcolsep{1pc}
\vcenter{\hbox{\xymatrix{
			**[l][\G^{op},\mathbf{Set}] \xtwocell[r]{}_{U}^{F}{'\perp}& **[r]\wgpd
}}}
\] is sent to an isomorphism when we evaluate $[\G^{op},\mathbf{Set}](D_0,-)$ at it. Thanks to Proposition 2.2 of \cite{Nik}, the unit is an $I$-cellular map, and therefore it is $0$-bijective.
\end{proof}
\end{lemma}
To begin with, we start with operations of dimension $1$, i.e. extending the functor $\cyl$ to $\mathfrak{D}_1$. More precisely, we mean operations $h$ added as solutions of lifting problems of the following form, as in point (2) of Definition \ref{cell glob th}
\[
\bfig 
\morphism(0,0)|a|/@{>}@<3pt>/<500,0>[D_0`A;f]
\morphism(0,0)|b|/@{>}@<-3pt>/<500,0>[D_0`A;g]

\morphism(0,0)|r|/@{>}@<3pt>/<0,-400>[D_0`D_{1};\tau_0]
\morphism(0,0)|l|/@{>}@<-3pt>/<0,-400>[D_0`D_{1};\sigma_0]

\morphism(0,-400)|r|/@{>}@<0pt>/<500,400>[D_{1}`A;h]
\efig 
\]
We know that, since $\mathfrak{D}$ is assumed to be homogeneous, this implies $\dim(A)\leq 1$, and therefore either $f=g=1_{D_0}$, or $\dim(A)=1$ and necessarily $f=\partial_{\sigma},g=\partial_{\tau}$ thanks to the previous lemma. Therefore, setting $\cyl(h)=\hat{h}$ as in Definition \ref{rho hat defi} is a well-defined choice. Doing so for all the $1$-dimensional operations added at the first stage of the defining tower of globular theories that witnesses the cellularity of $\mathfrak{D}$, we get an extension of the form:
\begin{equation}
\label{ext to D1}	
\bfig
\morphism(0,0)|a|/@{>}@<0pt>/<500,0>[\Theta_0`\wgpd;\cyl]
\morphism(0,0)|a|/@{>}@<0pt>/<0,-400>[\Theta_0` \mathfrak{D}_{1};]
\morphism(0,-400)|r|/@{-->}@<0pt>/<500,400>[\mathfrak{D}_{1}`\wgpd;\cyl]
\efig
\end{equation}
Let us now address the problem of extending this to $2$-dimensional operations, i.e. morphisms $D_2 \rightarrow A$ in $\mathfrak{D}$. This can be done using the following result, whose proof is fundamental to this section and will be subdivided into several lemmas.
\begin{lemma}
	\label{canonical modif of 1 dim operations}
Given a map $\rho \colon D_1 \rightarrow A$ in $\mathfrak{D}$, there exists a modification \[\theta_{\rho}\colon \hat{\rho} \Rightarrow \cyl(\rho)\]
where $\cyl(\rho)$ exists thanks to \eqref{ext to D1}.
Furthermore, these modifications can be built in such a way that if $\rho_1,\rho_2\colon D_1 \rightarrow A$ are parallel maps then $\theta_{\rho_1}$ and $\theta_{\rho_2}$ are parallel.
\end{lemma}
Notice that, thanks to the extension of the definition of $\hat{\rho}$ to non-necessarily homogeneous maps $\rho$ in $\mathfrak{D}$ given in Definition \ref{ext of hat defi}, it is enough to define $\theta_{\rho}$ only on homogeneous operations $\rho$, thus forcing $\dim(A)\leq 1$.
If we assume the previous result, we can prove the next one.
\begin{prop}
	Given any operation $\rho\colon D_2 \rightarrow A$ in $\mathfrak{D}$ fitting into a diagram of the form
	\[ \bfig
	\morphism(0,0)|a|/@{>}@<2pt>/<750,0>[D_1`A;h_1]
	\morphism(0,0)|b|/@{>}@<-2pt>/<750,0>[D_1`A;h_2]
	\morphism(0,0)|r|/@{>}@<2pt>/<0,-500>[D_1`D_{2};\tau]
	\morphism(0,0)|l|/@{>}@<-2pt>/<0,-500>[D_1`D_{2};\sigma]
	\morphism(0,-500)|r|/@{>}@<0pt>/<750,500>[D_{2}`A;\rho]
	\efig
	\] 
	 we can associate to it a map $\cyl(\rho)\colon \cyl(D_2) \rightarrow \cyl(A)$ fitting into a diagram of the form
	 	\[ \bfig
	 \morphism(0,0)|a|/@{>}@<2pt>/<750,0>[\cyl(D_1)`\cyl(A);\cyl(h_1)]
	 \morphism(0,0)|b|/@{>}@<-2pt>/<750,0>[\cyl(D_1)`\cyl(A);\cyl(h_2)]
	 \morphism(0,0)|r|/@{>}@<2pt>/<0,-500>[\cyl(D_1)`\cyl(D_{2});\cyl(\tau)]
	 \morphism(0,0)|l|/@{>}@<-2pt>/<0,-500>[\cyl(D_1)`\cyl(D_{2});\cyl(\sigma)]
	 \morphism(0,-500)|r|/@{>}@<0pt>/<750,500>[\cyl(D_{2})`\cyl(A);\cyl(\rho)]
	 \efig
	 \]
	  Moreover, this map also comes endowed with a modification $\theta_{\rho}\colon \hat{\rho}\Rightarrow \cyl(\rho)$ whose boundary is given by $(\theta_{s(\rho)},\theta_{t(\rho)})$.
\begin{proof}
Using Lemma \ref{canonical modif of 1 dim operations}, we can apply Lemma \ref{extension of modifications} to the following diagram, where $\rho\colon D_2 \rightarrow A$ is assumed to be added as the solution of a lifting problem associated with an admissible pair of maps, as in the definition of $\mathfrak{D}$, and the solid triangle at the back commutes by \eqref{hat properties}.
\[
\bfig
\morphism(0,0)|a|/{@{>}@/^1.5em/}/<1700,0>[\cyl(S^{1})`\cyl(A);\left( \widehat{\rho\circ \sigma},\widehat{\rho\circ \tau}\right)]
\morphism(0,0)|b|/{@{>}@/^-1.5em/}/<1700,0>[\cyl(S^{1})`\cyl(A);\left(\cyl(\rho\circ \sigma),\cyl(\rho\circ \tau )\right)\ \ \ \ ]
\morphism(850,125)|r|/@{=>}@<0pt>/<0,-250>[`; (\theta_{\rho\circ \sigma},\theta_{\rho\circ \tau})]
\morphism(0,0)|a|/{@{>}@//}/<0,-1300>[\cyl(S^{1})`\cyl(D_2);(\cyl(\sigma),\cyl(\tau))]
\morphism(0,-1300)|a|/{@{>}@/^1.2em/}/<1700,1300>[\cyl(D_2)`\cyl(A);\hat{\rho}]
\morphism(0,-1300)|b|/{@{-->}@/^-1.2em/}/<1700,1300>[\cyl(D_2)`\cyl(A);\cyl(\rho)]
\morphism(800,-550)|r|/@{=>}@<0pt>/<125,-125>[`; \theta_{\rho}]
\efig
\]
It is straightforward to observe that this is enough to conclude the proof.
\end{proof}
\end{prop}
The following corollary follows immediately from the previous result and cellularity of $\mathfrak{D}$
\begin{cor}
There exists an extension of the form:
\[\bfig
\morphism(0,0)|a|/@{>}@<0pt>/<500,0>[\Theta_0`\wgpd;\cyl]
\morphism(0,0)|a|/@{>}@<0pt>/<0,-400>[\Theta_0` \mathfrak{D}_{2};]
\morphism(0,-400)|r|/@{-->}@<0pt>/<500,400>[\mathfrak{D}_{2}`\wgpd;\cyl]
\efig\]
\end{cor}
\begin{defi}
Given integers $i,k,q>0$, a map $\rho\colon D_1 \rightarrow D_1^{\otimes k}$ such that $\rho \circ \sigma=\partial_{\sigma},\rho\circ \tau=\partial_{\tau}$ and a cylinder $C\colon \cyl(D_1)\rightarrow \cyl(D_1^{\otimes q})$, we denote by $\hat{\rho}\circ_i C$ the following composite cylinder
\[
\bfig
\morphism(-800,0)|a|/@{>}@<0pt>/<800,0>[\cyl(D_{1})`\cyl(D_{1}^{\otimes k});\hat{\rho}]
\morphism(0,0)|a|/@{>}@<0pt>/<1600,0>[\cyl(D_{1}^{\otimes k})`\cyl(D_{1}^{\otimes k+q-1});1\coprod\ldots\coprod C \coprod \ldots \coprod 1]
\efig
\]
\end{defi}
The following lemmas altogether produce a proof of Lemma \ref{canonical modif of 1 dim operations}.
The first one simply follows from the universal property of colimits.
\begin{lemma}
	Given a map $\rho\colon D_1 \rightarrow D_1^{\otimes k}$ such that $\rho \circ \sigma=\partial_{\sigma},\rho\circ \tau=\partial_{\tau}$, integers $p,q>0$ and $1\leq i<j\leq k$ together with cylinders $C\colon\cyl(D_1)\rightarrow \cyl(D_1^{\otimes p})$ and $D\colon\cyl(D_1)\rightarrow \cyl(D_1^{\otimes q})$, such that the following diagram makes sense
	\[
	\bfig
	\morphism(0,0)|a|/{@{>}@/^0em/}/<1800,0>[\cyl(D_{1}^{\otimes k})`\cyl\left( D_{1}^{\otimes k+p+q-2}\right);1\coprod\ldots\coprod C\coprod\ldots\coprod D\coprod \ldots \coprod 1]
	\morphism(-800,0)|a|/@{>}@<0pt>/<800,0>[\cyl(D_{1})`\cyl(D_{1}^{\otimes k});\hat{\rho}]
	\efig 
	\]
	We then have an equality $(\hat{\rho}\circ_i C) \circ_{j+p-1} D = (\hat{\rho}\circ_j D)\circ_i C$.
\end{lemma}
\begin{lemma}
Given integers $i,k,q>0$, a map $\rho\colon D_1 \rightarrow D_1^{\otimes k}$ such that $\rho \circ \sigma=\partial_{\sigma},\rho\circ \tau=\partial_{\tau}$, and a modification $\Theta\colon 	C \Rightarrow D\colon \cyl(D_1)\rightarrow \cyl(D_1^{\otimes q})$ such that $\Theta_s,\Theta_t$ are identities (thus we think of $\Theta$ as a $3$-cell as explained in Example \ref{unraveling modifications}) we get an induced modification denoted by $\hat{\rho}\circ_i \Theta \colon \hat{\rho}\circ_i C \Rightarrow \hat{\rho}\circ_i D$ between the following $1$-cylinders
\[
\bfig 
\morphism(0,0)|a|/{@{>}@/^1.5em/}/<1700,0>[\cyl(D_{1}^{\otimes k})`\cyl(D_{1}^{\otimes k+q-1});1\coprod\ldots\coprod C \coprod \ldots \coprod 1]
\morphism(0,0)|b|/{@{>}@/^-1.5em/}/<1700,0>[\cyl(D_{1}^{\otimes k})`\cyl(D_{1}^{\otimes k+q-1});1\coprod\ldots\coprod D \coprod \ldots \coprod 1]
\morphism(-800,0)|a|/@{>}@<0pt>/<800,0>[\cyl(D_{1})`\cyl(D_{1}^{\otimes k});\hat{\rho}]
\efig 
\]
whenever this diagram makes sense, where both $C$ and $D$ appear at the $i$-th coordinate. Moreover, this modification is essentially given by a 3-cell.
\begin{proof}
We prove this representably, i.e. we assume given compatible $1$-cylinders $C_i\colon A_i \curvearrowright B_i$ in an $\infty$-groupoid $X$, with $s(C_i)=w_i$ and $t(C_i)=w_{i+1}$. We let $E=(C_i,\ldots ,C_{i+q-1})\circ C$ and $F=(C_i,\ldots ,C_{i+q-1})\circ D$, and we observe that $\Theta$ essentially consists of a 2-cell $\overline{E}\rightarrow\overline{F}\colon w_{i+q}E_0 \rightarrow E_1 w_i$ in $\Omega \left( X,s(E_0),t(E_1) \right)$. Moreover, we denote the effect of composing using $\rho$ with juxtaposition of $k$ $1$-cells. The composite of the following pasting diagram in $\Omega \left( X, s(A_1),t(B_{k+q-1})\right)$ transpose via the adjunction $\Sigma \dashv \Omega$ to give the 3-cell in $X$ which $\hat{\rho}\circ_i \Theta$ essentially corresponds to:
\[
\bfig 
\morphism(0,0)|a|/@{>}@<0pt>/<0,-400>[w_{k+q}(A_{k+q-1}\ldots A_{i+q} E_0 A_{i-1}\ldots A_1)`\ldots;]
\morphism(0,-400)|a|/@{>}@<0pt>/<0,-400>[\ldots`B_{k+q-1}\ldots B_{i+q}(w_{i+q}E_0)A_{i-1}\ldots A_1;]
\morphism(0,-800)|r|/{@{>}@/^1.5em/}/<0,-600>[B_{k+q-1}\ldots B_{i+q}(w_{i+q}E_0)A_{i-1}\ldots A_1`B_{k+q-1}\ldots B_{i+q}(E_1w_i)A_{i-1}\ldots A_1;B_{k+q-1}\ldots B_{i+q}\overline{F} A_{i-1}\ldots A_1]
\morphism(0,-800)|l|/{@{>}@/^-1.5em/}/<0,-600>[B_{k+q-1}\ldots B_{i+q}(w_{i+q}E_0)A_{i-1}\ldots A_1`B_{k+q-1}\ldots B_{i+q}(E_1w_i)A_{i-1}\ldots A_1;B_{k+q-1}\ldots B_{i+q}\overline{E} A_{i-1}\ldots A_1]
\morphism(-100,-1100)|a|/@{=>}@<0pt>/<200,0>[`;\ldots \overline{\Theta}\ldots ]
\morphism(0,-1400)|a|/@{>}@<0pt>/<0,-400>[B_{k+q-1}\ldots B_{i+q}(E_1 w_i)A_{i-1}\ldots A_1`\ldots;]
\morphism(0,-1800)|a|/@{>}@<0pt>/<0,-400>[\ldots`(B_{k+q-1}\ldots B_{i+q} E_1 B_{i-1}\ldots B_1)w_1;]
\efig 
\] where we have implicitly used the fact that $E_i=F_i$ for $i=0,1$. Notice that, by definition, the left-hand side composite is $(C_1,\ldots C_{i-1},E,C_{i+q},\ldots,C_{k+q-1})\circ \hat{\rho}$, and the right-hand side one is $(C_1,\ldots C_{i-1},F,C_{i+q},\ldots,C_{k+q-1})\circ \hat{\rho}$, which concludes the proof.
\end{proof}
\end{lemma}
We will need something a bit stronger, namely the following generalization of the previous lemma, whose proof is left to the reader.
\begin{lemma}
	\label{functoriality of modif}
Assume given integers $k,q_j>0$ for $1\leq j\leq k$, a map $\rho\colon D_1 \rightarrow D_1^{\otimes k}$ such that $\rho \circ \sigma=\partial_{\sigma},\rho\circ \tau=\partial_{\tau}$, and modifications $\Theta_j\colon 	C_j \Rightarrow D_j\colon \cyl(D_1)\rightarrow \cyl(D_1^{q_j})$ such that ${\Theta_j}_s,{\Theta_j}_t$ are identities for every $j$. We then get an induced modification, denoted by $ (\Theta_1,\ldots,\Theta_k)\circ \hat{\rho}= (C_1,\ldots,C_k)\circ \hat{\rho}\Rightarrow(D_1,\ldots,D_k)\circ \hat{\rho}$, between the following $1$-cylinders
\[
\bfig 
\morphism(0,0)|a|/{@{>}@/^1.5em/}/<1700,0>[\cyl(D_{1}^{\otimes k})`\cyl\left( D_{1}^{\otimes (\sum_j q_j)}\right);\plus{1\leq j \leq k}C_j]
\morphism(0,0)|b|/{@{>}@/^-1.5em/}/<1700,0>[\cyl(D_{1}^{\otimes k})`\cyl\left( D_{1}^{\otimes (\sum_j q_j) }\right) ;\plus{1\leq j \leq k}D_j]
\morphism(-800,0)|a|/@{>}@<0pt>/<800,0>[\cyl(D_{1})`\cyl(D_{1}^{\otimes k});\hat{\rho}]

\efig 
\] which, again, is essentially given by a $3$-cell.
\end{lemma}
\begin{lemma}
Given compatible operations $\rho \colon D_1 \rightarrow D_1^{\otimes k}$, $\phi_j\colon D_i \rightarrow D_1^{\otimes q_j}$ for $1\leq j \leq k$ similarly to the previous lemma, there is an induced modification
\begin{equation}
\label{coherence mod}
\bfig 
\morphism(0,400)|a|/{@{>}@/^0em/}/<1700,-400>[\cyl(D_{1}^{\otimes k})`\cyl\left( D_{1}^{\otimes (\sum_j q_j) }\right);\plus{1\leq j \leq k}\widehat{\phi_j}]
\morphism(-800,0)|a|/@{>}@<0pt>/<800,400>[\cyl(D_{1})`\cyl(D_{1}^{\otimes k});\hat{\rho}]
\morphism(-800,0)|b|/@{>}@<0pt>/<2500,0>[\cyl(D_{1})`\cyl\left( D_{1}^{\otimes (\sum_j q_j) }\right);{ (\left(\phi_j\right)_{1\leq j \leq k}\circ \rho)}^{\wedge}]
\morphism(150,250)|a|/{@{=>}@/^0em/}/<150,-150>[`;]
\efig 
\end{equation}
\begin{proof}
We prove this representably, thus we assume given an $\infty$-groupoid X and a family of compatible $1$-cylinders $C^m_r$ with $1\leq m\leq k$ and $1\leq r \leq q_m$. Denote with $\underline{C}^m$ the map $(C^m_1,\ldots,C^m_{q_m})\colon \cyl(D_1^{\otimes q_m})\rightarrow X$, and, for $\epsilon=0,1$, let $\underline{C_\epsilon}^m$ be the string of $1$-cells in $X$ given by  $({C^m_1}_{\epsilon},\ldots,{C^m_{q_m}}_{\epsilon})\colon D_1^{\otimes q_m}\rightarrow X$. In this way, we have that $\widehat{\phi_m}$ acts on $\underline{C}^m$, and $\phi_m$ acts on $\underline{C_\epsilon}^m$. Define $n=\sum_j q_j$ and denote the target of the $i$-th cylinder by $w_i$. Consider the following diagram in $\Omega(X,a,b)$, where $a=s\left(({C^1_1})_0\right)$ and $b=t\left(({C^m_{q_m}})_1\right)$, and juxtaposition represents the result of composing the 1-cells involved using $\rho$. The left-hand side (resp. right-hand side) composite coincides with the upper composite of \eqref{coherence mod} (resp. bottom map).
\[
\bfig 
\morphism(0,0)|a|/@{>}@<0pt>/<-1200,-400>[w_n\left(\phi_k(\underline{C_0}^k)\ldots \phi_1(\underline{C_0}^1)\right)`\left( w_n\phi_k(\underline{C_0}^k)\right)\phi_{k-1}(\underline{C_0}^{k-1})\ldots \phi_1(\underline{C_0}^1);]
\morphism(-1200,-400)|l|/@{>}@<0pt>/<0,-800>[\left( w_n\phi_k(\underline{C_0}^k)\right)\phi_{k-1}(\underline{C_0}^{k-1})\ldots \phi_1(\underline{C_0}^1)`\left(\phi_k(\underline{C_1}^k) w_{n-q_k}\right)\phi_{k-1}(\underline{C_0}^{k-1})\ldots \phi_1(\underline{C_0}^1);g]
\morphism(-1200,-1200)|a|/@{>}@<0pt>/<0,-400>[\left(\phi_k(\underline{C_1}^k) w_{n-q_k}\right)\phi_{k-1}(\underline{C_0}^{k-1})\ldots \phi_1(\underline{C_0}^1)`\phi_k(\underline{C_1}^k) \left(w_{n-q_k}\phi_k(\underline{C_0}^{k-1}\right)\ldots \phi_1(\underline{C_0}^1);]	
\morphism(-1200,-1600)|a|/@{>}@<0pt>/<0,-400>[\phi_k(\underline{C_1}^k) \left(w_{n-q_k}\phi_k(\underline{C_0}^{k-1}\right)\ldots \phi_1(\underline{C_0}^1)`\ldots;]	
\morphism(-1200,-2000)|a|/@{>}@<0pt>/<0,-400>[\ldots`\phi_k(\underline{C_1}^k)\ldots \left(\phi_1(\underline{C_1}^1)w_1\right);]
\morphism(-1200,-2400)|a|/@{>}@<0pt>/<1200,-400>[\phi_k(\underline{C_1}^k)\ldots \left(\phi_1(\underline{C_1}^1)w_1\right)`\left(\phi_k(\underline{C_1}^k)\ldots \phi_1(\underline{C_1}^1)\right)w_1;]

\morphism(0,0)|a|/@{>}@<0pt>/<1200,-400>[w_n\left(\phi_k(\underline{C_0}^k)\ldots \phi_1(\underline{C_0}^1)\right)`\phi_k \left( (w_n (C^k_{q_k})_0) \ldots (C^k_1)_0 \right) \phi_{k-1}(\underline{C_0}^{k-1})\ldots \phi_1(\underline{C_0}^1);]
\morphism(1200,-400)|a|/@{>}@<0pt>/<0,-400>[\phi_k \left( (w_n (C^k_{q_k})_0) \ldots (C^k_1)_0 \right) \phi_{k-1}(\underline{C_0}^{k-1})\ldots \phi_1(\underline{C_0}^1)`\ldots;]
\morphism(1200,-800)|a|/@{>}@<0pt>/<0,-400>[\ldots`\phi_k \left(  (C^k_{q_k})_1 \ldots (C^k_1)_1 w_{n-q_k} \right) \phi_{k-1}(\underline{C_0}^{k-1})\ldots \phi_1(\underline{C_0}^1);]
\morphism(1500,-800)|a|/@{ }@<0pt>/<0,0>[ ` ;(2)]
\morphism(-1200,-1200)|a|/@{>}@<0pt>/<2400,0>[\left(\phi_k(\underline{C_1}^k) w_{n-q_k}\right)\phi_{k-1}(\underline{C_0}^{k-1})\ldots \phi_1(\underline{C_0}^1)`\phi_k \left(  (C^k_{q_k})_1 \ldots (C^k_1)_1 w_{n-q_k} \right) \phi_{k-1}(\underline{C_0}^{k-1})\ldots \phi_1(\underline{C_0}^1);]	
\morphism(0,-600)|a|/@{}@<0pt>/<0,0>[`;(1)]
\morphism(1200,-1200)|a|/@{>}@<0pt>/<0,-400>[\phi_k \left(  (C^k_{q_k})_1 \ldots (C^k_1)_1 w_{n-q_k} \right) \phi_{k-1}(\underline{C_0}^{k-1})\ldots \phi_1(\underline{C_0}^1)`\ldots;]	
\morphism(1200,-1600)|a|/@{>}@<0pt>/<-1200,-1200>[\ldots`\left(\phi_k(\underline{C_1}^k)\ldots \phi_1(\underline{C_1}^1)\right)w_1;]
\efig 
\] 
We now explain how to fill the part of the diagram labelled with (1) with a 2-cell, and the same argument will provide fillers for the other analogous subdivisions of the diagram, corresponding to the $\phi_j$ for $j< k$. Note $g$ is a whiskering (via an operation $\overline{\rho}\colon D_2 \rightarrow D_2 \plus{D_0}D_1^{\otimes k-1}$ with $\overline{\rho} \circ\epsilon=\partial_{\epsilon} \circ \rho$ for $\epsilon=\sigma,\tau$) of a composite of the form
\[
\bfig 
\morphism(0,0)|a|/@{>}@<0pt>/<600,0>[w_n\phi_k(\underline{C_0}^k)`\cdots;a_1]
\morphism(600,0)|a|/@{>}@<0pt>/<600,0>[\cdots`\phi_k(\underline{C_1}^k) w_{n-q_k};a_{p}]
\efig 
\] with the $1$-cell $\phi_{k-1}(\underline{C_0}^{k-1})\ldots \phi_1(\underline{C_0}^1)$, where $p=\vert \mathcal{L}(D_1^{\otimes q_k}) \vert$. Here, we denoted the set of $1$-cells which constitutes the vertical stack of 0-cylinders appearing in the construction of $\widehat{\phi_k}$ with $\{a_i\}_{1\leq i \leq \vert \mathcal{L}(D_1^{\otimes q_k})\vert }$.

By construction of $\hat{\rho}$ and contractibility of $\mathfrak{C}$, we see that there is a 2-cell $\alpha\colon g \rightarrow g'$, where $g'$ is the composite of 2-cells of the form $a_i \phi_{k-1}(\underline{C_0}^{k-1})\ldots \phi_1(\underline{C_0}^1)$, again obtained by using the operation $\overline{\rho}$. The target of $a_1 \phi_{k-1}(\underline{C_0}^{k-1})\ldots \phi_1(\underline{C_0}^1)$ is precisely given by \[ \phi_k \left( w_n (C^k_{q_k})_0 \ldots (C^k_1)_0 \right) \phi_{k-1}(\underline{C_0}^{k-1})\ldots \phi_1(\underline{C_0}^1)\] and each of the 2-cells $a_i \phi_{k-1}(\underline{C_0}^{k-1})\ldots \phi_1(\underline{C_0}^1)$ for $1<i < \vert \mathcal{L}(D_1^{\otimes q_k})\vert $ is parallel to (and appears in the same order as) one on the right-hand side composite labelled with (2). Each of this pair of parallel 2-cells factors, by construction, through the same globular sum. They are (possibly different) whiskerings of the same cells, and therefore there is a 3-cell between them. Using again contractibility of $\mathfrak{C}$ for the triangle of the form
\[
\bfig 
\morphism(0,0)|a|/@{>}@<0pt>/<0,-400>[w_n\left(\phi_k(\underline{C_0}^k)\ldots \phi_1(\underline{C_0}^1)\right)`\left( w_n\phi_k(\underline{C_0}^k)\right)\phi_{k-1}(\underline{C_0}^{k-1})\ldots \phi_1(\underline{C_0}^1);]
\morphism(0,0)|a|/@{>}@<0pt>/<2400,0>[w_n\left(\phi_k(\underline{C_0}^k)\ldots \phi_1(\underline{C_0}^1)\right)`\phi_k \left( w_n (C^k_{q_k})_0 \ldots (C^k_1)_0 \right) \phi_{k-1}(\underline{C_0}^{k-1})\ldots \phi_1(\underline{C_0}^1);]
\morphism(0,-400)|r|/@{>}@<0pt>/<2400,400>[\left( w_n\phi_k(\underline{C_0}^k)\right)\phi_{k-1}(\underline{C_0}^{k-1})\ldots \phi_1(\underline{C_0}^1)` \phi_k \left( w_n (C^k_{q_k})_0 \ldots (C^k_1)_0 \right) \phi_{k-1}(\underline{C_0}^{k-1})\ldots \phi_1(\underline{C_0}^1); \ \ \ a_1\phi_{k-1}(\underline{C_0}^{k-1})\ldots \phi_1(\underline{C_0}^1)]
\efig 
\] and for the analogous one at the bottom we get the remaining 3-cell fillers needed to provide the desired filler for (1). 
\end{proof}
\end{lemma}
Finally, the last intermediate result before the proof of Lemma \ref{canonical modif of 1 dim operations}.
\begin{lemma}
In the same situation as the previous lemma, assume given compatible modifications $\Theta_j\colon \cyl(\phi_j)\rightarrow \hat{\phi_j}$ with trivial boundary for every $1\leq j\leq k$. Then we get an induced modification $\plus{1\leq j \leq k}\cyl(\phi_j)\circ \cyl(\rho)=\cyl\left( (\phi_1,\ldots,\phi_k)\circ \rho \right) \Rightarrow \left( (\phi_1,\ldots,\phi_k)\circ \rho  \right)^{\wedge}$ that can be depicted as follows
\[
\bfig 
\morphism(0,400)|a|/{@{>}@/^0em/}/<1700,-400>[\cyl(D_{1}^{\otimes k})`\cyl\left( D_{1}^{\otimes (\sum_j q_j) }\right);\plus{1\leq j \leq k}\cyl(\phi_j)]
\morphism(-800,0)|a|/@{>}@<0pt>/<800,400>[\cyl(D_{1})`\cyl(D_{1}^{\otimes k});\cyl(\rho)]
\morphism(-800,0)|b|/@{>}@<0pt>/<2500,0>[\cyl(D_{1})`\cyl\left( D_{1}^{\otimes (\sum_j q_j) }\right);\left( (\phi_j)_{1\leq j\leq k}\circ\rho \right)^{\wedge}]
\morphism(150,250)|a|/{@{=>}@/^0em/}/<150,-150>[`;]
\efig 
\]
\begin{proof}
By Lemma \ref{functoriality of modif} we get a modification $\left( \cyl(\phi_1), \ldots, \cyl(\phi_k)\right) \circ \hat{\rho} \Rightarrow  \left( \hat{\phi}_1, \ldots, \hat{\phi}_k\right) \circ \hat{\rho}$, and thanks to the previous lemma we get one of the form $\left( \hat{\phi}_1, \ldots, \hat{\phi}_k\right) \circ \hat{\rho}\Rightarrow \left( (\phi_j)_{1\leq j\leq k}\circ\rho \right)^{\wedge}$. Moreover, we can precompose the former with \[\plus{1\leq j \leq k}\cyl(\phi_j) \circ \theta_{\rho}\colon \plus{1\leq j \leq k}\cyl(\phi_j)\circ \cyl(\rho) \Rightarrow \plus{1\leq j \leq k}\cyl(\phi_j) \circ \hat{\rho}\]
All of these modifications are essentially $3$-cells, so that we can compose them and conclude the proof.
\end{proof}
\end{lemma}
We now end this section with a proof of Lemma \ref{canonical modif of 1 dim operations}.  
\begin{proof}
	Without loss of generality we can assume $\rho$ is homogeneous, since the statement clearly holds for globular maps, and we can factorize $\rho$ into a homogeneous map followed by a globular one. In particular, $\dim(A)\leq 1$.
Let us define a category $\D'$ whose objects are globular sums of dimension 0 and 1, for which we set
\[\D'(A,B)=\{\rho\colon A \rightarrow B \in \mathfrak{D}_{\textbf{hom}}\vert \  \exists \widehat{\rho} \Rightarrow \cyl(\rho) \}\]
where $\mathfrak{D}_{\textbf{hom}}$ denotes the subcategory of homogeneous maps of $\mathfrak{D}$. Thanks to Definition \ref{ext of hat defi} and the previous lemma, this indeed defines a subcategory of $\mathfrak{D}_{\textbf{hom}}$, closed under 1-dimensional globular sums (note that we make implicit use of what noted in Remark \ref{modif are invert}).
It is clear that cellularity of $\mathfrak{D}$, together with Proposition \ref{fact of maps of gpds}, implies that we can construct a map $F\colon\mathfrak{D}_{\leq 1} \rightarrow \D'$, where $\mathfrak{D}_{\leq 1}$ is the full subcategory of $\mathfrak{D}_{\textbf{hom}}$ spanned by globular sums of dimension less than or equal to 1. In fact, it is enough to define it on the ``basic'' operations that we add in going from $\Theta_0$ to $\mathfrak{D}_1$ in the defining tower of the coherator $\mathfrak{D}$ (as explained in \ref{tower for D}), and we have already done so in \eqref{ext to D1}. We can now conclude by observing that the functor $F$ encodes the statement of the Lemma, together with the preservation of parallelism, which easily follows from the previous constructions.
\end{proof}

\end{document}